\documentclass[a4paper,reqno,
12pt]{amsart}
\usepackage{
amssymb,
amsmath,
amsthm,
eucal,
dsfont,
multicol,
mathrsfs,
tikz,
graphicx,
hyperref
}

\makeatletter\renewcommand*{\eqref}[1]{%
\hyperref[{#1}]{\textup{\tagform@{\!\!\ref*{#1}}}}%
}\makeatother 

\makeatletter
 
  \@addtoreset{equation}{section}
 \makeatother
\theoremstyle{plain}

\newtheorem{theorem}{Theorem}[section]
\newtheorem{lemma}[theorem]{Lemma}

\newtheorem{corollary}[theorem]{Corollary}
\theoremstyle{definition}

\newtheorem{remark}[theorem]{Remark}
\newtheorem{example}[theorem]{Example}
\newtheorem{assumption}{Assumption}

\newcommand{\bignorm}[1]{{\left\|#1\right\|}}
\newcommand{\norm}[1]{{\|#1\|}}

\def\supp{\mathop{\mathrm{supp}}\nolimits}

\def\Id{\mathop{\mathrm{Id}}\nolimits}

\def\Re{\mathop{\mathrm{Re}}\nolimits}
\def\Im{\mathop{\mathrm{Im}}\nolimits}
\def\loc{\mathop{\mathrm{loc}}\nolimits}

\def\R{{\mathbb{R}}}
\def\Z{{\mathbb{Z}}}
\def\N{{\mathbb{N}}}
\def\C{{\mathbb{C}}}

\def\S{{\mathcal{S}}}
\def\F{{\mathcal{F}}}
\def\H{{\mathcal{H}}}

\def\X{{\mathcal{X}}}
\def\Y{{\mathcal{Y}}}

\def\<{{\langle}}
\def\>{{\rangle}}

\def\ep{{\varepsilon}}
\def\ds{\displaystyle}

\title[Fractional operators with sharp Hardy potentials]{Kato smoothing, Strichartz and uniform Sobolev estimates for fractional operators with sharp Hardy potentials}

\author{Haruya Mizutani}
\address[H. Mizutani]{Department of Mathematics, Graduate School of Science, Osaka University, Toyonaka, Osaka 560-0043, Japan}
\email{haruya@math.sci.osaka-u.ac.jp}
\author{Xiaohua Yao}
\address[X. Yao]{Department of Mathematics and Hubei Province Key Laboratory of Mathematical Physics, Central China Normal University, Wuhan, 430079, P.R. China}
\email{yaoxiaohua@mail.ccnu.edu.cn}

\begin{document}

\begin{abstract}
Let $0<\sigma<n/2$ and $H=(-\Delta)^\sigma +V(x)$ be Schr\"odinger type operators on $\R^n$ with a class of scaling-critical potentials $V(x)$, which include the Hardy potential $a|x|^{-2\sigma}$ with a sharp coupling constant $a\ge -C_{\sigma,n}$ ($C_{\sigma,n}$ is the best constant of Hardy's inequality of order $\sigma$). In the present paper we consider several  sharp global estimates for the resolvent and the solution to the time-dependent Schr\"odinger equation associated with $H$. In the case of the subcritical coupling constant $a>-C_{\sigma,n}$, we first prove {\it uniform resolvent estimates} of Kato--Yajima type for all $0<\sigma<n/2$, which turn out to be equivalent to {\it Kato smoothing estimates} for the Cauchy problem. We then establish {\it Strichartz estimates} for $\sigma>1/2$ and {\it uniform Sobolev estimates} of Kenig--Ruiz--Sogge type for $\sigma\ge n/(n+1)$. These extend the same properties for the Schr\"odinger operator with the inverse-square potential to the higher-order and fractional cases. Moreover, we also obtain {\it improved Strichartz estimates with a gain of regularities} for general initial data if $1<\sigma<n/2$ and for radially symmetric data if $n/(2n-1)<\sigma\le1$, which extends the corresponding results for the free evolution to the case with Hardy potentials. These arguments can be further applied to a large class of  higher-order inhomogeneous elliptic operators and even to certain long-range metric perturbations of the Laplace operator. Finally, in the critical coupling constant case ({\it i.e.}, $a=-C_{\sigma,n}$), we show that the same results as in the subcritical case still hold for functions orthogonal to radial functions.
\end{abstract}

\maketitle


\section{Introduction}
\subsection{Background and problems}
\label{subsection_1_1}
In the present paper we mainly study generalized Schr\"odinger operators $H=(-\Delta)^\sigma+a|x|^{-2\sigma}$ on $\R^n$ for $ 0<\sigma<n/2$ and $a\in \R$ satisfying
\begin{align}
\label{best}
 a\ge-C_{\sigma,n}:=-\left\{\frac{2^\sigma\Gamma\left(\frac{n+2\sigma}{4}\right)}{\Gamma\left(\frac{n-2\sigma}{4}\right)}\right\}^2.
\end{align}
It is well known (see \cite[Theorem 2.5]{Her}) that $C_{\sigma,n}$ is the best constant in the following (generalized) Hardy inequality:
\begin{align}
\label{Hardy}
C_{\sigma,n}\int |x|^{-2\sigma}|u(x)|^2dx\le \int ||D|^\sigma u(x)|^2dx,\quad u\in C_0^\infty (\R^n),
\end{align}
where $|D|=(-\Delta)^{1/2}$. Note that \eqref{Hardy} holds if and only if $0<\sigma<n/2$ and that  $C_{\sigma,n}$ can be computed more explicitly in several cases (see e.g. \cite[Corollary 14]{DaHi}), e.g. $C_{1,n}=[(n-2)/2]^2$, $C_{2,n}=[n(n-4)/4]^2$ and $C_{1/2,3}=2/\pi$ are the sharp constants in {\it classical Hardy's, Rellich's and Kato's inequalities}, respectively. It follows from Hardy's inequality \eqref{Hardy} that the higher-order ($\sigma\ge1$) and fractional ($\sigma<1$) Schr\"odinger operators  $H=(-\Delta)^\sigma+a|x|^{-2\sigma}$ can be realized as self-adjoint operators on $L^2(\R^n)$ (see Subsection \ref{main result} for the precise definition).

The operator $H$, or more generally, the operator $H_{\Phi,\sigma}$ of the form
$$
H_{\Phi,\sigma}=(-\Delta)^\sigma+\Phi(x/|x|)|x|^{-2\sigma},\quad 0<\sigma<n/2,
$$
with $\Phi\in L^\infty(\mathbb S^{n-1})$ satisfying $\inf \Phi>-C_{\sigma,n}$, arise naturally in mathematical physics. If $\sigma=1/2$, then $H=|D|+a|x|^{-1}$ is the massless semi-relativistic Schr\"odinger operator with the Coulomb potential describing (massless) relativistic particles in the Coulomb field. If $\sigma=1$ and $\Phi(\theta)=\vec  a\cdot x/|x|$ with $\vec a\in \R^3$, then $H_{\Phi,1}$ is the Hamiltonian describing electrons in the electric point-like dipole field \cite{Levy}. Moreover, in the fractional case $\sigma\in (0,1)$, the operator $(-\Delta)^{\sigma}+V(x)$ is the Hamiltonian for the fractional quantum mechanics introduced by \cite{Laskin} and $H_{\Phi,\sigma}$ with $\sigma\in (0,1)$ can be regarded as a fractional counterpart of $H_{\Phi,1}$. In the higher-order case $\sigma>1$, the operators $H$ and $H_{\Phi,\sigma}$ also appear in the study of several higher-order dispersive partial differential equations (see e.g. \cite{Carles}).

Since $H$ is self-adjoint on $L^2(\R^n)$, the resolvent $(H-z)^{-1}$ for $\Im z\neq0$ and the unitary group $e^{-itH}$ can be well-defined  on $L^2(\R^n)$, which are related with the solutions to
the following stationary and time-dependent Schr\"odinger  equations  with Hardy potentials:
\begin{align}
\label{stationary}
\big(H-z\big)u(x)&=f(x),\ x\in \R^n;
\end{align}
\vskip-0.8cm
\begin{align}
\label{Cauchy}
\big(i\partial_t-H\big)\psi(t,x)&=F(t,x),\ \psi(0,x)=\psi_0(x),\ (t,x)\in \R^{1+n}.
\end{align}
For the fractional Laplacian $H=(-\Delta)^{\sigma}$ and the second-order operator $H=-\Delta +a|x|^{-2}$, the equations \eqref{stationary} and \eqref{Cauchy} have been extensively studied in several directions such as the spectral and scattering theory, and applications to the nonlinear Schr\"odinger (NLS for short) equations associated with $H$ (see below for more details).  The main purpose of the present paper is to prove several kinds of interesting global dispersive properties for the solutions to  \eqref{stationary} and \eqref{Cauchy} associated with $H=(-\Delta)^\sigma+a|x|^{-2\sigma}$ for the general cases $a\neq0$ and $\sigma\neq1$, which particularly gives a unified approach to  many of previous works for the cases $a=0$ or $\sigma=1$. Moreover, we also consider a more general class of dispersive operators $P_0(D)+V(x)$ than the operator $(-\Delta)^\sigma+a|x|^{-2\sigma}$, which particularly covers the operator $H_{\Phi,\sigma}$ discussed above (see Assumption \ref{assumption_A} and Section \ref{section_5} below)
\vskip0.3cm
Let us discuss the purpose of the paper and  the previous literature in more details.
The main objects include the following sharp dispersive estimates:
\vskip0.2cm
\begin{itemize}
\item {\it Uniform resolvent estimates:}
\begin{align}
\label{theorem_1_2'}
&\sup_{z\in \C\setminus\R}\norm{|x|^{-\sigma+\gamma}|D|^\gamma(H-z)^{-1}|D|^\gamma |x|^{-\sigma+\gamma}}_{L^2\to L^2}<\infty
\end{align}
 for any $n\ge2$, $0<\sigma<n/2$ and $\sigma-n/2< \gamma<\sigma-1/2$. In other words, $|x|^{-\sigma+\gamma}|D|^\gamma$ is $H$-supersmooth in the sense of Kato--Yajima \cite{KaYa}.
\vskip0.2cm
\item {\it Kato smoothing estimates:}
\begin{align}
\label{theorem_2_2'}
\norm{|x|^{-\sigma+\gamma}|D|^\gamma e^{-itH}\psi_0}_{L^2_tL^2_x}\lesssim \norm{\psi_0}_{L^2_x}\end{align}
for any $n\ge2$, $0<\sigma<n/2$ and $\sigma-n/2< \gamma<\sigma-1/2$.
\vskip0.2cm
\item {\it (Standard) Strichartz estimates:}
\begin{align}
\label{lemma_section_3_1_1'}
\norm{e^{-itH}\psi_0}_{L^p_tL^{q}_x}\lesssim \norm{\psi_0}_{L^2_x}
\end{align}
for all $\psi_0\in L^2$ if $1\le\sigma<n/2$ and for any radially symmetric $\psi_0\in L^2$ if $n/(2n-1)<\sigma<1$, where $n\ge3$ and $(p,q)$ is ${n}/{(2\sigma)}$-admissible (see \eqref{admissible} for the definition of admissible pairs).
\vskip0.2cm
\item {\it Strichartz estimates with a gain or loss of regularities:}
\begin{equation}
\begin{aligned}
\label{lemma_section_3_1_3'}
\norm{|D|^{2(\sigma-1)/p}e^{-itH}\psi_0}_{L^p_tL^{q}_x}&\lesssim \norm{\psi_0}_{L^2_x}
\end{aligned}
\end{equation}
for all $n\ge2$, $1/2<\sigma<n/2$ and ${n}/{2}$-admissible pairs $(p,q)$.
\vskip0.3cm
\item  {\it $L^p-L^{q}$ resolvent estimates:}
\begin{align}
\label{theorem_5_1'}
\big\|(H-z)^{-1}\big\|_{L^{p}-L^{p'}}\lesssim |z|^{\frac{n}{\sigma}\big(\frac
{1}{p}-\frac{1}{2}\big)-1},\quad z\in \C\setminus\{0\},
\end{align}
for $n\ge3$, $n/(n+1)\le \sigma<n/2$ and $2n/(n+2\sigma)\le p\le 2(n+1)/(n+3)$. We remark that the inequalities \eqref{theorem_5_1'} are called {\it uniform Sobolev estimates}  in the sense of Kenig--Ruiz--Sogge \cite{KRS}, and the ranges of $\sigma$ and $p$ are optimal.
\end{itemize}
\vskip0.2cm
We also study retarded estimates for the inhomogeneous evolution $\int_0^te^{-i(t-s)H}F(s)ds$ related to \eqref{theorem_2_2'}--\eqref{lemma_section_3_1_3'} which are of particular interest for applications to NLS equations associated with $H$.
Among these estimates, the estimates \eqref{lemma_section_3_1_3'} have an additional smoothing effect (a gain of regularities) for the higher-order case $\sigma>1$ or loss of regularities for the fractional case $\sigma<1$ compared with the second order case $\sigma=1$. This property reflects the stronger or weaker dispersive effect in the high frequency mode of dispersive equations of order $2\sigma\neq1$.
We also remark that the operators  $H=(-\Delta)^\sigma+a|x|^{-2\sigma}$ with $a>-C_{\sigma,n}$ are critical in several senses concerning the validity of these estimates (see Remark \ref{remark_theorems} below for more details).
\vskip0.2cm

These global estimates \eqref{theorem_1_2'}--\eqref{theorem_5_1'} have been extensively studied in many works for the free case $H=(-\Delta)^{\sigma}$ with any $\sigma>0$ and the second-order case $H=-\Delta+a|x|^{-2}$. In the case of the free fractional Laplacian $H=(-\Delta)^{\sigma}$, we refer to \cite{KaYa}, \cite{Wat}, \cite{Sug} and \cite{RuSu} for the uniform resolvent estimate \eqref{theorem_1_2'} and the smoothing estimate \eqref{theorem_2_2'}, to \cite{Str}, \cite{GiVe}, \cite{Yaj}, \cite{KPV}, \cite{KeTa}, \cite{Pau}, \cite{Guo} and \cite{GLNY} for the Strichartz estimates \eqref{lemma_section_3_1_1'} and \eqref{lemma_section_3_1_3'}, and to \cite{KRS}, \cite{Gut}, \cite{HYZ}, \cite{SYY} and \cite{Cue} for the uniform Sobolev estimate \eqref{theorem_5_1'}, respectively. In the second-order case  $H=-\Delta+a|x|^{-2}$ with $a>-C_{1,n}$, two estimates \eqref{lemma_section_3_1_1'} and \eqref{lemma_section_3_1_3'} (which are the same in this case), as well as \eqref{theorem_1_2'} and \eqref{theorem_2_2'},  were proved in seminal works by Burq et al \cite{BPST1,BPST2}, while the double endpoint Strichartz estimate for the inhomogeneous evolution $\int_0^te^{-i(t-s)H}F(s)ds$ and the uniform Sobolev estimate were obtained by \cite{BoMi} and \cite{Mizutani_JST}, respectively. For the critical constant case $a=-C_{1,n}$, we refer to \cite{Mizutani_JDE}, \cite{Mizutani_JST}.
\vskip0.2cm
It is well known that all of the estimates \eqref{theorem_1_2'}--\eqref{theorem_5_1'} (for the free or second-order cases) have played important roles in the study of broad areas, especially the spectral and scattering theory. In particular, Strichartz estimates \eqref{lemma_section_3_1_1'} and \eqref{lemma_section_3_1_3'} are one of fundamental tools for NLS equations (see e.g. \cite{Tao}). We also refer to \cite{ZZ}, \cite{KMVZZ} and references therein for a recent development on NLS equations with Hardy potentials. The uniform Sobolev estimate \eqref{theorem_5_1'} was originally used to proving unique continuation properties for the operator $-\Delta+V(x)$ with rough potentials $V\in L^{n/2}$. More recently, it has played a crucial role in studying Keller--Lieb--Thirring type eigenvalue bounds for Schr\"odinger operators with complex valued potentials (see \cite{Fra1}, \cite{Fra2}, \cite{Cue} and reference therein, also \cite{Mizutani_JST} for the case with Hardy potentials). For further applications of uniform resolvent and Sobolev estimates, we refer to \cite{Gut}, \cite{HYZ}, \cite{Mizutani_APDE} (see also the discussion after Theorem \ref{theorem_5} below).
\vskip0.2cm
Besides the free or second-order cases, we are mainly devoted to establish these estimates \eqref{theorem_1_2'}--\eqref{theorem_5_1'} for $H=(-\Delta)^\sigma+a|x|^{-2\sigma}$ with $0<\sigma<n/2$ and $a>-C_{\sigma,n}$. We also consider the critical case $a=-C_{\sigma,n}$ and prove some partial results. This naturally extends the known literatures for the operators $(-\Delta)^\sigma$ and $-\Delta+a|x|^{-2}$ describe above, to the operator $(-\Delta)^\sigma+a|x|^{-2\sigma}$. To our best knowledge, there is no previous literature on these estimates \eqref{theorem_1_2'}--\eqref{theorem_5_1'} for higher-order or fractional Schr\"odinger operators with large potentials $V(x)$ which has the critical decay rate, {\it i.e}, $V(x)=O(\<x\>^{-2\sigma})$.
Moreover, our model is more general than the operator $(-\Delta)^\sigma+a|x|^{-2\sigma}$ in the following sense. On one hand, we consider not only the Hardy potential $a|x|^{-2\sigma}$ but also a wide class of repulsive potentials $V(x)$, even including some examples satisfying $|x|^{2\sigma}V\notin L^\infty$ (see Assumption \ref{assumption_A} and Examples \ref{example_1} and \ref{example_2} below). On the other hand, our method can be applied to not only $(-\Delta)^\sigma$ but also a wide class of dispersive operators $P_0(D)$ (as the principal part of $H$). For instance, our class of $P_0(D)$ particularly includes the massive fractional operator $(-\Delta+m)^\sigma$ with $m>0$, the higher-order inhomogeneous elliptic operators of the form $\sum_{j=1}^J(-\Delta)^j$ and the Laplace-Beltrami operator of the form $-\nabla\cdot G_0(x)\nabla$ with  $G_0$ being a small long-range perturbation of the identity matrix (see Section \ref{section_5} for more details).
\vskip0.2cm
Among these desired estimates, the uniform resolvent estimate \eqref{theorem_1_2'} is fundamental and play a central role in proving the other estimates \eqref{theorem_2_2'}--\eqref{theorem_5_1'}. For the case $\sigma=1$, \eqref{theorem_1_2'} was obtained by \cite{BPST1} via the spherical harmonics decomposition and analysis of Hankel operators. \cite{BPST2} provided an alternative proof of \eqref{theorem_1_2'} (when $\sigma=1$) based on the method of multipliers. However, if $\sigma\notin \N$, it seems to be difficult to apply these methods due to the non-locality of $(-\Delta)^\sigma$. Also, even in the case $\sigma\in \N$, these methods will involve much longer and complicated computations compared with the case $\sigma=1$. To overcome these difficulties, we establish a different method, which is based on Mourre's theory \cite{Mou} and generalizes a previous argument by Hoshiro \cite{Hos}. {\it This method enables us to deal with general cases $0<\sigma<n/2$ in a unified way}.
\vskip0.2cm
Finally, it is worth noting that there are many interesting works on uniform resolvent, dispersive and Strichartz estimates for higher-order and fractional Schr\"odinger operators or Dirac operators, involving potentials which decay faster than $|x|^{-2\sigma}$ (see \cite{FSWY}, \cite{FSY}, \cite{GT}, \cite{EGG}, our subsequent work \cite{MiYa2} and references therein), where in particular, an amount of background analysis and related decay estimates about these operators can be found.
\vskip0.3cm
\subsection{Notations}
\label{subsection_notation}
To state our main results, we will use the following notations.
\begin{itemize}
\item For positive constants or operators $A,B$, $A\lesssim B$ (resp. $A\gtrsim B$) means $A\le cB$  (resp. $A\ge cB$) with some constant $c>0$. $A\sim B$ means $cB\le A\le c'B$ with some $c'>c>0$.
\vskip0.1cm
\item $\<\cdot\>$ stands for $\sqrt{1+|\cdot|^2}$.
\vskip0.1cm
\item $\mathbb B(X,Y)$ denotes the family of bounded operators from $X$ to $Y$, $\mathbb B(X)=\mathbb B(X,X)$ and $\norm{\cdot}_{X\to Y}:=\norm{\cdot}_{\mathbb B(X,Y)}$. We also set $\norm{f}:=\norm{f}_{L^2}$ and $\norm{A}:=\norm{A}_{L^2\to L^2}$.
\vskip0.1cm
\item $\H^s(\R^n)$ denotes the $L^2$-based Sobolev space. $L^{p,q}(\R^n)$ denotes the Lorentz space (see Appendix \ref{appendix_misc} for basic properties of Lorentz spaces). $\<f,g\>$ stands for the inner product  in $L^2$, as well as the duality couplings $\<\cdot,\cdot\>_{L^{p',q'},L^{p,q}}$ and $\<\cdot,\cdot\>_{\H^{-\sigma},\H^\sigma}$, where $p':=p/(p-1)$ is the H\"older conjugate exponent of $p$.
\vskip0.1cm
\item Given a Banach space $X$, $L^p_tX:=L^p(\R;X)$ denotes the Bochner space. In particular, $L^p_tL^q_x:=L^p(\R;L^q(\R^n))$. Let $L^2_\omega=L^2(\mathbb S^{n-1},d\omega)$ with the standard measure associated to the round metric and $\mathcal L^{p}_r=L^{p}(\R^+,r^{n-1}dr)$. Define the space $\mathcal L^{p}_rL^2_\omega$ by the following norm
$$
\norm{f}_{\mathcal L^{p}_rL^2_\omega}=\|\|f(r\omega)\|_{L^2_\omega}\|_{\mathcal L^{p}_r},\quad  r>0,\ \omega\in \mathbb S^{n-1}.
$$
Let $B[\mathcal L^{p}_rL^2_\omega]$ denote a Besov-type space defined by the norm
\begin{align}
\label{Besov}
\norm{f}_{B[\mathcal L^{p}_rL^2_\omega]}:=\Big(\sum_{j\in \Z}\norm{\varphi_j(D)f}_{\mathcal L^{p}_rL^2_\omega}^2\Big)^{1/2},
\end{align}
where $\{\varphi_j\}_{j\in \Z}$ with $\varphi_j(\xi)=\varphi(2^{-j}\xi)$ is the homogeneous dyadic partition of unity, namely $\varphi\in C_0^\infty(\R^n)$, $\varphi$ is radially symmetric and even, $0\le \varphi\le1$, $\varphi(\xi)=1$ for $1/2\le|\xi|\le2$, $\varphi(\xi)=0$ if $|\xi|<1/4$ or $4<|\xi|$ and
$
\sum_{j\in \Z}\varphi_j(\xi)=1$ for all $\xi\neq0$.
\end{itemize}
\vskip0.3cm
\subsection{Main results} \label{main result}
Let $0<\sigma<n/2$. In this subsection we state our main results for the following Schr\"odinger-type operator $H$ of order $2\sigma$ on $L^2(\R^n)$:
\begin{align}
\label{H}
H_0=(-\Delta)^\sigma
,\quad H=H_0+V(x),
\end{align}
where $\Delta=\sum_{j=1}^n\partial_{x_j}^2$. Note that we will consider a more general class of Fourier multipliers $P_0(D)$ as the main term $H_0$ in Section \ref{section_5}. To state the assumption on $V$, we set
\begin{equation}
\begin{aligned}
\label{symbol_ell}
H_\ell=(2\sigma)^\ell (-\Delta)^\sigma,\quad V_\ell(x)&=(-x\cdot\nabla_x)^\ell V(x),\quad \ell=1,2.
\end{aligned}
\end{equation}
If $P_0(\xi)=|\xi|^{2\sigma}$ denotes the symbol $(-\Delta)^\sigma$, then $H_\ell$ is written in the form
$$
H_\ell=P_\ell(D)=\F^{-1}P_\ell \F,\quad P_\ell(\xi)=(\xi\cdot\nabla_\xi)^\ell P_0(\xi),
$$
where $\F$ is the Fourier transform. We impose the following assumption on $V$.
\begin{assumption}
\label{assumption_A} $V$ is a real-valued function on $\R^n$ such that, for all $\ell=0,1,2$, $(x\cdot\nabla)^\ell V\in L^1_{\mathrm{loc}}(\R^n)$ and $(x\cdot\nabla)^\ell V$ is $H_0$-form bounded, namely $|(x\cdot\nabla)^\ell V|^{1/2}(H_0+1)^{-1/2}$ is bounded on $L^2(\R^n)$. Moreover, the following estimates hold for all $u\in C_0^\infty(\R^n)$:
\begin{align}
\label{assumption_A_1}
\<(H_0+V)u,u\>&\gtrsim \<(-\Delta)^\sigma u,u\>,\\
\label{assumption_A_2}
\<(H_1+V_1)u,u\>&\gtrsim \<(-\Delta)^\sigma u,u\>,\\
\label{assumption_A_3}
|\<(H_2+V_2)u,u\>|&\lesssim\<(H_1+V_1)u,u\>.
\end{align}
\end{assumption}
A typical example of $V$ is the Hardy potential $a|x|^{-2\sigma}$ with $a>-C_{\sigma,n}$ (see Example \ref{example_2}). Note that the critical case $a=-C_{\sigma,n}$ will be also considered later (see Section \ref{section_6}).
\vskip0.2cm
Here and in the sequel, we frequently use the following norm equivalence
$$
\<H_0u,u\>+\|u\|^2\sim \|u\|_{\H^\sigma}^2.
$$
It follows from this equivalence and Assumption \ref{assumption_A} that the sesquilinear form $$Q_H(u,v):=\<H_0u,v\>+\<Vu,v\>=\int(H_0u\cdot\overline{ v}+Vu\overline v)dx$$ is well-defined for $u,v\in C_0^\infty(\R^n)$ and satisfies
\begin{align}
\label{assumption_A_4}
\||D|^\sigma u\|^2\lesssim Q_H(u,u)\lesssim \|u\|_{\H^\sigma}^2.
\end{align}
In particular, $Q_H$ is closable. We denote by the same symbol $Q_H$ its closed extension through the graph norm $({Q_H(u,u)+\|u\|^2})^{1/2}$. Precisely speaking,  $H$ is defined as a unique self-adjoint operator generated by $Q_H$, that is the Friedrichs extension of $H_0+V$ defined on $C_0^\infty(\R^n)$. By virtue of \eqref{assumption_A_4}, the form domain $D(H^{1/2})$ coincides with $\H^\sigma(\R^n)$. Note that the spectrum of $H$ coincides with $[0,\infty)$ and is purely absolutely continuous (see Remark \ref{remark_theorem_1_1} below).
\vskip0.3cm
Here we give some sufficient conditions and examples to ensure Assumption \ref{assumption_A}.

\begin{example}
\label{example_1}
Thanks to O'Neil's inequality \eqref{Holder} and Sobolev's inequality \eqref{Sobolev}, it is enough to assume $(x\cdot\nabla)^\ell V\in L^{\frac{n}{2\sigma},\infty}(\R^n)+L^\infty(\R^n)$ for all $\ell=0,1,2$ to ensure $(x\cdot\nabla)^\ell V\in L^1_{\loc}(\R^n)$ and $H_0$-form boundedness. Moreover, if there exists $0<\delta<1$ such that
\begin{align}
\label{example_1_1}
-\int V|u|^2dx&\le (1-\delta)\int ||D|^{\sigma } u|^2dx,\\
\label{example_1_2}
-\int V_1|u|^2dx&\le (2\sigma-\delta)\int ||D|^{\sigma} u|^2dx,\\
\label{example_1_3}
\int|2\sigma V_1-V_2||u|^2dx&\lesssim\int (||D|^{\sigma} u|^2+V_1|u|^2)dx
\end{align}
for all $u\in C_0^\infty(\R^n)$, then \eqref{assumption_A_1}--\eqref{assumption_A_3} are satisfied (see Appendix \ref{appendix_example} below).
\end{example}

\begin{example}
\label{example_2}
The most important example is the Hardy potential
\begin{align}
\label{Hardy_potential}
V(x)=a|x|^{-2\sigma},\quad a>-C_{\sigma,n},
\end{align}
where $C_{\sigma,n}$ is given by \eqref{best}. One can also consider its generalization $$V(x)=\Phi(x/|x|)|x|^{-2\sigma},\quad \Phi\in L^\infty(\mathbb S^{n-1}),\quad \inf \Phi>-C_{\sigma,n}.$$
In these cases, Assumption \ref{assumption_A} follows easily from \eqref{Hardy} and Example \ref{example_1} since $H_\ell=(2\sigma)^\ell H$ and $V_\ell=(2\sigma)^\ell V$. $V(x)=a\<x\>^{-2\sigma}$ also satisfies Assumption \ref{assumption_A} if $a>-C_{\sigma,n}$. Moreover, it is also worth noting that Assumption A does not require any specific decay rate of $V$. In fact, Assumption \ref{assumption_A} allows several potentials decaying slower than $|x|^{-2\sigma}$. For instance, $a|x|^{-\mu}$ and $a\<x\>^{-\mu}$ fulfill the conditions in Example \ref{example_1} and hence Assumption \ref{assumption_A} if $\mu\in (0,2\sigma)$ and $a>0$.
\end{example}

Now we state the main results. The first result is a uniform resolvent estimate of Kato--Yajima type which will play a central role for other results in the paper.

\begin{theorem}	
\label{theorem_1}
If $n\ge2$, $0<\sigma<n/2$ and $H=(-\Delta)^\sigma+V$ satisfies Assumption \ref{assumption_A}, then \begin{align}
\label{theorem_1_1}
\sup_{z\in\C\setminus[0,\infty)}\norm{|x|^{-\sigma+\gamma}|D|^{\gamma}(H-z)^{-1}|D|^{\gamma}|x|^{-\sigma+\gamma}}_{L^2\to L^2}<\infty
\end{align}
for all $\sigma-n/2< \gamma<\sigma-1/2$. In other words, $|x|^{-\sigma+\gamma}|D|^{\gamma}$ is $H$-supersmooth in the sense of Kato--Yajima \cite{KaYa}. In particular, if $1/2<\sigma<n/2$ then the following estimate holds:
\begin{align}
\label{theorem_1_2}
\sup_{z\in \C\setminus[0,\infty)}\norm{|x|^{-\sigma}(H-z)^{-1}|x|^{-\sigma}}_{L^2\to L^2}<\infty.
\end{align}
\end{theorem}

\vskip0.2cm
\begin{remark}
\label{remark_theorem_1_1}
As a consequence of \eqref{theorem_1_1}, $H$ is purely absolutely continuous and has no eigenvalues, namely $\sigma(H)=\sigma_{\mathrm{ac}}(H)=[0,\infty)$ and $\sigma_{\mathrm{sc}}(H)=\sigma_{\mathrm{p}}(H)=\emptyset$ (see \cite[Theorem XIII. 23]{ReSi}).
\end{remark}

\vskip0.2cm
As a corollary, the limiting absorption principle and uniform bounds for the boundary resolvents $(H-\lambda\mp i0)^{-1}$ can be also derived.

\begin{corollary}
\label{corollary_1}
Assume in addition to the condition in Theorem \ref{theorem_1} that
\begin{align}
\label{corollary_1_1}
\<(H_1+V_1)u,u\>\gtrsim \<(H_0+V)u,u\>,\quad u\in C_0^\infty(\R^n).
\end{align}
Let $s>1/2$ and $A$ be the generator of the dilation group (see \eqref{dilation}). Then the limits $$\ds \<A\>^{-s}(H-\lambda\mp i0)^{-1} \<A\>^{-s}:=\lim_{\ep\searrow0} \<A\>^{-s}(H-\lambda\mp i\ep)^{-1} \<A\>^{-s}\in \mathbb B(L^2)$$
exist for all $\lambda>0$. Moreover, the following uniform estimate holds:
\begin{align}
\label{corollary_1_2}
\sup_{\lambda>0}\norm{|x|^{-\sigma+\gamma}|D|^{\gamma}(H-\lambda\mp i0)^{-1}|D|^{\gamma}|x|^{-\sigma+\gamma}}_{L^2\to L^2}<\infty.
\end{align}
\end{corollary}
\vskip0.2cm
Note that \eqref{corollary_1_1} holds if $V$ satisfies the conditions in Example \ref{example_1}. \vskip0.3cm
Next we consider the time-dependent Schr\"odinger equation associated with $H$:
\begin{align}
\label{Cauchy_2}
(i\partial_t-H) \psi(t,x)=F(t,x),\quad \psi(0,x)=\psi_0(x),
\end{align}
with given data $\psi_0$ and $F$, where the solution $\psi$ is given by the Duhamel formula:
\begin{equation}
\label{Duhamel}
\psi=e^{-itH}\psi_0-i\int_0^te^{-i(t-s)H}F(s)ds.
\end{equation}
 The next theorem, called Kato smoothing estimates or also Kato--Yajima estimates, is a direct consequence of Theorem \ref{theorem_1}.

\begin{theorem}	
\label{theorem_2}
Let $n\ge2$, $0<\sigma<n/2$, $\sigma-n/2<\gamma<\sigma-1/2$ and $H=(-\Delta)^\sigma+V$ satisfy Assumption \ref{assumption_A}. Then $\psi$ defined by \eqref{Duhamel} satisfies the following estimates:
\begin{align}
\label{theorem_2_1}
\norm{|x|^{-\sigma+\gamma}|D|^{\gamma}\psi}_{L^2_tL^2_x}&\lesssim \norm{\psi_0}_{L^2_x}+\norm{|x|^{\sigma-\gamma}|D|^{-\gamma}F}_{L^2_tL^2_x}.
\end{align}
In particular, if in addition $\sigma>1/2$, then the following local decay estimate holds:
\begin{align}
\label{theorem_2_2}
\norm{|x|^{-\sigma}\psi}_{L^2_tL^2_x}&\lesssim \norm{\psi_0}_{L^2_x}+\norm{|x|^\sigma F}_{L^2_tL^2_x}.
\end{align}
\end{theorem}
\vskip0.2cm
\begin{remark}
\label{remark_smoothing}
For the free case $V\equiv0$, the estimate \eqref{theorem_2_1} is well known and the condition $\sigma-n/2<\gamma<\sigma-1/2$ is known to be sharp (see e.g. \cite{RuSu} and references therein).
\end{remark}

We have seen that the above theorems hold for all $V$ satisfying Assumption \ref{assumption_A}, including some slowly decaying potentials given in Example \ref{example_2}. In the following three theorems, we impose an additional condition $|x|^\sigma V\in L^{n/\sigma,\infty}$ or $|x|^{2\sigma} V\in L^\infty$, both of which particularly holds for the Hardy potential \eqref{Hardy_potential}. This roughly means that $V$ decays like $|x|^{-2\sigma}$ at infinity. The optimality of this decay rate will be discussed in Remark \ref{remark_theorems} below.
 \vskip0.3cm
Now we state the result on Strichartz estimates which, from a viewpoint of applications to nonlinear problems, is probably the most important consequence of the paper. Recall that, for a given $\alpha>0$, a pair $(p,q)$ is said to be sharp $\alpha$-admissible if
\begin{align}
\label{admissible}
2\le p,q\le \infty,\quad 1/p=\alpha(1/2-1/q),\quad(p,q,\alpha)\neq(2,\infty,1).
\end{align}
In what follows we omit the word {\it sharp}, calling a pair $(p,q)$ satisfying \eqref{admissible} to be $\alpha$-admissible for simplicity. When $\alpha\ge1$, the pair $(2,2\alpha/(\alpha-1))$ is called the endpoint.

\vskip0.3cm
\begin{theorem}[Higher-order case]
\label{theorem_3}
Let $n\ge3$, $1<\sigma<n/2$, $H=(-\Delta)^\sigma+V$ satisfy Assumption \ref{assumption_A} and $|x|^{\sigma}V\in L^{n/\sigma,\infty}(\R^n)$. Then the following statements hold for $\psi$ given by \eqref{Duhamel}.
\vskip0.3cm
\begin{itemize}
\item \underline{(Standard) Strichartz estimates}: if $(p_1,q_1)$ and $(p_2,q_2)$ are $n/(2\sigma)$-admissible, then
\begin{align}
\label{theorem_3_1}
\norm{\psi}_{L^{p_1}_tL^{q_1}_x}\lesssim \norm{\psi_0}_{L^2_x}+\norm{F}_{L^{p_2'}_tL^{q_2'}_x}.
\end{align}
In particular, the following endpoint estimates hold:
\begin{align*}
\|e^{-itH}\psi_0\|_{L^{2}_tL^{\frac{2n}{n-2\sigma}}_x}&\lesssim \|\psi_0\|_{L^2_x},\\
\bignorm{\int _{0}^t e^{-i(t-s)H}F(s)ds}_{L^2_tL^{\frac{2n}{n-2\sigma}}_x}&\lesssim \|F\|_{L^2_tL^{\frac{2n}{n+2\sigma}}_x}.
\end{align*}
 \vskip0.3cm
\item \underline{Improved Strichartz estimates}: if $(p,q)$ and $(\tilde p,\tilde q)$ are ${n}/{2}$-admissible, then
\begin{align}
\label{theorem_3_2}
\norm{|D|^{{2(\sigma-1)}/{p}}\psi}_{L^p_tL^{q}_x}\lesssim \norm{\psi_0}_{L^2_x}+\norm{|D|^{{2(1-\sigma)}/\tilde p}F}_{L^{\tilde p'}_tL^{\tilde q'}_x},
\end{align}
which particularly implies the following endpoint estimates:
\begin{align*}
\big\||D|^{\sigma-1}e^{-itH}\psi_0\big\|_{L^2_tL^{\frac{2n}{n-2}}_x}&\lesssim \norm{\psi_0}_{L^2_x},\\
\bignorm{|D|^{\sigma-1}\int _{0}^t e^{-i(t-s)H}F(s)ds}_{L^2_tL^{\frac{2n}{n-2}}_x}&\lesssim \||D|^{1-\sigma} F\|_{L^2_tL^{\frac{2n}{n+2}}_x}.
\end{align*}
\end{itemize}
\end{theorem}
\vskip0.3cm
Note that in case of $\sigma=1$, the two estimates \eqref{theorem_3_1} and \eqref{theorem_3_2} are the same and were obtained by \cite{BPST1} and \cite{BPST2}. On the other hand, if $\sigma>1$, \eqref{theorem_3_1} in fact follows from \eqref{theorem_3_2}. Indeed, if $(p,q)$ is $n/2$-admissible and $(p,q_1)$ is $n/(2\sigma)$-admissible, then $p,q,q_1$ satisfy $1/q-1/q_1=2(\sigma-1)/(np)$ and Sobolev's inequality \eqref{Sobolev} thus implies
\begin{align}
\label{Sobolev_2}
\norm{f}_{L^{q_1}}\lesssim \norm{|D|^{{2(\sigma-1)}/{p}}f}_{L^{q}}.
\end{align}
In this sense, {\it \eqref{theorem_3_2} has an additional smoothing effect compared with \eqref{theorem_3_1} if $\sigma>1$.} This is one of features of higher-order dispersive equations.
\vskip0.3cm
\begin{theorem}[Second-order and fractional cases]
\label{theorem_4}
Let $n\ge2$, $1/2<\sigma\le1$, $H=(-\Delta)^\sigma+V$ satisfy Assumption \ref{assumption_A} and $|x|^{2\sigma}V\in L^\infty(\R^n)$. Then the following statements hold for $\psi$ given by \eqref{Duhamel}.
\begin{itemize}
\vskip0.3cm
\item \underline{Strichartz estimates with a loss of derivatives}: if $(p,q)$ and $(\tilde p,\tilde q)$ are ${n}/{2}$-admissible and $p,\tilde p>2$, then the same estimates as \eqref{theorem_3_2} are satisfied.
\vskip0.3cm
\item \underline{Spherically averaged Strichartz estimates}: let $(p_1,q_1)$ and $(p_2,q_2)$ satisfy
\begin{align}
\label{theorem_4_1}
2<p_j,q_j\le\infty,\quad 1/p_j\le (n-1/2)(1/2-1/q_j)
\end{align}
and $1/p_j\neq (n-1/2)(1/2-1/q_j)$ if $n=2$, and let $$s_j=s(p_j,q_j):=-n(1/2-1/q_j)+2\sigma/p_j.$$
Then the following estimates hold:
\begin{align}
\label{theorem_4_2}
\norm{|D|^{s_1} \psi}_{L^{p_1}_t B[\mathcal L^{q_1}_rL^2_\omega]}\lesssim \norm{\psi_0}_{L^2_x}+\||D|^{-s_2}F\|_{L^{p_2'}_t B[\mathcal L^{q_2'}_rL^2_\omega]},
\end{align}
where the Besov-type space $B[\mathcal L^{p}_rL^2_\omega]$ is defined by \eqref{Besov}.
\vskip0.3cm
\item \underline{Improved endpoint Strichartz estimates under the radial symmetry}: let $V$, $\psi_0$ and $F$ be radially symmetric. If $n\ge3$, $n/(2n-1)<\sigma\le1$,  $q_1,q_2>(4n-2)/(2n-3)$, then the following endpoint estimates hold:
\begin{align}
\label{theorem_4_3}
\norm{|D|^{s(2,q_1)}\psi}_{L^2_t L^{q_1}_x}\lesssim \norm{\psi_0}_{L^2_x}+\||D|^{-s(2,q_2)}F\|_{L^{2}_tL^{q_2'}_x},
\end{align}
where $s(2,q_j):=-n(1/2-1/q_j)+\sigma$ for $j=1,2$. As $\sigma=1$, one particularly has the following improved Strichartz estimates for the Schr\"odinger operator $-\Delta+V(x)$:
\begin{align}
\label{theorem_4_4}
\norm{|D|^{s}\psi}_{L^2_tL^{\frac{2n}{n-2+2s}}_x}\lesssim \norm{\psi_0}_{L^2_x}+\norm{|D|^{-s}F}_{L^2_tL^{\frac{2n}{n+2-2s}}_x},\quad 0<s<\frac{n-1}{2n-1}.
\end{align}
 \end{itemize}
\end{theorem}
\vskip0.3cm
Theorems \ref{theorem_3} and \ref{theorem_4} are well known in the free case $H=(-\Delta)^\sigma$ (see \cite{KPV}, \cite{KeTa}, \cite{Pau}, \cite{Guo} and \cite{GLNY}). Moreover, {\it improved Strichartz estimates} such as \eqref{theorem_3_2}, \eqref{theorem_4_2} and \eqref{theorem_4_3} (in the free case) have played an important role in the study of higher-order and fractional NLS equations (see e.g. \cite{Pau}, \cite{Guo}, \cite{GLNY} and references therein). On the other hand, there are very few previous literatures on such improved Strichartz estimates for the case with potentials $V(x)$ (see a recent preprint \cite{GuNa} for the second order case). We therefore believe that Theorems \ref{theorem_3} and \ref{theorem_4}, as well as the method of their proofs, have many potential applications to such NLS type equations with potentials (see e.g. \cite{ZZ} and  \cite{KMVZZ} for related works on NLS equations with the Hardy potential). 
\vskip0.3cm
\begin{remark}
\label{remark_theorem_4_1}
We here make more specific comments on Theorems \ref{theorem_3} and \ref{theorem_4}.

\vskip0.2cm
 (i) If $\psi_0,F$ and $V$ are radially symmetric, then the Besov-type spaces $B[\mathcal L^{q_1}_rL^2_\omega]$ and $B[\mathcal L^{q_2'}_rL^2_\omega]$ in \eqref{theorem_4_2} can be replaced by $L^{q_1}$ and $L^{q_2'}$. Indeed, in such a case, the solution $\psi$ to \eqref{Cauchy_2} is also radially symmetric. Moreover, if $f$ is radially symmetric, then $\|f\|_{\mathcal L^p_rL^2_\omega}\sim \|f\|_{L^p}$, and hence $\norm{f}_{B[\mathcal L^{p}_rL^2_\omega]}\lesssim\norm{f}_{L^p}$ for $1<p\le2$ and $\norm{f}_{B[\mathcal L^{p}_rL^2_\omega]}\gtrsim\norm{f}_{L^p}$ for $2\le p<\infty$ by the standard square function estimates for the Littlewood--Paley decomposition.
\vskip0.2cm
(ii) All of \eqref{theorem_3_1}, \eqref{theorem_3_2}, \eqref{theorem_4_2} and \eqref{theorem_4_3} have the same scaling symmetry. Namely, if $(\tilde p,\tilde q)$ is $n/(2\sigma)$-admissible, $(p,q)$ is $n/2$-admissible, then all of the four norms:
 $$\norm{f}_{L^{\tilde p}_tL^{\tilde q}_x}, \ \norm{|D|^{{2(\sigma-1)}/{p}}f}_{L^p_tL^{q}_x},\ \norm{|D|^{s(2,q_1)}f}_{L^2_t L^{q_1}_x} \ {\rm and }\ \norm{|D|^{s(p_1,q_1)} f}_{L^{p_1}_tB[\mathcal L^{q_1}_rL^2_\omega]}$$ have the same scaling structure under the map $$f(t,x)\mapsto f_\lambda(t,x)=f(\lambda^{2\sigma} t,\lambda x),\ \ \lambda>0.$$
%
\vskip0.2cm
(iii) Let $\psi_0,F$ and $V$ be radially symmetric.
By virtue of \eqref{Sobolev_2}, (i) and (ii) in Remarks \ref{remark_theorem_4_1}, the estimates \eqref{theorem_4_2} and \eqref{theorem_4_3} imply the estimate \eqref{theorem_3_2} in the fractional case $\sigma<1$. Furthermore, {\it the interest of \eqref{theorem_4_2} and \eqref{theorem_4_3} is that if $\sigma>n/(2n-1)$ then one can choose exponents $p_1,p_2,q_1,q_2$ in Theorem \ref{theorem_4} in such a way that $s_1,s_2,s(2,q_1),s(2,q_2)>0$}. In particular, in such a case, \eqref{theorem_4_2} and \eqref{theorem_4_3}  imply the standard Strichartz estimates  \eqref{theorem_3_1} for $n/(2\sigma)$-admissible pairs (and radially symmetric  $\psi_0,F,V$). Therefore, Theorem \ref{theorem_4} shows that there is also an additional smoothing effect even for the fractional case if radial symmetry is assumed (or more generally, one takes the spherical average). Finally, {\it the estimate \eqref{theorem_4_4} provides a new result for the second order case $\sigma=1$} (under the radial symmetry), which is stronger than the usual endpoint Strichartz estimate with $s=0$.
\end{remark}

\vskip0.1cm
The next result is the uniform Sobolev estimates of Kenig--Ruiz--Sogge type \cite{KRS}.
\begin{theorem}	
\label{theorem_5}
Let $n\ge3$, $n/(n+1)\le \sigma<n/2$, $H=(-\Delta)^\sigma+V$, $V$ satisfy Assumption \ref{assumption_A} and $|x|^{\sigma}V\in L^{n/\sigma,\infty}(\R^n)$. Then, for any $2n/(n+2\sigma)\le p\le 2(n+1)/(n+3)$,
\begin{align}
\label{theorem_5_1}
\norm{(H-z)^{-1}f}_{L^{p'}}\lesssim |z|^{\frac{n}{\sigma}(\frac{1}{p}-\frac{1}{2})-1}\norm{f}_{L^{p}},\quad f\in L^p,\quad z\in \C\setminus[0,\infty).
\end{align}
In particular, the following uniform estimate of Sobolev type holds:
\begin{align}
\label{theorem_5_2}
\norm{u}_{L^{\frac{2n}{n-2\sigma}}}\lesssim \norm{(H-z)u}_{L^{\frac{2n}{n+2\sigma}}},\quad u\in C_0^\infty(\R^n\setminus\{0\}),\ z\in \C.
\end{align}
Moreover, if we assume in addition \eqref{corollary_1_1}, then \eqref{theorem_5_1} also holds for $z>0$ in which case $(H-z)^{-1}$ may be taken to be the outgoing or incoming resolvent $(H-z\mp i0)^{-1}$.
\end{theorem}
\vskip0.3cm
It is worth  noticing that the boundary resolvents  $(H-\lambda\mp i0)^{-1}$ are closely connected with the spectral density  $dE_H(\lambda)$ of $H$ by the following Stone formula:
$$dE_H(\lambda)=(2\pi i)^{-1}\big((H-\lambda +i0)^{-1}-(H-\lambda-i0)^{-1}\big).$$
Hence, as a consequence of \eqref{theorem_5_1}, the following spectral measure estimate holds:
\begin{align}
\label{spect-measure}
\norm{dE_H(\lambda)}_{L^{\frac{2(n+1)}{n+3}}\to L^{\frac{2(n+1)}{n-1}}}\lesssim \lambda^{\frac{n}{\sigma(n+1)}-1}, \quad \lambda>0,
\end{align}
which can be interestingly used to establish the $L^p$-bounds of Mikhlin-H\"ormander type for the spectral multiplier $\psi(H)$ (see e.g. \cite{HYZ}, \cite{SYY}). Moreover, we remark that the spectral measure estimate \eqref{spect-measure} is actually equivalent to the famous restriction estimates of Stein-Tomas in the free case $H=(-\Delta)^\sigma$. So  the estimate \eqref{spect-measure} is also called by {\it Restriction type estimate associated with $H$}. It is also seen from this equivalence that the condition $\sigma\ge n/(n+1)$ is optimal in the sense that \eqref{theorem_5_1} fails if $\sigma<n/(n+1)$ (see \cite{HYZ}).
\vskip0.3cm
\begin{remark}
\label{remark_theorem_5}
When $\sigma\ge1$, the following estimate also holds:
\begin{align}
\label{theorem_5_3}
\norm{|D|^{\sigma-1}u}_{L^{\frac{2n}{n-2}}}\lesssim \norm{|D|^{1-\sigma}(H-z)u}_{L^{\frac{2n}{n+2}}},\quad z\in \C.
\end{align}
By virtue of \eqref{Sobolev_2}, this estimate is stronger than \eqref{theorem_5_2} if $\sigma>1$.
\end{remark}

\subsection{Further comments}
We here provides several further remarks on all theorems above, especially the optimality of the results for the operator $(-\Delta)^{\sigma}+a|x|^{-2\sigma}$.

\begin{remark}[{{\it The second order case $\sigma=1$}}]As explained above, these theorems extend the results for $\sigma=1$ proved by \cite{BPST1,BPST2} and \cite{BoMi} to the higher-order and fractional cases. Moreover, radial-improved Strichartz estimates \eqref{theorem_4_4} are new even for the case $\sigma=1$.
\end{remark}

\begin{remark}[{{\it The condition $1/2<\sigma<n/2$\,}}]
The condition $\sigma<n/2$ is mainly due to the following two points:

Firstly, in the definition of $H$ as well as the proof of the main theorems, we frequently use Hardy's inequality \eqref{Hardy} or Sobolev's inequality of the form
$$\norm{f}_{L^{\frac{2n}{n-2\sigma},2}(\R^n)}\lesssim \norm{|D|^\sigma f}_{L^2(\R^n)},
$$ both of which hold for all $f\in C_0^\infty(\R^n)$ if and only if $0\le\sigma<n/2$ (note that $|x|^{-2\sigma}\notin L^1_{\loc}(\R^n)$ if $\sigma\ge n/2$). If $\sigma\ge n/2$, one can still define $H$, at least in case of $V=a|x|^{-2\sigma}$ with $a>0$, as the Friedrichs extension of $Q_H$ defined on $C_0^\infty(\R^n\setminus\{0\})$. However, in such a case, the form domain of $H$ is strictly larger than $\H^{\sigma}$. Moreover, the self-adjoint extension of $(-\Delta)^\sigma|_{C_0^\infty(\R^n\setminus\{0\})}$ is possibly different from the usual one with domain $\H^{2\sigma}$ (see \cite[Theorem X.11]{ReSi}). These actually cause several crucial difficulties when considering the problems with Hardy potential if $\sigma\ge n/2$.
\vskip0.1cm
Secondly, if $\sigma\ge n/2$, none of the uniform resolvent estimate \eqref{theorem_1_2} with $\gamma=0$, Kato smoothing estimate \eqref{theorem_2_2} with $\gamma=0$ and the endpoint Strichartz estimates \eqref{theorem_3_1} with $p$ or $\tilde p=2$ holds even for the free case $H=H_0$. At a technical level, due to this absence of the estimates for $H_0$, the perturbation argument used in Sections \ref{section_3} and \ref{section_4} below breaks down. On the other hand, the condition $\sigma>1/2$ for Strichartz and uniform Sobolev estimates is due to the use of \eqref{theorem_2_2} in the proofs. Although it is a very important problem in view of applications to nonlinear equations, the validity of Strichartz estimates for the relativistic Schr\"odinger operator with Coulomb potentials $|D|+a|x|^{-1}$ (or more interestingly, Strichartz estimates for the Dirac operator $-i\alpha\cdot\nabla+a|x|^{-1}$) still remains open even if $a$ is sufficiently small (see \cite{EGG} and references therein for the case with short-range potentials $V(x)=O(\<x\>^{-1-\ep})$).
\end{remark}
\begin{remark}[{{\it Optimality\,}}]
\label{remark_theorems}Here we discuss the optimality of the several conditions on potentials $V$ in Theorems \ref{theorem_3} and \ref{theorem_4} in case of the Hardy potential.
\vskip0.2cm
(i) {\it The decay rate}. For the case $\sigma=1$, \cite{GVV} found a class of repulsive potentials, which  decay slower than $|x|^{-2}$ and are non-negative, but neither positive everywhere nor radially symmetric, such that Strichartz estimates cannot hold except for the trivial $L^\infty_tL^2_x$ estimate. On the other hand, for higher-order Schr\"odinger operators $H=(-\Delta)^m+V$ with  $m\in \N$, $m\ge2$ and $V=O(\<x\>^{-\beta})$ with some large $\beta>2m$, \cite{FSWY} has proved Strichartz estimates provided that $H$ has neither non-negative eigenvalues nor zero resonance (see \cite{RoSc} for the case $m=1$). We also refer to our subsequent work \cite{MiYa2} which will establish basically the same results as in the present work for the case with $m\in \N$, $m\ge2$ and arbitrarily $\beta>2m$. It follows from these observations that the decay rate of our potential $V$, say $V=O(\<x\>^{-2\sigma})$, is essentially critical for the validity of Strichartz estimates.
\vskip0.2cm
(ii) {\it The singularity and the lower bound of coupling constant}. The singularity of $V$ at the origin and the condition \eqref{assumption_A_1} are also critical as follows. Let us consider the case $H=(-\Delta)^\sigma+a|x|^{-\gamma}$ for simplicity. On one hand, if either $a<0$ and $\gamma>2\sigma$ or $a<-C_{\sigma,n}$ and $\gamma=2\sigma$, then due to the optimality of Hardy's inequality \eqref{Hardy}, $H$ is not bounded from below and any its self-adjoint extension may have infinitely many (possibly embedded) eigenvalues which prevents any kind of global estimates (except for the conservation laws). On the other hand, for $\sigma=1$ and $V=-C_{1,n}|x|^{-2}$, it is known (see \cite[Remark 3.6]{Mizutani_JST}) that \eqref{theorem_1_2} cannot hold even if the weight $|x|^{-1}$ is replaced by $\chi\in C_0^\infty(\R^n)$. In such a case, the endpoint Strichartz estimate \eqref{theorem_3_1} with $p$ or $\tilde p=2$ can also fail as shown by \cite{Mizutani_JDE}. In Section \ref{section_6} below, we will discuss the operator $(-\Delta)^\sigma-C_{\sigma,n}|x|^{-\gamma}$ in more details.
\end{remark}
\begin{remark}[{{\it Some open problems\,}}]
As a consequence of above remarks, the Hardy potential $a|x|^{-2\sigma}$ with $a>-C_{\sigma,n}$ is critical for the validity of the above theorems, especially Theorems \ref{theorem_3}, \ref{theorem_4} and \ref{theorem_5}, in terms of the decay rate at infinity, the singularity at the origin and the lower bound of the coupling constant $a$. We however note that there is still a hope to obtain some of the above results, even in the case when $V=-C_{\sigma,n}|x|^{-2\sigma}$ or $V$ is slowly decaying. On one hand, \cite{Mizutani_JDE} proved a weak-type endpoint Strichartz estimate for $H=-\Delta-C_{1,n}|x|^{-2}$, which particularly implies non-endpoint Strichartz estimates with $p,\tilde p>2$ (see Section \ref{section_6} below for more details on the critical case $a=-C_{\sigma,n}$). On the other hand, \cite{Mizutani_JFA} recently showed that  the operator $-\Delta+\<x\>^{-\mu}$ satisfies Strichartz estimates for all $\mu\in (0,2)$ (including the long-range case $0<\mu<1$). It would be interesting to investigate if similar results hold for the higher-order or fractional cases. The  validity of Strichartz estimates for $H=(-\Delta)^{\sigma}+a|x|^{-2\sigma}$ with $0<\sigma\le1/2$ and $a\neq0$, which is completely open, would be also an interesting and important problem.
\end{remark}

\subsection{Outline of the paper}
Here we outline the ideas of proofs of the above theorems, as well as describe the organization of the rest of the paper.
\vskip0.3cm
{\it Section \ref{section_2}} deals with uniform resolvent and Kato smoothing estimates.
The proof of Theorem \ref{theorem_1} is based on a version of Mourre's theory which is similar to the argument used by Hoshiro \cite{Hos} (see also the original work by Mourre \cite{Mou}). We compute the first and second commutators of $H$ with $iA=(x\cdot\nabla+\nabla\cdot x)/2$, which are given by
\begin{align*}
S_1:=[H,iA]=H_1+V_1,\quad
S_2:=[[H,iA],iA]=H_2+V_2.
\end{align*}
By \eqref{assumption_A_2} and \eqref{assumption_A_3}, the following estimates (in the sense of forms) hold:
$$
S_1\ge S_1^{1/2}S_1^{1/2}\gtrsim |D|^\sigma |D|^\sigma\ge0,\quad -S_1\lesssim S_2\lesssim S_1,
$$
without any spectral localization or compact error term. Then, for $\delta>0$ and $1/2<s<1$, we apply the differential inequality technique by Mourre \cite{Mou,Mourre_CMP} to the following operator
$$
F_\ep(z)=\<A\>^{-s}\<\ep A\>^{-1+s}|D|^\sigma e^{-\delta|D|}(H-z-i\ep S_1)^{-1}e^{-\delta|D|} |D|^\sigma \<\ep A\>^{-1+s}\<A\>^{-s},\ \ep\in (0,1],
$$
and obtain the uniform boundedness of $\norm{F_\ep(z)}$ in $\ep$, $\delta$ and $z$. As $\ep\searrow0$, one has
$$
\sup_{z\in \C\setminus\R,\,\delta>0}\norm{\<A\>^{-s}|D|^\sigma e^{-\delta|D|}(H-z)^{-1}e^{-\delta|D|} |D|^\sigma \<A\>^{-s}}_{L^2\to L^2}<\infty.
$$
When $\sigma-n/2<\gamma\le \sigma-1$, it follows from Hardy's inequality \eqref{Hardy} that $|x|^{-\sigma+\gamma}|D|^{\gamma-\sigma}\<A\>$ is bounded on $L^2$. Moreover, if $\sigma-1<\gamma<\sigma-1/2$, one can also show by \eqref{Hardy} and an interpolation technique that $|x|^{-\sigma+\gamma}|D|^{\gamma-\sigma}\<A\>^{\sigma-\gamma}$  is bounded on $L^2$. Therefore, in both cases, the weight $\<A\>^{-s}|D|^\sigma$ can be replaced by $|x|^{-\sigma+\gamma}|D|^{\gamma}$. Taking the limit $\delta\searrow 0$, we therefore obtain the desired estimate  \eqref{theorem_1_2}.
\vskip0.1cm
Regarding Corollary \ref{corollary_1}, the existence of $\<A\>^{-s}(H-\lambda\mp i0)^{-1}\<A\>^{-s}$ is a direct consequence of Assumption \ref{assumption_A}, \eqref{corollary_1_1} and the original Mourre theory, while its uniform estimate \eqref{corollary_1_2} follows from Theorem \ref{theorem_1} and a limiting argument.
\vskip0.1cm
Once Theorem \ref{theorem_1} is verified, Theorems \ref{theorem_2} follows from an abstract smooth perturbation theory by Kato \cite{Kat}. Moreover, \eqref{theorem_1_1} is in fact equivalent to \eqref{theorem_2_1}.
\vskip0.3cm
{\it Section \ref{section_3}} is devoted to studying Strichartz estimates. The proof of Theorem \ref{theorem_3} relies on a perturbation method by Rodnianski--Schlag \cite{RoSc} (see also  \cite{BPST2}, \cite{BoMi}). Let $U_H,\Gamma_H$ be homogeneous and inhomogeneous Schr\"odinger propagators defined by
\begin{align*}
U_Hf=e^{-itH}f,\ \Gamma_HF=\int_0^t e^{-i(t-s)H}F(s)ds,
\end{align*}
which satisfy the following Duhamel formulas
$$
U_H=U_{H_0}+i\Gamma_{H_0}VU_H,\quad \Gamma_H=\Gamma_{H_0}+i\Gamma_{H_0}V\Gamma_{H}=\Gamma_{H_0}+i\Gamma_{H}V\Gamma_{H_0}.
$$
Thanks to these formulas, O'Neil's inequality \eqref{Holder}  and Sobolev's inequality \eqref{Sobolev_2}, \eqref{theorem_3_2} in Theorem \ref{theorem_3} follows from \eqref{theorem_1_2} and the same Strichartz estimates as \eqref{theorem_3_2} for $U_{H_0}$ and $\Gamma_{H_0}$. In particular, the following estimate will play an essential role:
\begin{align}
\label{outline_1}
\norm{|D|^{2(\sigma-1)/p}\Gamma_{H_0}|x|^\sigma VF}_{L^p_tL^{q,2}_x}\lesssim |\norm{|x|^\sigma VF}_{L^2_tL^{\frac{2n}{n-2\sigma},2}_x}\lesssim \norm{|x|^\sigma V}_{L^{\frac n\sigma,\infty}}\norm{F}_{L^2_tL^{2}_x}.
\end{align}
\vskip0.1cm
As for Theorem \ref{theorem_4}, since Sobolev's inequality \eqref{Sobolev_2} cannot hold if $\sigma<1$, we do not know if \eqref{outline_1} holds. Instead, with the aid of Christ--Kiselev's lemma \cite{ChKi} (see Appendix \ref{appendix_misc}), we use \eqref{theorem_2_2} with $V\equiv0$, Strichartz estimates for $U_{H_0}$ and its dual estimate to obtain
\begin{align}
\label{outline_2}
\norm{|D|^{2(\sigma-1)/p}\Gamma_{H_0}|x|^\sigma VF}_{L^p_tL^{q}_x}\lesssim \norm{|x|^{2\sigma} V}_{L^{\infty}}\norm{F}_{L^2_tL^{2}_x}
\end{align}
for all non-endpoint $n/2$-admissible pair $(p,q)$. Then the first half of Theorem \ref{theorem_4} can be verified by using \eqref{outline_2} and a similar perturbation method as for Theorem \ref{theorem_3}.  The last half of Theorem \ref{theorem_4} is also obtained by using the same argument as for Theorem \ref{theorem_3} and improved Strichartz estimates for $U_{H_0}$ and $\Gamma_{H_0}$ instead of \eqref{outline_1} or \eqref{theorem_3_2}.
\vskip0.3cm
{\it Section \ref{section_4}} deals with uniform Sobolev estimates. Let $R_{H}(z)=(H-z)^{-1}$. With Theorem \ref{theorem_1} at hand, we can see that \eqref{theorem_5_1} follows from uniform Sobolev estimates for the free resolvent $R_{H_0}(z)$ obtained by \cite{HYZ}, via the second resolvent equations
$$
R_H(z)=R_{H_0}(z)-R_{H_0}(z)VR_H(z)=R_{H_0}(z)-R_{H}(z)VR_{H_0}(z).
$$
\eqref{theorem_5_2} is an immediate consequence of \eqref{theorem_5_1} and Corollary \ref{corollary_1}. Moreover, if $\sigma>1$, one can obtain \eqref{theorem_5_3} (which are stronger than \eqref{theorem_5_2} if $\sigma>1$) by plugging the function $u=e^{izt}f$ to the double endpoint Strichartz estimate \eqref{theorem_3_2} with $p=\tilde p=2$.
\vskip0.3cm
{\it Section \ref{section_5}}  is devoted to a generalization of the above theorems to two kinds of dispersive operators. The first one is the operator $H=P_0(D)+V$ with a class of inhomogeneous elliptic operators of the form $P_0(D)=\sum_{j=1}^J (-\Delta)^{\sigma_j}$. The second one is the Schr\"odinger operator with variable coefficients of the form $H=-\nabla \cdot G_0(x)\nabla +V(x)$ with $G_0$ being a small long-range perturbation of the identity matrix.
\vskip0.3cm
{\it Section \ref{section_6}} is concerned with the critical constant case $H=(-\Delta)^\sigma-C_{\sigma,n}|x|^{-2\sigma}$. We show that the same results as in the subcritical case $a>-C_{\sigma,n}$ still hold for functions orthogonal to all radial functions. For the radial case, we also prove some new results for a critical operator related with $(-\Delta)^\sigma-C_{\sigma,n}|x|^{-2\sigma}$ and discuss a few known results and open problems for $(-\Delta)^\sigma-C_{\sigma,n}|x|^{-2\sigma}$.
\vskip0.3cm
{\it Appendices} consist of the following contents: the proof of Strichartz estimates for the free evolution (Appendix \ref{appendix_A}); the proof of Example \ref{example_1} (Appendix \ref{appendix_example}); several supplementary materials from Harmonic Analysis used in the paper, including Lorentz spaces, real interpolation spaces, Sobolev's inequality and Christ--Kiselev's lemma (Appendix \ref{appendix_misc}).

\vskip0.3cm
\section{Uniform resolvent and Kato smoothing estimates}
\label{section_2}
In this section,  we prove Theorem \ref{theorem_1}, Corollary \ref{corollary_1} and Theorem \ref{theorem_2}. To these purposes, we set
\begin{align}
\label{dilation}
A =\frac{1}{2i}(x\cdot\nabla+\nabla\cdot x)=\frac{1}{i}x\cdot\nabla +\frac{n}{2i},
\end{align}
which is the self-adjoint generator of the dilation unitary group $e^{itA}f(x)=e^{nt/2}f(e^tx)$ on $L^2(\R^n)$. Then we have the following useful lemma.

\begin{lemma}
\label{lemma_section_2_1}
Suppose $\varphi$ and $\xi\cdot \nabla_\xi\varphi$ belong to $L^1_{\loc}(\R^n)\cap \S'(\R^n)$ and set $\tilde \varphi(\xi)=\xi\cdot \nabla_\xi\varphi$. Then one has $$[\varphi(D),iA]f=\tilde \varphi(D)f,\quad f\in C_0^\infty(\R^n).$$
\end{lemma}

\begin{proof}
For any $f\in C_0^\infty(\R^n)$ and $t\in \R$, a direct computation yields
\begin{align}
\label{lemma_section_2_1_proof_1}
\frac{d}{dt}\left(e^{-itA}\varphi(D)e^{itA}f(x)\right)=\frac{1}{(2\pi)^{n/2}}\int e^{ix\cdot \xi}\frac{\partial}{\partial t}\varphi(e^t\xi)\hat f(\xi)d\xi=\tilde \varphi(e^tD)f(x)
\end{align}
and the assertion follows by taking $t=0$.
\end{proof}

Note that \eqref{lemma_section_2_1_proof_1} with $\varphi(\xi)=\<\xi\>^s$ implies that $e^{itA}\H^s\subset \H^s$ for any $s\in \R$.
\vskip0.3cm
The following theorem is the main ingredient for the proof of Theorem \ref{theorem_1}.
\begin{theorem}	
\label{theorem_section_2_1}
Let $n\ge1$, $0<\sigma<n/2$, $s>1/2$ and $H=(-\Delta)^\sigma+V$ satisfy Assumption A. Then the following estimate holds:
\begin{align}
\label{theorem_section_2_1_1}
\sup_{\delta>0}\sup_{z\in \C\setminus[0,\infty)}\norm{\<A\>^{-s}|D|^\sigma e^{-\delta|D|} (H-z)^{-1}e^{-\delta|D|}|D|^\sigma  \<A\>^{-s}}_{L^2\to L^2}<\infty.
\end{align}
\end{theorem}
\vskip0.2cm

Theorem \ref{theorem_section_2_1} (also Theorem \ref{theorem_section_5_1} in Section \ref{section_5}) extends the result of Hoshiro \cite{Hos} for the free case $V\equiv0$, to the cases with $V\not\equiv 0$. Note that, thanks to the factor $e^{-\delta|D|}$, the operator $\<A\>^{-s}|D|^\sigma e^{-\delta|D|} (H-z)^{-1}e^{-\delta|D|}|D|^\sigma  \<A\>^{-s}$ is clearly well-defined as a composition of bounded operators on $L^2$ for all $\delta>0$ and $z\in \C\setminus[0,\infty)$, while it is not clear whether $\<A\>^{-s}|D|^\sigma (H-z)^{-1}|D|^\sigma \<A\>^{-s}f$ is well-defined even for $f\in C_0^\infty(\R^n)$ since we do not know if $\<A\>^{-s}:C_0^\infty(\R^n)\to  \H^\sigma$ or not. This is the reason to add the factor $e^{-\delta|D|}$ in \eqref{theorem_section_2_1_1}.

\vskip0.3cm
Before proving Theorem \ref{theorem_section_2_1}, we come to show how it implies Theorem \ref{theorem_1}. Throughout this section, we use the following standard shorthand notations:
$$
\|f\|:=\|f\|_{L^2},\quad \|L\|:=\|L\|_{L^2\to L^2}
$$
for a function $f\in L^2(\R^n)$ and an operator $L\in \mathbb B(L^2)$.
\vskip0.2cm
\begin{proof}[\underline{Proof of Theorem \ref{theorem_1}}]
Let $f,g\in C_0^\infty(\R^n\setminus\{0\})$. Then $|D|^\gamma |x|^{-\sigma+\gamma}f$ and $|D|^\gamma |x|^{-\sigma+\gamma}g$ belong to $L^2$. Moreover, for each $z\in \C\setminus[0,\infty)$, the quantity
$$
\<(H-z)^{-1}e^{-\delta |D|}|D|^\gamma |x|^{-\sigma+\gamma}f,e^{-\delta |D|}|D|^\gamma |x|^{-\sigma+\gamma}g\>
$$
converges to $\<(H-z)^{-1}|D|^\gamma |x|^{-\sigma+\gamma}f,|D|^\gamma |x|^{-\sigma+\gamma}g\>$ as $\delta\searrow0$. Hence, if the estimate
\begin{align}
\label{theorem_1_proof_0}
|\<(H-z)^{-1}e^{-\delta |D|}|D|^\gamma |x|^{-\sigma+\gamma}f,e^{-\delta |D|}|D|^\gamma |x|^{-\sigma+\gamma}g\>|\lesssim \|f\|\,\|g\|
\end{align}
holds uniformly in $\delta>0$ and $z\in \C\setminus[0,\infty)$, then letting $\delta\searrow 0$ in \eqref{theorem_1_proof_0} and using the duality and density arguments, we obtain the desired bound \eqref{theorem_1_1}.

To obtain \eqref{theorem_1_proof_0}, we first consider the case $\sigma-n/2<\gamma\le \sigma-1$ and write
$$
e^{-\delta |D|}|D|^{\gamma}|x|^{-\sigma+\gamma}=e^{-\delta |D|}|D|^\sigma  \<A\>^{-1}\cdot\<A\>(A+i)^{-1}\cdot (A+i)|D|^{-\sigma+\gamma}|x|^{-\sigma+\gamma}.
$$
Since $|\xi|^{-\sigma+\gamma}\in L^1_{\loc}(\R^n)$ if $\sigma-\gamma<n/2$, we know by Lemma \ref{lemma_section_2_1} and \eqref{dilation} that
\begin{align*}
|D|^{-\sigma+\gamma}(A-i)f
&=A|D|^{-\sigma+\gamma}f+[|D|^{-\sigma+\gamma},A]f-i|D|^{-\sigma+\gamma}f\\
&=-ix\cdot\nabla |D|^{-\sigma+\gamma} f-i\left(\frac n2-\gamma-1\right)|D|^{-\sigma+\gamma}f
\end{align*}
Hardy's inequality \eqref{Hardy} then implies
$$
\norm{|x|^{-\sigma+\gamma}|D|^{-\sigma+\gamma}(A-i)f}\lesssim \norm{|x|^{-\sigma+\gamma+1}\nabla|D|^{-\sigma+\gamma}f}+\norm{|x|^{-\sigma+\gamma}|D|^{-\sigma+\gamma}f}
\lesssim \norm{f},
$$
where the condition $\sigma-n/2<\gamma\le \sigma-1$ was used  to ensure $-\sigma+\gamma+1\le 0$ and $|x|^{-\sigma+\gamma}\in L^1_{\loc}(\R^n)$. Thus $|x|^{-\sigma+\gamma}|D|^{-\sigma+\gamma}(A-i)$ is bounded on $L^2$ and so is $(A+i)|D|^{-\sigma+\gamma}|x|^{-\sigma+\gamma}$ by duality. This fact, Theorem \ref{theorem_section_2_1} with $s=1$ and the bound $\|\<A\>(A+i)^{-1}\|\lesssim1$ imply
\begin{align*}
&\sup_{\delta>0}\sup_{z\in \C\setminus[0,\infty)}|\<(H-z)^{-1}e^{-\delta |D|}|D|^{\gamma}|x|^{-\sigma+\gamma}f,e^{-\delta |D|}|D|^{\gamma}|x|^{-\sigma+\gamma}g\>|\\
&\lesssim \norm{(A+i)|D|^{-\sigma+\gamma}|x|^{-\sigma+\gamma}f}\,\norm{(A+i)|D|^{-\sigma+\gamma}|x|^{-\sigma+\gamma}g}\\
&\lesssim \norm{f}\,\norm{g}
\end{align*}
and hence the bound \eqref{theorem_1_proof_0} follows for $\sigma-n/2<\gamma\le \sigma-1$.
\vskip0.2cm
Next, for the case $\sigma-1<\gamma<\sigma-1/2$, setting $s:=\sigma-\gamma\in (1/2,1)$, we compute
$$
e^{-\delta |D|}|D|^{\gamma}|x|^{-\sigma+\gamma}=e^{-\delta |D|}|D|^\sigma\<A\>^{-s}\cdot \<A\>^s|D|^{-s}|x|^{-s}.
$$
By the same argument as above, it suffices to show that
\begin{align}
\label{theorem_1_proof_1}
\norm{|x|^{-s}|D|^{-s}\<A\>^s f}\lesssim \norm{f}.
\end{align}
To this end, we shall apply Stein's interpolation theorem \cite{Ste} to the operator
$$
T_\ep(z)=(|x|+\ep)^{-z}(|D|+\ep)^{-z}\<A\>^z,\quad \ep>0,\ 0\le \Re z\le1.
$$
For each $f,g\in C_0^\infty(\R^n)$, the complex function
$
\varphi_\ep(z)=\<T_\ep(z) f,g\>
$ is clearly bounded on $\{z\in \C\ |\ 0\le \Re z\le 1\}$ and analytic in $\{z\in \C\ |\ 0<\Re z<1\}$. Moreover, $T_\ep(iy)$ satisfies
\begin{align}
\label{theorem_1_proof_2}
\|T_\ep(iy)f\|\le\norm{f}
\end{align} for all $\ep>0,y\in \R$. To deal with the operator $T_\ep(1+iy)$, we see from Lemma \ref{lemma_section_2_1} that
$$
(|D|+\ep)^{-1-iy}A=A(|D|+\ep)^{-1-iy}+i(1+iy) |D|(|D|+\ep)^{-2-iy}.
$$
This, combined with Hardy's inequality \eqref{Hardy}, implies
\begin{align*}
&\|(|x|+\ep)^{-1-iy}(|D|+\ep)^{-1-iy}(A+i)f\|\\
&\lesssim \| |x|^{-1}x\cdot\nabla (|D|+\ep)^{-1-iy}f\|+\| |x|^{-1}(|D|+\ep)^{-1-iy}f\|\\&\quad+\<y\>\| |x|^{-1}|D|(|D|+\ep)^{-2-iy}f\|\\
&\lesssim \||D|(|D|+\ep)^{-1-iy}f\|+\<y\>\||D|^2(|D|+\ep)^{-2-iy}f\|\\
&\lesssim\<y\>\|f\|
\end{align*}
uniformly in $\ep>0$. Hence there exists $C>0$ independent of $\ep>0$ and $y\in \R$ such that\begin{align}
\label{theorem_1_proof_3}
\norm{T_\ep(1+iy)f}
\lesssim \<y\>\norm{(A+i)^{-1}\<A\>^{1+iy}f}
\le C\<y\>\norm{f}.
\end{align}
By \eqref{theorem_1_proof_2} and \eqref{theorem_1_proof_3}, we can apply Stein's interpolation theorem \cite[Theorem 1]{Ste} (with the choice of $A_0(y)=1$ and $A_1(y)=C\<y\>$ in this theorem) to obtain that,  for any $0<s<1$, $$\|T_\ep(s)f\|\lesssim \|f\|$$
uniformly in $\ep>0$. Since $(|x|+\ep)^{-s}(|D|+\ep)^{-s}\<A\>^sf$ converges to $|x|^{-s}|D|^{-s}\<A\>^{s}f$ as $\ep\searrow0$ in the distribution sense, by letting $\ep\searrow0$ in the estimate  $$
|\<T_\ep(s)f,g\>|\lesssim\norm{f}\,\norm{g}
$$
and using the duality and density arguments, we obtain the estimate \eqref{theorem_1_proof_1}. This implies the bound \eqref{theorem_1_proof_0} for $\sigma-1<\gamma<\sigma-1/2$ and we complete the proof of Theorem \ref{theorem_1}.
\end{proof}
\vskip0.3cm
{Now we turn to the proof of  Theorem \ref{theorem_section_2_1}.}
To this end,  we first give the rigorous definitions of the unbounded commutators $[H,iA]$ and $[[H,iA],iA]$ by the sesquilinear forms.
Let $u,v\in C_0^\infty(\R^n)$. Recall that $$Q_H(u,v):=\<H_0u,v\>+\<Vu,v\>$$ and that $H_\ell,V_\ell$ were given by \eqref{symbol_ell}. Consider the following two sesquilinear forms
\begin{align*}
Q_{S_1}(u,v):=Q_H(iAu,v)+Q_H(u,iAv),\quad
Q_{S_2}(u,v):=Q_{S_1}(iAu,v)+Q_{S_1}(u,iAv).
\end{align*}
Recalling that $P_0(\xi)=|\xi|^{2\sigma}$ and $P_\ell(\xi)=(\xi\cdot\nabla)^\ell P_0(\xi)$, we find by Lemma \ref{lemma_section_2_1} that
\begin{align*}
[H_0,iA]=P_1(D)=H_1,\quad [[H_0,iA],iA]=P_2(D)=H_2
\end{align*}
as sesquilinear forms on $C_0^\infty(\R^n)$. Hence we have
$$
Q_{S_1}(u,v)=\<[H,iA]u,v\>=\<(H_1+V_1)u,v\>,
$$
$$
Q_{S_2}(u,v)=\<[[H,iA],iA]u,v\>=\<(H_2+V_2)u,v\>,
$$
on $C_0^\infty(\R^n)$, where $[V,iA]=V_1$ and $[V_1,iA]=V_2$.  By virtue of these two formulas and Assumption \ref{assumption_A}, $Q_{S_1}$ and $Q_{S_2}$ satisfy
\begin{align}
\label{lemma_section_2_2_0}
|Q_{S_1}(u,v)|\lesssim \norm{u}_{\H^\sigma }\norm{v}_{\H^\sigma },\quad \||D|^\sigma u\|^2+|Q_{S_2}(u,u)|&\lesssim Q_{S_1}(u,u).
\end{align}
These estimates allow us extending $Q_{S_1}$ to a continuous positive sesquilinear form on $\H^\sigma $ for which we use the same symbol $Q_{S_1}$.  Similarly, so do for $Q_{S_2}$.

Let $S_1$ be a unique positive self-adjoint operator generated by the closed form $Q_{S_1}$ such that $Q_{S_1}(u,v)=\<S_1u,v\>$. Clearly,  $S_1$ is the self-adjoint extension of the commutator $[H,iA]$ on $C_0^\infty(\R^n)$. Its square root $S_1^{1/2}$ with domain $D(S_1^{1/2})=\H^\sigma $ can be defined via the spectral theorem. $Q_{S_2}$ also generates a self-adjoint operator $S_2$ with form domain $\H^\sigma$ satisfying $Q_{S_2}(u,v)=\<S_2u,v\>$. Similarly, $S_2$ is the self-adjoint extension of the double commutator $[[H,iA],iA]$ on $C_0^\infty(\R^n)$.  Moreover, we have for $u\in \H^\sigma$,
\begin{align}
\label{lemma_section_2_2_1}
\norm{|D|^\sigma u}\lesssim \norm{S_1^{1/2}u}\lesssim \norm{u}_{\H^\sigma },\quad
|\<S_2u,u\>|\lesssim \|S_1^{1/2}u\|^2.
\end{align}
In particular $S^{1/2}_1$ is positive definite.

We also define operators $\tilde H,\tilde S_1,\tilde S_2:\H^{\sigma }\to \H^{-\sigma }$ by
$$
\tilde Hu:=Q_H(u,\cdot),\quad \tilde S_1u:=Q_{S_1}(u,\cdot),\quad \tilde S_2u:=Q_{S_2}(u,\cdot).
$$
These are extensions of $H,S_1,S_2$, namely $H\subset \tilde H$, $S_1\subset\tilde S_1$ and $S_2\subset \tilde S_2$.
\vskip0.3cm
Before starting the proof of Theorem \ref{theorem_section_2_1}, we need one more lemma.
\begin{lemma}
\label{lemma_section_2_3}
Define $W_\ep:=\<A\>^{-s}\<\ep A\>^{s-1}$ for $0<s\le1$ and $0\le  \ep\le1$. Then $W_\ep\in C([0,1]_\ep;\mathbb B(L^2))\cap C^1((0,1]_\ep;\mathbb B(L^2))$. Moreover, the following estimates hold:
$$
\norm{W_\ep}\le 1,\quad
\norm{AW_\ep}\le \ep^{s-1},\quad
\norm{W_\ep'}\le (1-s) \ep^{s-1},
$$
where $ W_\ep'=\frac{d}{d\ep}W_\ep=(s-1)\ep\<A\>^{-s}|A|^2\<\ep A\>^{s-3}$.
\end{lemma}

\begin{proof}
By direct computations, the following estimates hold for $\lambda\in \R$ and $0<\ep\le1$:
\begin{align*}
\<\lambda\>^{-s}\<\ep \lambda\>^{s-1}&\le1,\\
|\lambda|\<\lambda\>^{-s}\<\ep \lambda\>^{s-1}&=|\lambda|^s\<\lambda\>^{-s}(\ep|\lambda|)^{-s+1}\<\ep \lambda\>^{s-1}\ep^{s-1}\le \ep^{s-1},\\
\ep\<\lambda\>^{-s}|\lambda|^2\<\ep \lambda\>^{s-3}&=|\lambda|^s\<\lambda\>^{-s}(\ep|\lambda|)^{-s+2}\<\ep \lambda\>^{s-3}\ep^{s-1}\le \ep^{s-1}.
\end{align*}
Since $\<\lambda\>^{-s}\<\ep \lambda\>^{s-1}\in C([0,1]_\ep)\cap C^1((0,1]_\ep)$ for all $\lambda\in \R$ and $A$ is self-adjoint, the desired result follows from these three estimates and the spectral decomposition theorem.
\end{proof}
\vskip0.3cm
\begin{proof}[\underline{Proof of Theorem \ref{theorem_section_2_1}}]
The proof is based on a version of Mourre's theory by Hoshiro \cite{Hos}. We may assume $1/2<s<1$ and $\Im z>0$. In what follows, we use the notation $\eta_\delta:=|D|^\sigma e^{-\delta|D|}$ for short. For $0\le \ep\le 1$ and $\delta>0$, we set
\begin{align*}
G_\ep=G_\ep(z):=(\tilde H-z-i\ep \tilde S_1)^{-1},\ F_\ep=F_\ep(z):=W_\ep\eta_\delta G_\ep(z) \eta_\delta W_\ep,
\end{align*}
Then the estimate \eqref{theorem_section_2_1_1} will be obtained by showing that $F_\ep$ is bounded on $L^2$ uniformly in $\ep,z$ and $\delta$. The proof is decomposed into two steps.
\vskip0.2cm
{\it Step 1}. We first check that $F_\ep$ is a well-defined bounded operator on $L^2$. To this end, we consider the operator $\tilde H-z-i\ep \tilde S_1:\H^\sigma\to\H^{-\sigma}$ which satisfies, for $f\in \H^\sigma$,
\begin{align*}
X:=\Re\<(\tilde H-z-i\ep \tilde S_1)f,f\>&=Q_H(f,f)-\Re z\|f\|^2,\\
Y:=-\Im \<(\tilde H-z-i\ep \tilde S_1)f,f\>&=\Im z\|f\|^2+\ep Q_{S_1}(f,f)\ge \Im z\|f\|^2.
\end{align*}
By plugging the second estimate into the first one, we have
$$
Q_H(f,f)\le X+(\Im z)^{-1}{|\Re z|}Y\lesssim \<z\>|\Im z|^{-1}|\<(\tilde H-z-i\ep \tilde S_1)f,f\>|,
$$
which, together with \eqref{assumption_A_4}, implies the following coercivity estimate:
\begin{align}
\label{coercivity}
 \|f\|_{\H^\sigma}^2\lesssim \<z\>|\Im z|^{-1}|\<(\tilde H-z-i\ep \tilde S_1)f,f\>|.
\end{align}
The inverse $G_\ep(z):\H^{-\sigma}\to \H^\sigma$ of $\tilde H-z-i\ep \tilde S_1$ thus exists and satisfies
\begin{align}\label{proof_theorem_section_2_1_1}
\sup_{\ep\in [0,1]}\norm{G_\ep(z)}_{\H^{-\sigma}\to \H^\sigma}\lesssim \<z\>|\Im z|^{-1}
\end{align} by the Lax--Milgram theorem. Similarly, $G_{-\ep}(\overline z)$ exists and satisfies the same estimate as \eqref{proof_theorem_section_2_1_1}. Since $W_\ep\eta_\delta,\eta_\delta W_\ep \in \mathbb B(L^2)$, $F_\ep=W_\ep\eta_\delta G_\ep \eta_\delta W_\ep$ thus is bounded on $L^2$.
\vskip0.1cm
Next, we claim that the following uniform estimate  (in $\ep,z$ and $\delta$) holds:
\begin{align}
\label{proof_theorem_section_2_1_2}
\sup\{\norm{F_\ep(z)}\ |\ \ep\in [0,1],\ \delta>0,\ \Im z>0\}<\infty.
\end{align}
Since $G_0(z)=(H-z)^{-1}$ on $L^2$ and $W_0=\<A\>^{-s}$, as $\ep=0$, \eqref{proof_theorem_section_2_1_2} implies
$$
\sup_{\delta>0}\sup_{\Im z>0}|\<(H-z)^{-1}\eta_\delta\<A\>^{-s},\eta_\delta\<A\>^{-s}g\>|\lesssim \norm{f}\,\norm{g},\quad f,g\in L^2.
$$
Taking the complex conjugate, we also have the same estimate for $\Im z<0$ and hence
$$
\sup_{\delta>0}\sup_{z\in \C\setminus\R}|\<(H-z)^{-1}\eta_\delta\<A\>^{-s}f,\eta_\delta\<A\>^{-s}g\>|\lesssim \norm{f}\,\norm{g}.
$$
Since $(H-z)^{-1}$ is analytic for $z\in \C\setminus[0,\infty)$, by letting $\Im z\to0$ while keeping $\Re z < 0$, one can replace $z\in \C\setminus\R$ by $z\in \C\setminus[0,\infty)$ in the supremum. Therefore, we arrive at the desired bound \eqref{theorem_section_2_1_1} by the duality argument.
\vskip0.2cm
{\it Step 2}. It remains to prove the estimate \eqref{proof_theorem_section_2_1_2}. At first, we find by Lemma \ref{lemma_section_2_3}, the fact $\tilde S_1:\H^\sigma\to \H^{-\sigma}$ and the following formula
\begin{align*}
G_\ep(z)-G_{\ep'}(z)=i(\ep-\ep')G_\ep(z)\tilde S_1G_{\ep'}(z),
\end{align*}
that the map $\ep \mapsto F_\ep\in \mathbb B(L^2)$ is continuous on $[0,1]$,  differentiable in $(0,1)$ and satisfies
\begin{align}
\label{proof_theorem_section_2_1_3}
\partial_\ep F_\ep&=iW_\ep \eta_\delta G_\ep\tilde S_1G_\ep \eta_\delta W_\ep+W_\ep' \eta_\delta G_\ep \eta_\delta W_\ep^*
+W_\ep \eta_\delta G_\ep \eta_\delta W_\ep',\quad 0<\ep<1.
\end{align}
We next compute $-i\partial_t(W_\ep \eta_\delta e^{-itA} G_\ep e^{itA}\eta_\delta W_\ep)|_{t=0}$ in two ways. On one hand, we have
\begin{align*}
\partial_t (e^{-itA}G_\ep e^{itA})|_{t=0}
=-[iA,G_\ep]
=-G_\ep(\tilde S_1-i\ep\tilde S_2)G_\ep ,
\end{align*}
which, combined with \eqref{proof_theorem_section_2_1_3}, implies
\begin{align}
\nonumber
&-i\partial_t (W_\ep \eta_\delta e^{-itA}G_\ep e^{itA}\eta_\delta W_\ep)|_{t=0}\\
\label{proof_theorem_section_2_1_4}
&=\partial_\ep F_\ep-W_\ep' \eta_\delta G_\ep \eta_\delta W_\ep
-W_\ep \eta_\delta G_\ep \eta_\delta  W_\ep'+\ep W_\ep \eta_\delta G_\ep\tilde S_2G_\ep \eta_\delta W_\ep^*.
\end{align}
On the other hand, since
$
[\eta_\delta,iA]=\sigma \eta_\delta-\delta |D|\eta_\delta
$
 by Lemma \ref{lemma_section_2_1}, we have
\begin{align*}
-i\partial_t (\eta_\delta e^{-itA})|_{t=0}&=(-A+i\sigma) \eta_\delta-\delta |D|\eta_\delta,\\
-i\partial_t (e^{itA}\eta_\delta )|_{t=0}&=\eta_\delta (A+i\sigma)-\delta |D|\eta_\delta,
\end{align*}
which imply
\begin{align}
\nonumber
&-i\partial_t(W_\ep \eta_\delta e^{-itA}G_\ep e^{itA} \eta_\delta W_\ep)|_{t=0}\\
\label{proof_theorem_section_2_1_5}
&=2i\sigma F_\ep-AF_\ep+F_\ep A-\delta (W_\ep |D|\eta_\delta G_\ep \eta_\delta W_\ep +W_\ep \eta_\delta G_\ep \eta_\delta |D|W_\ep).
\end{align}
By \eqref{proof_theorem_section_2_1_4} and \eqref{proof_theorem_section_2_1_5}, we obtain
\begin{align}
\nonumber
\partial_\ep F_\ep
&=2i\sigma F_\ep
+\Big(-AF_\ep+F_\ep A\Big)
+\Big(W_\ep' \eta_\delta G_\ep \eta_\delta W_\ep+W_\ep \eta_\delta G_\ep \eta_\delta  W_\ep'\Big)\\
\label{proof_theorem_section_2_1_6}
&-\delta \Big(W_\ep |D|\eta_\delta G_\ep \eta_\delta W_\ep +W_\ep \eta_\delta G_\ep \eta_\delta |D|W_\ep\Big)
-\ep W_\ep \eta_\delta G_\ep\tilde S_2G_\ep \eta_\delta W_\ep^*.
\end{align}
\vskip0.2cm
Now we apply Mourre's differential inequality technique to $F_\ep$. For short, we set
\begin{align*}
I_1&=-AF_\ep+F_\ep A,\\
I_2&=W_\ep' \eta_\delta G_\ep \eta_\delta W_\ep+W_\ep \eta_\delta G_\ep \eta_\delta  W_\ep',\\
I_3&=-\delta \Big(W_\ep |D|\eta_\delta G_\ep \eta_\delta W_\ep +W_\ep \eta_\delta G_\ep \eta_\delta |D|W_\ep\Big),\\
I_4&=-\ep W_\ep \eta_\delta G_\ep\tilde S_2G_\ep \eta_\delta W_\ep^*
\end{align*}
so that $\partial_\ep F_\ep=2i\sigma F_\ep+I_1+I_2+I_3 +I_4$. Recalling that $\eta_\delta=|D|^\sigma e^{-\delta|D|}$, we have
$$
\norm{F_\ep }\le \norm{W_\ep}\,\norm{e^{-\delta |D|}}\,\norm{|D|^\sigma G_\ep \eta_\delta W_\ep}\lesssim \norm{S_1^{1/2}G_\ep \eta_\delta W_\ep},
$$
uniformly in $\ep,z$ and $\delta$, where we have used \eqref{lemma_section_2_2_1} and Lemma \ref{lemma_section_2_3}. For $f\in L^2$ we have
\begin{align*}
\|S_1^{1/2} G_\ep \eta_\delta W_\ep f\|^2&=\<\tilde S_1 G_\ep \eta_\delta W_\ep f,G_\ep \eta_\delta W_\ep f\>\\
&\le \ep^{-1}\Big(\ep \<\tilde S_1 G_\ep \eta_\delta W_\ep f,G_\ep \eta_\delta W_\ep f\>+\Im z\norm{G_\ep \eta_\delta W_\ep f}^2\Big)\\
&=-\ep^{-1}\Im\<(\tilde H-z-i\ep \tilde S_1)G_\ep \eta_\delta W_\ep f,G_\ep \eta_\delta W_\ep f\>\\
&=-\ep^{-1}\Im \<\eta_\delta W_\ep f,G_\ep \eta_\delta W_\ep f\>\\
&=-\ep^{-1}\Im \<f,F_\ep f\>\\
&\le \ep^{-1}\|F_\ep\|\|f\|^2.
\end{align*}
Therefore, we obtain the following bounds for $F_\ep$, $|D|^\sigma G_\ep \eta_\delta W_\ep$ and $S_1^{1/2}G_\ep \eta_\delta W_\ep$:
\begin{align}
\label{proof_theorem_section_2_1_7}
\norm{F_\ep}\le \norm{|D|^\sigma G_\ep \eta_\delta W_\ep}\lesssim \norm{S_1^{1/2}G_\ep \eta_\delta W_\ep }\lesssim \ep^{-1/2}\|F_\ep\|^{1/2}\lesssim \ep^{-1}.
\end{align}
Similarly, \eqref{proof_theorem_section_2_1_7} and Lemma \ref{lemma_section_2_3} imply
\begin{align}
\norm{I_1}
\le 2\norm{AF_\ep }\lesssim \ep^{s-1} \norm{|D|^\sigma G_\ep \eta_\delta W_\ep}
\label{proof_theorem_section_2_1_8}
\lesssim\ep^{s-3/2}\norm{F_\ep}^{1/2}.
\end{align}
We also obtain by the same argument and Lemma \ref{lemma_section_2_3} that
\begin{align}
\label{proof_theorem_section_2_1_9}
\norm{I_2}&\le 2\norm{W_\ep' \eta_\delta G_\ep \eta_\delta W_\ep}\lesssim \ep^{s-1} \norm{|D|^\sigma G_\ep \eta_\delta W_\ep }\lesssim\ep^{s-3/2}\norm{F_\ep}^{1/2}.
\end{align}
For the part $I_3$, since $\delta |\xi|e^{-\delta|\xi|}\le 1$ and $s<1$, one has
\begin{align}
\|I_3\|\le 2 \|W_\ep\| \cdot\| \delta |D|e^{-\delta |D|}\| \cdot \||D|^\sigma G_\ep \eta_\delta W_\ep\|\lesssim \ep^{-1/2}\|F_\ep\|^{1/2}\lesssim \ep^{s-3/2}\|F_\ep\|^{1/2}.
\end{align}
Moreover, for the term $I_4$, \eqref{lemma_section_2_2_1} and \eqref{proof_theorem_section_2_1_7} imply
\begin{align}
\norm{I_4}
=\ep\norm{W_\ep \eta_\delta G_\ep \tilde S_2G_\ep \eta_\delta W_\ep}
\label{proof_theorem_section_2_1_10}
\le \ep \|S_1^{1/2}G_{\ep}\eta_\delta W_\ep\|^2
\lesssim \norm{F_\ep}.
\end{align}
Note that all of the implicit constants in \eqref{proof_theorem_section_2_1_7}--\eqref{proof_theorem_section_2_1_10} are independent of $\ep,z$ and $\delta$. Plugging \eqref{proof_theorem_section_2_1_7}--\eqref{proof_theorem_section_2_1_10} into \eqref{proof_theorem_section_2_1_6} yields the following differential inequality:
\begin{align}
\label{proof_theorem_section_2_1_11}
\norm{\partial_\ep F_\ep}\lesssim \norm{F_\ep}+\ep^{s-3/2}\norm{F_\ep}^{1/2},
\end{align}
where the implicit constant is again independent of $\ep,\delta$ and $z$. Hence we have
$$
\norm{F_\ep}\le \norm{F_1}+C_0\int_\ep^1\left(\norm{F_t}+t^{s-3/2}\norm{F_t}^{1/2}\right)dt,\quad 0<\ep\le 1,
$$
with some $C_0>0$ independent of $\ep,z$ and $\delta$.
Let
$$
f(\ep)=\norm{F_1}+C_0\int_\ep^1\left(\norm{F_t}+t^{s-3/2}\norm{F_t}^{1/2}\right)dt,\quad \Psi(\ep)=C_0(\ep-1).
$$
Then $\norm{F_\ep}\le f(\ep)$ and
$$
\partial_\ep (e^{\Psi(\ep)}f(\ep))=C_0e^{\Psi(\ep)}(f(\ep)-\norm{F_\ep}-\ep^{s-3/2}\norm{F_\ep}^{1/2})\ge -C_0\ep^{s-3/2}e^{\Psi(\ep)/2}(e^{\Psi(\ep)}f(\ep))^{1/2},
$$
which, together with \eqref{proof_theorem_section_2_1_7} at $\ep=1$ and the condition $s>1/2$, implies
$$
f(\ep)\le e^{-\Psi(\ep)}\left(f(1)^{1/2}+\frac{C_0}{2}\int_\ep^1 t^{s-3/2}e^{\Psi(t)/2}dt\right)^2\lesssim \norm{F_1}+1\lesssim1
$$
uniformly in $\ep,z$ and $\delta$. This shows \eqref{proof_theorem_section_2_1_2}, completing the proof.
\end{proof}
\vskip0.3cm
\begin{proof}[\underline{Proof of Corollary \ref{corollary_1}}]
The proof is based on Mourre's method \cite{Mou,Mourre_CMP}. Let $I\Subset I_1\Subset I_2\Subset I_3\Subset (0,\infty)$ be relatively compact  intervals and $I^\pm=\{z\in \C:\pm\Im z>0,\ \Re z\in I\}$.
Let $\varphi_1,\varphi_2\in C_0^\infty((0,\infty))$ be such that $\supp \varphi_j\subset I_{j+1}$, $\varphi_j\equiv1$ on $I_j$ for $j=1,2$. Define
$$
H_{\varphi_2}:=H\varphi_2^2(H),\quad M_1:=\varphi_1(H)[H_{\varphi_2},iA]\varphi_1(H),\quad M_2:=[M_1,iA].
$$
All of them are extended to bounded self-adjoint operators on $L^2$. Then one has
$$
M_1=\varphi_1(H)\Big(\varphi_2^2(H)HiA-iAH\varphi_2^2(H)\Big)\varphi_1(H)=\varphi_1(H)[H,iA]\varphi_1(H)
$$
since $\varphi_1\varphi_2\equiv\varphi_1$. Moreover, the assumption \eqref{corollary_1_1} implies Mourre's inequality for $M_1$:
$$
M_1\gtrsim \varphi_1(H)H\varphi_1(H)\gtrsim \varphi_1(H)^2.
$$
Hence, by Mourre's theory (\cite{Mou,Mourre_CMP}), the limit
$\<A\>^{-s}(H_{\varphi_2}-\lambda-i0)^{-1}\<A\>^{-s}\in \mathbb B(L^2)$
 exists for all $\lambda\in I$ as long as $s>1/2$. To remove the operator $\varphi_2^2(H)$, we write
$$(H-z)^{-1}=\varphi_1(H)(H_{\varphi_2}-z)^{-1}+(1-\varphi_1)(H)(H-z)^{-1}.$$ Since $\supp(1-\varphi_1)\subset I_1^c$ and $\mathrm{\mathop{dist}}(I_1^c ,I)\gtrsim1$, the term $(1-\varphi_1)(H)(H-z)^{-1}$ is analytic in $\{z\in \C\ :\ \Re z\in I\}$, satisfying
$$
\norm{\<A\>^{-s}(1-\varphi_1)(H)(H-z)^{-1}\<A\>^{-s}}\lesssim1,\quad \Re z\in I,\ \Im z\in \R.
$$
Moreover, since $\<A\>^{-1}\varphi_1(H)\<A\>$ and $\varphi_1(H)$ are bounded on $L^2$, $\<A\>^{-s}\varphi_1(H)\<A\>^s$ is also bounded on $L^2$ by the complex interpolation. By the result for $H_{\varphi_2}$ and these remarks, we thus obtain that the limits
$\<A\>^{-s}(H-\lambda\mp i0)^{-1}\<A\>^{-s}\in \mathbb B(L^2)$
also exist for $\lambda\in I$ (and hence for all $\lambda>0$).

Finally, we conclude the uniform bound \eqref{corollary_1_2} by letting $\ep\to 0$ in the following estimate
$$
|\<(H-\lambda-i\ep)^{-1}f,g\>|\lesssim\norm{|x|^{\sigma-\gamma}|D|^{-\gamma}f}\ \norm{|x|^{\sigma-\gamma}|D|^{-\gamma}g},\quad f,g\in C_0^\infty(\R^n),
$$
obtained by Theorem \ref{theorem_1}. \end{proof}
\vskip0.3cm
Once we have Theorem \ref{theorem_1}, Theorem \ref{theorem_2} follows from the following abstract result with the choice of $G=|x|^{-\sigma+\gamma}|D|^{\gamma}$ (see e.g.,  \cite[Theorem 5.1]{Kat} and \cite[Theorem 2.3]{DAn}).

\begin{lemma}
\label{lemma_section_2_6}
Let $H$ be a self-adjoint operator and $G$ a densely defined closed operator on a Hilbert space $\H$. Then the following two statements are equivalent.
\begin{itemize}
\item $G$ is $H$-supersmooth in the sense that
$$
\sup_{z\in \C\setminus\R}|\<(H-z)^{-1}G^*f,G^*f\>_{\H}|\lesssim\|f\|^2_{\H},\quad f\in D(G^*).
$$
\item $e^{-itA}\H\subset D(G)$ for a.e $t\in \R$. Moreover, for all $\psi_0\in \H$ and all simple function $F:\R\to D(G^*)$ with $F\in L^2(\R;\H)$, the following estimates are satisfied:
\begin{align*}
\norm{Ge^{-itH}\psi_0}_{L^2(\R;\H)}&\lesssim \norm{\psi_0}_{\H},\\ \left\|G\int_0^te^{-i(t-s)H}G^*F(s)ds\right\|_{L^2(\R;\H)}&\lesssim \norm{F}_{L^2(\R;\H)}.
\end{align*}
\end{itemize}
\end{lemma}

Thus we have finished all the proofs of uniform resolvent and Kato smoothing estimates.

\section{Strichartz estimates}
\label{section_3}
\vskip0.2cm
In this section we will  prove Theorems \ref{theorem_3} and \ref{theorem_4}.
Let us first recall corresponding Strichartz estimates for the free evolutions. Given a self-adjoint operator $H$, we set
$$
\Gamma_{H}F(t,x):=\int_0^te^{-i(t-s)H}F(s,x)ds.
$$

\begin{lemma}
\label{lemma_section_3_1}
Let $n\ge1$, $\sigma>0$, $\sigma\neq1/2$, $H_0=(-\Delta)^\sigma$ and $(p,q),(\tilde p,\tilde q)$ be $n/2$-admissible. Then $e^{-itH_0}$ satisfies the following Strichartz estimates with a gain or loss of regularities:
\begin{align}
\label{lemma_section_3_1_1}
\norm{|D|^{2(\sigma-1)/p}e^{-itH_0}\psi_0}_{L^p_tL^{q,2}_x}&\lesssim \norm{\psi_0}_{L^2_x};
\end{align}
\begin{align}
\label{lemma_section_3_1_2}
\||D|^{2(\sigma-1)/p}\Gamma_{H_0}F\|_{L^p_tL^{q,2}_x}&\lesssim \norm{|D|^{2(1-\sigma)/\tilde p}F}_{L^{\tilde p'}_tL^{\tilde q',2}_x}.
\end{align}
\end{lemma}
\vskip0.2cm

\begin{lemma}
\label{lemma_section_3_2}
Let $n\ge1$, $1/2<\sigma\le1$ $H_0=(-\Delta)^\sigma$  and $(p_1,q_1)$ satisfy
$$
2\le p_1,q_1\le \infty,\quad 1/p_1\le (n-1/2)(1/2-1/q_1),\quad (p_1,q_1)\neq (2,(4n-2)/(2n-3))
$$
and $s_1=-n(1/2-1/q_1)+2\sigma/p_1$. Then $e^{-itH}$ satisfies
\begin{align}
\label{lemma_section_3_2_1}
\norm{|D|^{s_1}e^{-itH_0}\psi_0}_{L^{p_1}_tB[\mathcal L^{q_1}_rL^2_\omega]}&\lesssim \norm{\psi_0}_{L^2_x}.
\end{align}
Moreover, if $n\ge3$, $q_1,q_2>(4n-2)/(2n-3)$ and $F$ is radially symmetric, then one has
\begin{align}
\label{lemma_section_3_2_2}
\||D|^{s(2,q_1)}\Gamma_{H_0}F\|_{L^2_tL^{q_1,2}_x}&\lesssim \| |D|^{-s(2,q_2)}F\|_{L^{2}_tL^{q_2',2}_x},
\end{align}
where $s(2,q_j)=-n(1/2-1/q_j)+\sigma$.
\end{lemma}
\vskip0.2cm
Lemma \ref{lemma_section_3_1} is well-known and Lemma \ref{lemma_section_3_2} was obtained by \cite{Guo} and \cite{GLNY}. We will give their proofs in Appendix \ref{appendix_A} for the sake of self-containedness. In addition to Theorem \ref{theorem_2} and these lemmas, the proof of Theorems \ref{theorem_3} and \ref{theorem_4} relies on the following abstract perturbation method due to Rodnianski--Schlag \cite{RoSc} (see also Burq et al \cite{BPST2} for the homogeneous endpoint case and \cite{BoMi} for the inhomogeneous endpoint case).

Let $\X,\Y$ be two Banach spaces of functions on $\R^n$ such that $\X\cap L^2$ (resp. $\Y\cap L^2$) is dense in $\X$ (resp. $\Y$). Let $H_0$ be a self-adjoint operator on $L^2(\R^n)$ with form domain $\H^\sigma$ and $V\in L^{n/(2\sigma),\infty}(\R^n)$ be a real-valued potential such that the form sum $H=H_0+V$ defines a self-adjoint operator  with form domain $\H^\sigma$. Let $W_1,W_2\in L^{n/\sigma,\infty}(\R^n)$ be such that $V=W_1W_2$. Consider the following series of estimates:
\begin{align}
\label{lemma_section_3_3_1}
\norm{e^{-itH_0}\psi_0}_{L^p_t\Y}&\lesssim \norm{\psi_0}_{L^2_x},\\
\label{lemma_section_3_3_2}
\norm{\Gamma_{H_0}F}_{L^2_t\Y}&\lesssim \norm{F}_{L^2_t\X},\\
\label{lemma_section_3_3_3}
\norm{W_1e^{-itH_0}\psi_0}_{L^2_tL^2_x}&\lesssim \norm{\psi_0}_{L^2_x},\\
\label{lemma_section_3_3_4}
\norm{\Gamma_{H_0}W_1F}_{L^2_t\Y}&\lesssim \norm{F}_{L^2_tL^2_x},\\
\label{lemma_section_3_3_5}
\norm{W_1\Gamma_{H_0}F}_{L^2_tL^2_x}&\lesssim \norm{F}_{L^2_t\X},\\
\label{lemma_section_3_3_6}
\norm{W_2\Gamma_{H_0}F}_{L^2_tL^2_x}&\lesssim \norm{F}_{L^2_t\X},\\
\label{lemma_section_3_3_7}
\norm{W_2e^{-itH}\psi_0}_{L^2_tL^2_x}&\lesssim \norm{\psi_0}_{L^2_x},\\
\label{lemma_section_3_3_8}
\norm{W_2\Gamma_{H}W_2F}_{L^2_tL^2_x}&\lesssim \norm{F}_{L^2_tL^2_x}.
\end{align}
\vskip0.3cm
\begin{lemma}
\label{lemma_section_3_3}
Under the above setting, the following statements are satisfied:
\vskip0.2cm
\begin{itemize}
\item[(1)] \underline{The endpoint case}: if \eqref{lemma_section_3_3_1} with $p=2$, \eqref{lemma_section_3_3_4} and \eqref{lemma_section_3_3_7} hold, then one has
\begin{align}
\label{lemma_section_3_3_9}
\norm{e^{-itH}\psi_0}_{L^2_t\Y}&\lesssim \norm{\psi_0}_{L^2_x}.
\end{align}
Moreover, if \eqref{lemma_section_3_3_2}, \eqref{lemma_section_3_3_4}--\eqref{lemma_section_3_3_6} and \eqref{lemma_section_3_3_8} hold, then one has
\begin{align}
\label{lemma_section_3_3_10}
\norm{\Gamma_{H}F}_{L^2_t\Y}&\lesssim \norm{F}_{L^2_t\X}.
\end{align}
\item[(2)] \underline{The non-endpoint case}: if \eqref{lemma_section_3_3_1} with $p>2$, \eqref{lemma_section_3_3_3} and \eqref{lemma_section_3_3_7} hold, then one has
\begin{align}
\label{lemma_section_3_3_11}
\norm{e^{-itH}\psi_0}_{L^p_t\Y}&\lesssim \norm{\psi_0}_{L^2_x}.
\end{align}
\end{itemize}
\end{lemma}

\begin{proof}
The complete proof of the lemma can be found in \cite[Theorems 4.9 and 4.10]{BoMi}. We here give a brief sketch of the proof for the sake of self-containedness.

We begin with the following Duhamel formulas (see \cite[Proposition 4.4]{BoMi}):
\begin{align}
\label{lemma_section_3_3_proof_1}
U_H&=U_{H_0}-i\Gamma_{H_0}VU_H,\\
\label{lemma_section_3_3_proof_2}
\Gamma_H
&=\Gamma_{H_0}-i\Gamma_{H_0}V\Gamma_{H},\\
\label{lemma_section_3_3_proof_3}
\Gamma_H&=\Gamma_{H_0}-i\Gamma_{H}V\Gamma_{H_0},
\end{align}
where $U_{H_0}=e^{-itH_0}$ and $U_H=e^{-itH}$ (Strictly speaking, these formulas should be regarded in the sense of quadratic forms. We omit the details for simplicity and refer to \cite[Section 4]{BoMi}). Then the homogeneous endpoint estimate \eqref{lemma_section_3_3_9} is easy to obtain as follows:
\begin{align*}
\norm{U_{H}\psi_0}_{L^2_t\Y}
&\le \norm{U_{H_0}\psi_0}_{L^2_t\Y}+\norm{\Gamma_{H_0}VU_H\psi_0}_{L^2_t\Y}\\
&\lesssim \norm{\psi_0}_{L^2_x}+\norm{\Gamma_{H_0}W_1}_{L^2_tL^2_x\to L^2_t\Y}\norm{W_2U_H\psi_0}_{L^2_tL^2_x}
\lesssim \norm{\psi_0}_{L^2_x},
\end{align*}
where we have used \eqref{lemma_section_3_3_proof_1}, \eqref{lemma_section_3_3_1}, \eqref{lemma_section_3_3_4} and \eqref{lemma_section_3_3_7}.

In order to derive \eqref{lemma_section_3_3_10}, we first use \eqref{lemma_section_3_3_proof_2}, \eqref{lemma_section_3_3_2} and \eqref{lemma_section_3_3_4} to obtain
\begin{align*}
\norm{\Gamma_{H}F}_{L^2_t\Y}
&\le\norm{\Gamma_{H_0}F}_{L^2_t\Y}+\norm{\Gamma_{H_0}V\Gamma_{H}F}_{L^2_t\Y},\\
&\lesssim \norm{F}_{L^2_t\X}+\norm{\Gamma_{H_0}W_1}_{L^2_tL^2_x\to L^2_t\Y}\norm{W_2\Gamma_{H}F}_{L^2_tL^2_x},\\
&\lesssim \norm{F}_{L^2_t\X}+\norm{W_2\Gamma_{H}F}_{L^2_tL^2_x}.
\end{align*}
Applying \eqref{lemma_section_3_3_proof_3}, \eqref{lemma_section_3_3_6}, \eqref{lemma_section_3_3_9} and \eqref{lemma_section_3_3_5} to the term $\norm{W_2\Gamma_{H}F}_{L^2_tL^2_x}$ then yields
\begin{align*}
\norm{W_2\Gamma_{H}F}_{L^2_tL^2_x}
&\le \norm{W_2\Gamma_{H_0}F}_{L^2_tL^2_x}+\norm{W_2\Gamma_{H}V\Gamma_{H_0}F}_{L^2_tL^2_x}\\
&\lesssim \norm{F}_{L^2_t\X}+\norm{W_2\Gamma_{H}W_2}_{L^2_tL^2_x\to L^2_tL^2_x}\norm{W_1\Gamma_{H_0}F}_{L^2_tL^2_x}\\
&\lesssim \norm{F}_{L^2_t\X},
\end{align*}
and \eqref{lemma_section_3_3_10} follows.

To prove \eqref{lemma_section_3_3_11}, by the same argument as above, it is enough to show the bound
$$
\norm{\Gamma_{H_0}W_1F}_{L^p_t\Y}\lesssim \norm{F}_{L^2_tL^2_x}.
$$
Since $p>2$, Christ--Kiselev's lemma (see Appendix \ref{appendix_misc} (vi)) allows us to replace the time interval $[0,t]$ in the formula of $\Gamma_{H_0}$ by $[0,\infty)$. Hence it suffices to show the estimate
$$
\left\|e^{-itH_0}\int_0^\infty e^{isH_0}W_1F(s)ds\right\|_{L^p_t\Y}\lesssim \norm{F}_{L^2_tL^2_x}
$$
which follows from \eqref{lemma_section_3_3_1} and the dual estimate of \eqref{lemma_section_3_3_3}.
\end{proof}
\vskip0.3cm
In the following proofs of Theorems \ref{theorem_3} and \ref{theorem_4}, we set for short that
\begin{align}
\label{2^*}
2^*(\sigma)=\frac{2n}{n-2\sigma},\ \ 2_*(\sigma)=(2^*(\sigma))'=\frac{2n}{n+2\sigma},\ \ 2^*=\frac{2n}{n-2},\ \  2_*=\frac{2n}{n+2}.
\end{align}
\vskip0.3cm
\begin{proof}[\underline{Proof of Theorem \ref{theorem_3}}]
By virtue of \eqref{Sobolev_2}, it suffices to show \eqref{theorem_3_2} only. Let $(p,q)$ be $n/2$-admissible, and let $W_1=|x|^\sigma V\in L^{n/\sigma,\infty}$ and $W_2=|x|^{-\sigma}\in L^{n/\sigma,\infty}$. Define Banach spaces $\X_{2},\Y_p$ through the norms
$
\norm{f}_{\X_{2}}=\norm{|D|^{1-\sigma}f}_{L^{2_*,2}_x}$ and $\norm{f}_{\Y_p}=\norm{|D|^{2(\sigma-1)/p}f}_{L^{q,2}_x}$.
\vskip0.2cm
Let us check that the conditions in Lemma \ref{lemma_section_3_3} are satisfied. At first, Lemma \ref{lemma_section_3_1} implies
\begin{align}
\label{theorem_3_proof_1}
\norm{U_{H_0}\psi_0}_{L^p_t\Y_p}\lesssim \norm{\psi_0}_{L^2_x},\quad\norm{U_{H_0}\psi_0}_{L^2_t\Y_2}\lesssim \norm{\psi_0}_{L^2_x},\quad
\norm{\Gamma_{H_0}F}_{L^2_t\Y_2}\lesssim \norm{F}_{L^2_t\X_2}.
\end{align}
Next, note that $$1/2^*-1/2^*(\sigma)=(\sigma-1)/n,\ 1/2_*(\sigma)-1/2_*=(\sigma-1)/n,$$
$$1/2=\sigma/n+1/2^*(\sigma),\ 1/2_*(\sigma)=1/2+\sigma/n,$$
hence the inequalities \eqref{Holder} and  \eqref{Sobolev} yield that the following estimates hold for $j=1,2$:
\begin{align*}
\norm{f}_{L^{2^*(\sigma),2}}&\lesssim \norm{f}_{\Y_2},\\
 \norm{f}_{\X_{2}}&\lesssim \norm{f}_{L^{2_*(\sigma),2}},\\
\norm{W_jf}_{L^2_x}&\lesssim \norm{W_j}_{L^{n/\sigma,\infty}}\norm{f}_{L^{2^*(\sigma),2}}\lesssim  \norm{f}_{\Y_2},\\
\norm{W_jf}_{\X_{2}}&\lesssim \norm{W_jf}_{L^{2_*(\sigma),2}}\lesssim \norm{W_j}_{L^{n/\sigma,\infty}}\norm{f}_{L^2_x}\lesssim \norm{f}_{L^2_x}.
\end{align*}
Then Lemma \ref{lemma_section_3_1} with $p=2$ and these estimates listed above imply
\begin{align}
\label{theorem_3_proof_2}
\norm{W_2U_{H_0}\psi_0}_{L^2_tL^2_x}
&\lesssim \norm{U_{H_0}\psi_0}_{L^2_t\Y_2}\lesssim \norm{\psi_0}_{L^2_x},\\
\label{theorem_3_proof_3}
\norm{W_j\Gamma_{H_0}F}_{L^2_tL^2_x}&\lesssim \norm{\Gamma_{H_0}F}_{L^2_t\Y_2}\lesssim \norm{F}_{L^2_t\X_2},\quad j=1,2,\\
\label{theorem_3_proof_4}
\norm{\Gamma_{H_0}W_1G}_{L^2_t\Y_2}&\lesssim \norm{W_1G}_{L^2_t\X_2}\lesssim \norm{G}_{L^2_tL^2_x}.
\end{align}
Finally, \eqref{theorem_2_2} implies
\begin{align}
\label{theorem_3_proof_5}
\norm{W_2U_{H}\psi_0}_{L^2_tL^2_x}\lesssim \norm{\psi_0}_{L^2_x},\quad \norm{W_2\Gamma_{H}W_2F}_{L^2_tL^2_x}\lesssim\norm{F}_{L^2_tL^2_x}.
\end{align}
By \eqref{theorem_3_proof_1}--\eqref{theorem_3_proof_5}, we have obtained all of estimates \eqref{lemma_section_3_3_1}--\eqref{lemma_section_3_3_8}. Therefore, we can apply Lemma \ref{lemma_section_3_3} to obtain \eqref{theorem_3_2} for the double endpoint case with $p=\tilde p=2$ and the homogeneous non-endpoint cases with $p>2$ and $F\equiv0$. Using Christ--Kiselev's lemma (see Appendix \ref{appendix_misc}), we also obtain all of the other cases from these two cases. \end{proof}
\vskip0.2cm
\begin{proof}[\underline{Proof of Theorem \ref{theorem_4}}]
Let $W_j$ be as above. Under the condition $|x|^{2\sigma}V\in L^\infty$, $|x|^\sigma W_j\in L^\infty$ and hence $W_1,W_2$ are $H_0$- and $H$-supersmooth by \eqref{theorem_2_2}. In particular, one has
\begin{align}
\label{theorem_4_proof_1}
\norm{W_jU_{H_0}\psi}_{L^2_tL^2_x}+\norm{W_jU_{H}\psi}_{L^2_tL^2_x}&\lesssim \norm{\psi_0}_{L^2_x},\quad j=1,2.
\end{align}

To obtain the estimates \eqref{theorem_3_2} for non-endpoint cases, by virtue of Lemma \ref{lemma_section_3_3} (2) with $\Y=\Y_p$ defined above, \eqref{lemma_section_3_1_1} and \eqref{theorem_4_proof_1}, it suffices to show the following bound
$$
\norm{|D|^{2(\sigma-1)/p}\Gamma_{H_0}W_1F}_{L^p_tL^q_x}\lesssim \norm{F}_{L^2_tL^2_x}.
$$
Since $p>2$, as in the proof of Lemma \ref{lemma_section_3_3}, this estimate follows from \eqref{lemma_section_3_1_1}, the dual estimate of \eqref{theorem_4_proof_1} for $j=1$ and Christ--Kiselev's lemma.
\vskip0.2cm
Let $(p_1,q_1)$ satisfy \eqref{theorem_4_1}. The estimate \eqref{theorem_4_2} can be obtained similarly. Indeed, by virtue of Lemma \ref{lemma_section_3_3} (2) with $\Y$ defined by the norm $\norm{f}_\Y=\norm{|D|^{s_1}f}_{B[\mathcal L^{q_1}_rL^2_\omega]}$, \eqref{lemma_section_3_2_1}, \eqref{theorem_4_proof_1} and Christ--Kiselev's lemma, it suffices to show the following bound
$$
\left\||D|^{s_1}e^{-itH_0}\int_0^\infty e^{isH_0}W_1F(s)ds\right\|_{L^{p_1}_tB[\mathcal L^{q_1}_rL^2_\omega]}\lesssim \norm{F}_{L^2_tL^2_x},
$$
which follows from \eqref{lemma_section_3_2_1} and the dual estimate of \eqref{theorem_4_proof_1}.

Finally, in order to obtain the endpoint estimate \eqref{theorem_4_3}, we shall apply Lemma \ref{lemma_section_3_3} (1) with the choice of $\X=|D|^{s(2,q_2)}L^{q_2',2}$ and $\Y=|D|^{-s(2,q_1)}L^{q_1,2}$ defined by the norms $$\norm{f}_{\X}=\norm{|D|^{-s(2,q_2)}f}_{L^{q_2',2}_x},\quad \norm{f}_{\Y}=\norm{|D|^{s(2,q_1)}f}_{L^{q_1,2}_x}.$$
Let $\psi_0,F,$ and $V$ be radially symmetric and $q_1,q_2>(4n-2)/(2n-3)$. Note that $\psi$ is also radially symmetric.
By Lemma \ref{lemma_section_3_2} for radially symmetric data, \eqref{lemma_section_3_3_1} and \eqref{lemma_section_3_3_2} hold. 
Since the condition $\sigma>n/(2n-1)$ is equivalent to the inequality $2^*(\sigma)>(4n-2)/(2n-3)$, we can use \eqref{lemma_section_3_2_2} with $q_1$ or $q_2=2^*(\sigma)$ to obtain for radially symmetric $F,G$ that
\begin{align*}
\norm{\Gamma_{H_0}W_1F}_{L^2_t\Y}
&\lesssim \norm{W_1F}_{L^2_tL^{2_*(\sigma),2}_x}
\lesssim \norm{W_1}_{L^{n/\sigma,\infty}_x}\norm{F}_{L^2_tL^2_x},\\
\norm{W_j\Gamma_{H_0}G}_{L^2_tL^2_x}
&\lesssim \norm{W_j}_{L^{n/\sigma,\infty}_x}\norm{\Gamma_{H_0}G}_{L^2_tL^{2^*(\sigma),2}_x}
\lesssim \norm{G}_{L^2_t\X}.
\end{align*}
Hence \eqref{lemma_section_3_3_4}--\eqref{lemma_section_3_3_6} are satisfied. \eqref{lemma_section_3_3_7} and \eqref{lemma_section_3_3_8} also hold since $W_2$ is $H$-supersmooth. With the embeddings $L^{q_1,2}\hookrightarrow L^{q_1}$ and $L^{q_2'}\hookrightarrow L^{q_2',2}$ at hand, we thus can apply Lemma \ref{lemma_section_3_3} (1) to obtain the bound \eqref{theorem_4_3}. This completes the proof. 
\end{proof}

\section{Uniform Sobolev estimates}
\label{section_4}
This section is devoted to the proof of Theorem \ref{theorem_5} and Remark \ref{remark_theorem_5}. We begin with recalling uniform Sobolev estimates for the free resolvent $(H_0-z)^{-1}$.
Let \begin{align*}
A_0&=\Big(\frac{n+1}{2n},\frac{n-4\sigma+1}{2n}\Big)\ \ \text{if}\ \ \sigma\le \frac{n+1}{4},\quad
A_0=\Big(\frac{2\sigma}{n},0\Big)\ \ \text{if}\ \ \sigma\ge\frac{n+1}{4},\\
A_1&=\Big(\frac{n+2\sigma}{2n},\frac{n-2\sigma}{2n}\Big),\quad B_1=\Big(\frac{n+3}{2(n+1)},\frac{n-1}{2(n+1)}\Big),
\end{align*}
and $A_0'$ be the dual point of $A_0$, namely $A_0'=(1/q',1/p')$ if $A_0=(1/p,1/q)$.
\vskip0.1cm
Define $\Omega\subset (0,1)\times(0,1)$ be the union of the interior of the triangle $A_0B_1A_0'$ and the open line segment $A_0A_0'$  and the point $B_1$. (see Figure 1 below).
\vskip0.2cm
\begin{lemma}
\label{lemma_section_4_1}
Let $n\ge3$, ${n}/{(n+1)}\le \sigma<n/2$ and $H_0=(-\Delta)^\sigma$. Then for any $(1/p,1/q)\in \Omega$,
\begin{align}
\label{lemma_section_4_1_1}
\norm{(H_0-z)^{-1}f}_{L^q}\lesssim |z|^{\frac{n}{2\sigma}(\frac1p-\frac1q)-1}\norm{f}_{L^p},\quad z\in \C\setminus[0,\infty).
\end{align}
\end{lemma}
\vskip0.3cm
\begin{proof}[Sketch of proof]
\eqref{lemma_section_4_1_1} was proved by \cite[Theorem 1.4 and Corollary 3.2]{HYZ} if $(p,q)$ satisfies \begin{itemize}
\vskip0.2cm
\item either that $1/p-1/q=2\sigma/n$ and $2n/(n+4\sigma-1)<p<2n/(n+1)$,
\vskip0.2cm
\item or that $1/p+1/q=1$ and $2n/(n+2\sigma)\le p\le 2(n+1)/(n+3)$,
\vskip0.2cm
\end{itemize}
that is $(1/p,1/q)\in A_0A_0'\cup \overline{A_1B_1}$ (see Figure 1). Then the assertion follows from   this existing result and the interpolation theory (see Appendix \ref{appendix_misc} below).
\end{proof}

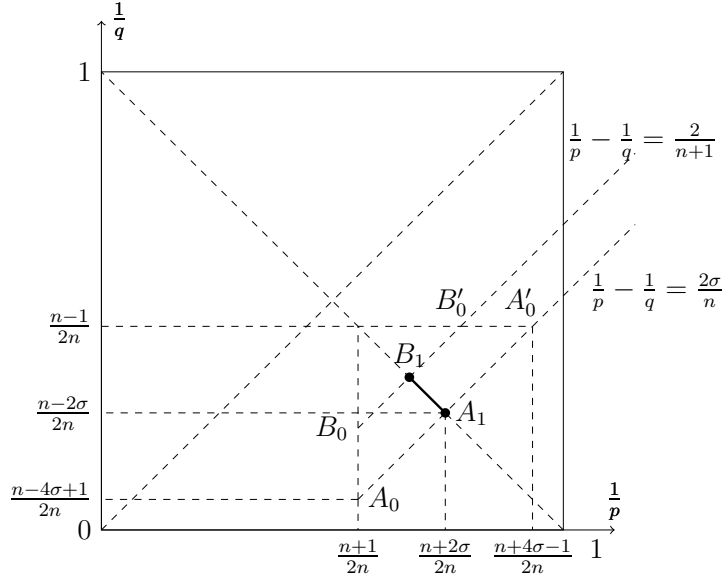
\begin{figure}[htbp]
\label{figure}
\begin{center}
\scalebox{0.9}[0.9]{
\begin{tikzpicture}

\draw (0,0) rectangle (6.75,6.75);
\draw[->]  (0,0) -- (0,7.5);
\draw[->]  (0,0) -- (7.5,0);
\draw (6.75,0) node[below] {$\ \ \ \ \ \ \ 1$};
\draw (0,6.75) node[left] {$1$};
\draw (0,0) node[below, left] {$0$};
\draw (7.5,0) node[above] {$\frac1p$};
\draw (0,7.5) node[right] {$\frac1q$};
\draw (7.5,0) node[above] {$\frac1p$};
\draw (0,7.5) node[right] {$\frac1q$};
\draw[dashed] (0,0) -- (6.75,6.75);
\draw[dashed] (6.75,0) -- (0,6.75);

\draw[dashed] (6.3,3) -- (7.8,4.5);
\draw (7,3.5) node[right] {$\frac1p-\frac1q=\frac {2\sigma}{n}$};
\draw[dashed] (5.25,3) -- (7.8,5.55);
\draw (7.9,5.30) node[above] {$\frac1p-\frac1q=\frac{2}{n+1}$};

\draw[dashed] (3.75,1.5) -- (3.75,3.0) -- (5.25,3.0);
\draw[dashed] (3.75,0.45) -- (0,0.45);
\draw[dashed] (6.3,3) -- (6.3,0);
\draw (3.75,0.45) node[right] {$A_0$}; 
\draw (6.3,3)  node[above] {$\!\!\!\!\!A_0'$}; 
\draw (3.75,1.5)  node[left] {$B_0$};   
\draw (5.25,3)  node[above] {$\!\!\!\!B_0'$}; 
\fill (5.025,1.725) circle (2pt) node[right] {$A_1$}; 
\fill (4.5,2.25) circle (2pt) node[above] {$B_1$};      
\draw[dashed] (3.75,0.45) -- (6.3,3); 
\draw[dashed]  (3.75,1.5) -- (5.25,3); 
\draw[dashed] (3.75,0.45) -- (3.75,1.5) ; 
\draw[dashed] (5.25,3) -- (6.3,3) ; 
\draw[line width=1pt] (5.025,1.725) -- (4.5,2.25); 

\draw[dashed] (3.75,0.45) -- (3.75,0);
\draw (3.75,0) node[below] {$\frac{n+1}{2n}$};
\draw (0,0.45) node[left] {$\frac{n-4\sigma+1}{2n}$};
\draw(6.3,0) node[below]  {$\frac{n+4\sigma-1}{2n}$};
\draw (0,3) node[left] {$\frac{n-1}{2n}$};
\draw (5.025,0) node[below] {$\frac{n+2\sigma}{2n}$};
\draw (0,1.725) node[left] {$\frac{n-2\sigma}{2n}$};
\draw[dashed] (5.025,1.725) -- (5.025,0);
\draw[dashed] (5.025,1.725) -- (0,1.725);
\draw[dashed] (0,3) -- (3.75,3.0);

\end{tikzpicture}
}
\end{center}
\caption{In case of $\sigma\le (n+1)/4$, $\Omega=\mathrm{\mathop{int}}(A_0B_1A_0')\cup A_0A_0'\cup B_1$. The admissible set of $(1/p,1/q)$ in Theorem \ref{theorem_5} is the closed line segment $\overline{A_1B_1}$. An expected optimal range for \eqref{lemma_section_4_1_1} is $\mathrm{\mathop{int}}(A_0B_0B_0'A_0')\cup A_0A_0'\cup B_0B_1$. In the case $\sigma>(n+1)/4$, $A_0,A_0'$ are the intersection points of the two lines $1/p-1/q=2\sigma/n$ and $1/q=0$ or $1/p=1$, respectively.}
\end{figure}
\begin{remark}
\label{remark_section_4_1_1}
In case of $\sigma=1$, \eqref{lemma_section_4_1_1} is known to hold if and only if
$$
\frac{2}{n+1}\le \frac1p-\frac1q\le \frac2n,\quad \frac{2n}{n+3}< p<\frac{2n}{n+1},\quad \frac{2n}{n-1}<q<\frac{2n}{n-3}
$$
(see \cite{KRS} and \cite{Gut}). For higher-order cases $\sigma\in\N$ and $\sigma\ge1$, uniform Sobolev estimates were studied by \cite{SYY} and \cite{HuZh} for more general constant coefficient elliptic operators (possibly with small decaying potentials) than $(-\Delta)^\sigma$. For Schr\"odinger operators $H=-\Delta+V$ with large potentials $V\in L^{n/2}$, we refer to \cite{HYZ}, \cite{BoMi}, \cite{Mizutani_JST} and \cite{Mizutani_APDE}.
\end{remark}

Let $R_T(z):=(T-z)^{-1}$ for $T=H_0,H$. The proof of Theorem \ref{theorem_5} is in some sense analogous to that of Theorem \ref{theorem_3} and  relies on the following abstract perturbation lemma.

\begin{lemma}
\label{lemma_section_4_2}
Let $H_0$, $V=W_1W_2$, $H=H_0+V$, $\X$ and $\Y$ be as in Lemma \ref{lemma_section_3_3} and $z\notin \sigma(H_0)\cup\sigma(H)$. Suppose there exist constants $r_1,...,r_5>0$ possibly depending on $z$ such that the following series of estimates are satisfied:
\begin{align}
\label{lemma_section_4_2_1}
\norm{R_{H_0}(z)f}_{\Y}&\le r_1\norm{f}_{\X},\\
\label{lemma_section_4_2_2}
\norm{W_2R_{H_0}(z)f}_{L^2}&\le r_2\norm{f}_{\X},\\
\label{lemma_section_4_2_3}
\norm{W_1R_{H_0}(z)f}_{L^2}&\le r_3\norm{f}_{\X},\\
\label{lemma_section_4_2_4}
\norm{R_{H_0}(z)W_1f}_{\Y}&\le r_4\norm{f}_{L^2},\\
\label{lemma_section_4_2_5}
\norm{W_2R_{H}(z)W_2f}_{L^2}&\le r_5\norm{f}_{L^2}.
\end{align}
Then the following resolvent estimate for $H$ holds:
\begin{align}
\label{lemma_section_4_2_6}
\norm{R_{H}(z)f}_{\Y}\le (r_1+r_2r_4+r_3r_4r_5)\norm{f}_{\X} .
\end{align}
\end{lemma}

\begin{proof}
A more general version of the lemma with its complete proof can be found in \cite[Proposition 4.1]{BoMi}. Hence  only a brief sketch of the proof is given here. The proof is based on the following resolvent formulas (see \cite[Section 4]{BoMi}):
\begin{align}
\label{lemma_section_4_2_proof_1}
R_H(z)
=R_{H_0}(z)-R_{H_0}(z)VR_H(z)
=R_{H_0}(z)-R_{H}(z)VR_{H_0}(z).
\end{align}
Since the desired estimate for $R_{H_0}$ is assumed in \eqref{lemma_section_4_2_1}, it suffices to deal with the term $R_{H_0}(z)VR_H(z)$. The estimate \eqref{lemma_section_4_2_4} implies
\begin{align}
\label{lemma_section_4_2_proof_3}
\norm{R_{H_0}(z)VR_H(z)f}_\Y\le r_4\norm{W_2R_H(z)f}_{L^2}
\end{align}
Using \eqref{lemma_section_4_2_proof_1}, \eqref{lemma_section_4_2_2}, \eqref{lemma_section_4_2_5} and \eqref{lemma_section_4_2_3}, we obtain
\begin{align}
\nonumber
\norm{W_2R_H(z)f}_{L^2}
&\le \norm{W_2R_{H_0}(z)f}_{L^2}+\norm{W_2R_H(z)W_2}\norm{W_1R_{H_0}(z)f}_{L^2}\\
\label{lemma_section_4_2_proof_4}
&\le (r_2+r_3r_5)\norm{f}_{\X}.
\end{align}
Then \eqref{lemma_section_4_2_6} follows from \eqref{lemma_section_4_2_1}, \eqref{lemma_section_4_2_proof_3} and \eqref{lemma_section_4_2_proof_4}.
\end{proof}
\vskip0.3cm
\begin{proof}[\underline{Proof of Theorem \ref{theorem_5} and Remark \ref{remark_theorem_5}}]
Let $\Omega_0:=\Omega\setminus\{B_1\}$. For any $(1/p,1/q)\in\Omega_0$ there exists an open line segment $I\subset \Omega_0$ containing $(1/p,1/q)$, which is not parallel to either the vertical or the horizontal lines (see Figure 1). Hence the real interpolation (see Appendix \ref{appendix_misc}) allows us to replace $L^p$ and $L^q$ in \eqref{lemma_section_4_1_1} by $L^{p,2}$ and $L^{q,2}$ if $(1/p,1/q)\in \Omega_0$.

\vskip0.3cm
{\it Step 1}. Let $2n/(n+2\sigma)\le p_0\le 2(n+1)/(n+3)$. To prove \eqref{theorem_5_1}, it suffices to show
\begin{align}
\label{theorem_5_proof_1}
\norm{(H-z)^{-1}f}_{L^{p_0'}}\lesssim |z|^{\frac{n}{\sigma p_0}-\frac{n+2\sigma}{2\sigma}}\norm{f}_{L^{p_0}},\quad f\in C_0^\infty(\R^n),\ z\in \C\setminus[0,\infty).
\end{align}
Indeed, combined with Corollary \ref{corollary_1}, \eqref{theorem_5_proof_1} also implies the same uniform estimate for $(H-z\mp i0)^{-1}$ with $z>0$ under the condition \eqref{corollary_1_1}.  Then the density argument yields the desired result \eqref{theorem_5_1} for all $f\in L^{p_0}$.
\vskip0.2cm
Recall that $2^*(\sigma)=2n/(n-2\sigma)$ and $2_*(\sigma)=2n/(n+2\sigma)$. Let $W_1:=|x|^\sigma V,\ W_2:=|x|^{-\sigma}\in L^{n/\sigma,\infty}$. Since $1/2={1}/{2^*(\sigma)}+\sigma/n$ and ${1}/{2_*(\sigma)}=\sigma/n+1/2$, one has
\begin{align}
\label{theorem_5_proof_2}
\norm{W_jf}_{L^2}\lesssim \norm{f}_{L^{2^*(\sigma),2}},\quad \norm{W_jf}_{L^{2_*(\sigma),2}}\lesssim \norm{f}_{L^2}
\end{align}
for $j=1,2$. Note that $(1/p_0,1/2^*(\sigma)),(1/2_*(\sigma),p_0')\in \Omega_0$ (see Figure 1).
Applying Lemma \ref{lemma_section_4_1} with $(p,q)=(p_0,2^*(\sigma))$ or with $(p,q)=(2_*(\sigma),p_0')$ and \eqref{theorem_5_proof_2}, we then have
\begin{align}
\label{theorem_5_proof_3}
\norm{W_jR_{H_0}(z)f}_{L^2}&\lesssim |z|^{\frac{n}{2\sigma}(1/p_0-1/{2^*(\sigma)})-1}\norm{f}_{L^{p_0,2}},\ j=1,2;
\end{align}
\begin{align}
\label{theorem_5_proof_4}
\norm{R_{H_0}(z)W_1f}_{L^{p_0',2}}&\lesssim |z|^{\frac{n}{2\sigma}(1/{2_*(\sigma)}-1/p_0')-1}\norm{f}_{L^2}.
\end{align}
Moreover, the Kato--Yajima estimate \eqref{theorem_1_2} implies
\begin{align}
\label{theorem_5_proof_5}
\norm{W_2R_H(z)W_2f}_{L^2}\lesssim \norm{f}_{L^2}.
\end{align}
By virtue of \eqref{lemma_section_4_1_1}, \eqref{theorem_5_proof_3}--\eqref{theorem_5_proof_5}, one can apply Lemma \ref{lemma_section_4_2} with the choice of $\X=L^{p_0}$ and $\Y=L^{p_0'}$, to obtain
 the desired bound for $R_H(z)$ since
$$
\frac{n}{2\sigma}\Big(\frac{1}{p_0}-\frac{1}{2^*(\sigma)}\Big)-1+\frac{n}{2\sigma}\Big(\frac{1}{2_*(\sigma)}-\frac{1}{p_0'}\Big)-1
=\frac{n}{\sigma p_0}-\frac{n+2\sigma}{2\sigma},
$$
$L^{p_0}\hookrightarrow L^{p_0,2}$ and $L^{p_0',2}\hookrightarrow L^{p_0'}$ (see Appendix \ref{appendix_misc}). This completes the proof of \eqref{theorem_5_proof_1}.
\vskip0.3cm
{\it Step 2}. We next show \eqref{theorem_5_2}. Let $\Im z\ge0$, $\ep>0$, $u\in C_0^\infty(\R^n\setminus\{0\})$ and $f=(H-z)u$. Since $H_0u,Vu\in L^2\cap L^{{2_*(\sigma)}}$ by \eqref{Holder} and $|x|^\sigma V\in L^{n/\sigma,\infty}$, we have $f\in L^2\cap L^{{2_*(\sigma)}}$ and$$(H-z-i\ep)^{-1}f=u+i\ep (H-z-i\ep)^{-1}u$$ in $L^2(\R^n)$. Since $\Im (z+i\ep)\ge \ep>0$, we know by \eqref{theorem_5_1} with $p=2_*(\sigma)$ that
$$
\norm{u}_{L^{2^*(\sigma)}}\lesssim \norm{f}_{L^{2_*(\sigma)}}+\ep\norm{u}_{L^{2_*(\sigma)}}
$$
uniformly in $z$ and $\ep>0$, which implies \eqref{theorem_5_2} for $\Im z\ge0$ by letting $\ep\searrow0$. By taking the complex conjugate, we also obtain \eqref{theorem_5_2} for $\Im z\le0$.

\vskip0.3cm
{\it Step 3}. We finally show \eqref{theorem_5_3} for $\sigma>1$. The proof is different from the above argument, based on a simple trick due to T. Duyckaerts \cite{Duyckaerts} as follows. By the double endpoint Strichartz estimate \eqref{theorem_3_2} with $(p,q)=(\tilde p,\tilde q)=(2,2^*)$, we obtain
$$
\norm{|D|^{\sigma-1}\psi}_{L^2([-T,T];L^{\frac{2n}{n-2}}_x)}\lesssim \norm{\psi_0}_{L^2_x}+\norm{|D|^{1-\sigma}F}_{L^2([-T,T];L^{\frac{2n}{n+2}}_x)},
$$
uniformly in $T$. Plugging $\psi=e^{-izt}u$, which solves \eqref{Duhamel} with $F=e^{-izt}(H-z)u$, implies
$$
\rho_z(T)\norm{|D|^{\sigma-1}u}_{L^{\frac{2n}{n-2}}}\lesssim \norm{u}_{L^2}+\rho_z(T)\norm{|D|^{1-\sigma}(H-z)u}_{L^{\frac{2n}{n+2}}}
$$
uniformly in $z\in \C$ and $T>0$, where $\rho_z(T):=\norm{e^{-izt}}_{L^2_T}\ge T^{1/2}$ since $|e^{\Im zt}|\ge1$ on either $[0,T]$ or $[-T,0]$. Hence, dividing by $\rho_z(T)$ and letting $T\to \infty$ yield \eqref{theorem_5_3}.
\end{proof}
\vskip0.3cm
\begin{remark}\label{remark_uniform_Sobolev}Although only the case $q=p'$ was considered, one can also show by the same argument that $R_H(z)$ satisfies the same estimates as \eqref{lemma_section_4_1_1} for $(1/p,1/q)$ belonging to the closed square (with its inside) having the line segment $\overline{A_1B_1}$ as a diagonal line (see Figure 1). However, it is far from the expected optimal range (see Remark \ref{remark_section_4_1_1}).\end{remark}
\vskip0.3cm
\section{Generalization to some dispersive operators}
\label{section_5}
This section discusses a generalization of the above results to a class of dispersive operators. We provide two types of examples: inhomogeneous elliptic operators and Schr\"odinger operators with variable coefficients.
\subsection{Inhomogeneous elliptic operator}
\label{subsection_5_1}
Let $0<\sigma<n/2$ and $P_0\in C(\R^n)\cap C^2(\R^n\setminus\{0\})$ be a non-negative symbol of order $2\sigma$. Suppose there exist $C_{1,\ell},C_{2,\ell},C>0$ such that
\begin{align}
\label{P_0_1}
C_{1,\ell}|\xi|^{2\sigma}\le  (\xi\cdot\nabla)^\ell P_0(\xi)\le C_{2,\ell}\<\xi\>^{2\sigma},\quad
|(\xi\cdot\nabla)^{2} P_0(\xi)|\le C(\xi\cdot\nabla)P_0(\xi)
\end{align}
for all $\xi\in \R^n$ and $\ell=0,1$. The following two examples are of particular interest:
\begin{itemize}
\item \underline{{\it Massive fractional Laplacian}}:
$$P_0(\xi)=(|\xi|^2+m)^\sigma,\ \ m>0. $$

\item \underline{{\it Sum of fractional Laplacians of different orders}}: $$P_0(\xi)=\sum_{j=1}^Ja_j|\xi|^{2\sigma_j}$$ where $J\in \N$, $0<\sigma_1<\sigma_2<...<\sigma_J=\sigma<n/2$, $a_j>0$.
\end{itemize}
\vskip0.3cm
Let $P_\ell(\xi)=(\xi\cdot\nabla)^\ell P_0(\xi)$ and $H_\ell=P_\ell(D)$ for $l=0,1,2$. Under the condition \eqref{P_0_1}, $H_0$ is self-adjoint on $L^2(\R^n)$ with domain $\H^{2\sigma}$, satisfying
\begin{align}
\label{P_0_3}
\<H_0u,u\>\gtrsim\<(-\Delta)^\sigma u,u\>,\ \<H_1u,u\>\gtrsim \<(-\Delta)^\sigma u,u\>,\ |\<H_2u,u\>|\lesssim \<H_1u,u\>
\end{align}
for $u\in C_0^\infty(\R^n)$. Under Assumption \ref{assumption_A} associated with these $H_\ell$ and $V_\ell=(-x\cdot\nabla_x)^\ell V$ for $\ell=1,2$, $H=H_0+V$ thus can be defined as the Friedrichs extension of the quadratic form $Q_H(u,v)=\<H_0u,v\>+\<Vu,v\>$ on $C_0^\infty(\R^n)$ such that $D(H^{1/2})=\H^{\sigma}(\R^n)$.

\begin{theorem}	
\label{theorem_section_5_1}
Let $n\ge2$, $0<\sigma<n/2$, $H_0=P_0(D)$ satisfy \eqref{P_0_1}, $V$ satisfy Assumption \ref{assumption_A} (associated with this $H_0$) and $H=H_0+V$. Then, for any $\sigma-n/2< \gamma<\sigma-1/2$, $H$ satisfies the same estimate as \eqref{theorem_1_1}, namely the following resolvent estimate holds:
\begin{align}
\label{theorem_section_5_1_1}
\sup_{z\in\C\setminus[0,\infty)}\norm{|x|^{-\sigma+\gamma}|D|^{\gamma}(H-z)^{-1}|D|^{\gamma}|x|^{-\sigma+\gamma}}_{L^2\to L^2}<\infty.
\end{align}
Moreover, the solution $\psi$ to \eqref{Cauchy_2} given by \eqref{Duhamel} (associated with this $H$) satisfies
\begin{align}
\label{theorem_section_5_1_2}
\norm{|x|^{-\sigma+\gamma}|D|^{\gamma}\psi}_{L^2_tL^2_x}\lesssim \norm{\psi_0}_{L^2_x}+\norm{|x|^{\sigma-\gamma}|D|^{-\gamma}F}_{L^2_tL^2_x}.
\end{align}
\end{theorem}

\begin{proof}
By Lemma \ref{lemma_section_2_1}, we have
$
[H_0,iA]=P_1(D)$ and $[[H_0,iA],iA]=P_2(D)$. Then, taking the inequalities \eqref{P_0_3} into account, one can use the completely same argument as that in Section \ref{section_2} to obtain
$$
\sup_{\delta>0}\sup_{z\in \C\setminus[0,\infty)}\norm{\<A\>^{-s}|D|^\sigma e^{-\delta|D|} (H-z)^{-1}e^{-\delta|D|}|D|^\sigma  \<A\>^{-s}}_{L^2\to L^2}<\infty
$$
and hence \eqref{theorem_section_5_1_1}. By Lemma \ref{lemma_section_2_6}, \eqref{theorem_section_5_1_1} implies \eqref{theorem_section_5_1_2}. This completes the proof.
\end{proof}
\vskip0.3cm
Note that one can also extend Corollary \ref{corollary_1} to the operator $H=P_0(D)+V$ by the same argument as in the last part of Section \ref{section_2}.
\vskip0.2cm
About Strichartz and uniform Sobolev estimates, we have seen, in Lemmas \ref{lemma_section_3_3} and \ref{lemma_section_4_2}, a general criterion  to deduce them for $H$ from corresponding estimates for $H_0$, the $H_0$-supersmoothness of $W_1$ and the $H$-supersmoothness of $W_2$, where $W_1,W_2\in L^{n/\sigma,\infty}$ satisfy $V=W_1W_2$. To apply this criterion to $H=P_0(D)+V$, one requires corresponding Strichartz or uniform Sobolev estimates for $P_0(D)$. Although these estimates have been extensively studied under various conditions on $P_0(\xi)$, we only focus for the sake of simplicity on the following higher-order inhomogeneous elliptic operator:
\begin{align}
\label{P_0_2}
P_0(D)=\sum_{j=1}^\sigma a_j(-\Delta)^{j},
\end{align}
where $\sigma\in \N$, $a_\sigma=1$ and $a_j\ge0$ for $1\le j\le \sigma-1$. For readers interested in these topics, we refer to \cite{KPV}, \cite{HHZ}, \cite{Guo}, \cite{CMY} and \cite{GLNY} for Strichartz estimates and \cite{SYY} and references therein for uniform Sobolev estimates. Note that, in the case with \eqref{P_0_2}, $V(x)=a|x|^{-2\sigma}$ fulfills the conditions in the following Theorem \ref{theorem_section_5_2} if $a>-C_{\sigma,n}$ since $P_\ell(D)\ge (2\sigma)^{2\ell}(-\Delta)^\sigma$.

\begin{theorem}	
\label{theorem_section_5_2}
Let $n\ge3$, $\sigma\in \N$, $\sigma<n/2$, $H_0=P_0(D)$ be given by \eqref{P_0_2}, $V$ satisfy Assumption \ref{assumption_A} associated with this $H_0$ and $|x|^\sigma V\in L^{n/\sigma,\infty}(\R^n)$. Then the following statements hold:
\begin{itemize}
\vskip0.3cm
\item \underline{Strichartz estimates}: for any two ${n}/{2}$-admissible pairs $(p,q)$ and $(\tilde p,\tilde q)$, one has \begin{align}
\label{theorem_section_5_2_1}
\norm{|D|^{{2(\sigma-1)}/{p}}\psi}_{L^p_tL^{q}_x}\lesssim \norm{\psi_0}_{L^2_x}+\norm{|D|^{{2(1-\sigma)}/\tilde p}F}_{L^{\tilde p'}_tL^{\tilde q'}_x}
\end{align}
where $\psi$ is given by \eqref{Duhamel} associated with the present $H$.
\vskip0.3cm
\item \underline{Uniform Sobolev estimates}: for any $2n/(n+2\sigma)\le p\le 2(n+1)/(n+3)$, one has
\begin{align*}
\norm{(H-z)^{-1}f}_{L^{p'}}\lesssim |z|^{\frac{n}{\sigma p}-\frac{n+2\sigma}{2\sigma}}\norm{f}_{L^{p}},\quad z\in \C\setminus \{0\},
\end{align*}
\end{itemize}
\end{theorem}

\begin{proof}[Proof of Theorems \ref{theorem_section_5_2}]
Let $W_1=|x|^{\sigma}V$ and $W_2=|x|^{-\sigma}$. Under the above conditions, the free evolution $e^{-itH_0}$ satisfies the same Strichartz estimates as \eqref{theorem_section_5_2_1} (see Appendix \ref{appendix_A} below). Moreover, we know by Theorem \ref{theorem_section_5_1} and the double endpoint Strichartz estimates for $e^{-itH_0}$ that $W_1$ is $H_0$-supersmooth and $W_2$ is $H$-supersmooth. Finally, the same uniform Sobolev estimates for $H_0=P_0(D)$ as in Lemma \ref{lemma_section_4_1} have been proved by \cite{SYY} and \cite{HuZh}. Therefore, Lemma \ref{lemma_section_3_3} and \ref{lemma_section_4_2} yield the desired results.
\end{proof}

\subsection{Schr\"odinger operator with variable coefficients}
Next, consider the following second order elliptic operator with variable coefficients in divergence form on $\R^n$:
\begin{align}
\label{divergence_form}
H_0:=-\nabla \cdot G_0(x)\nabla=-\sum_{j,k=1}^n\partial_{x_j}g^{jk}(x)\partial_{x_k},
\end{align}
where $G_0(x)=(g^{jk}(x))_{j,k=1}^n$, $g^{jk}\in C^2(\R^n;\R)$ and $G_0(x)$ is symmetric and uniformly elliptic (see the condition \eqref{theorem_section_5_3_1} below). Then $H_0$ is self-adjoint on $L^2(\R^n)$ with domain $\H^2$. Let $G_\ell(x)=[2-(x\cdot\nabla)]^\ell G_0(x)$. Then a direct computation yields that
$$
[H_0,iA]=-\nabla \cdot G_1(x)\nabla,\quad [[H_0,iA],iA]=-\nabla \cdot G_2(x)\nabla
$$
on $C_0^\infty(\R^n)$ and hence these commutators are still in divergence form. Therefore, the completely same argument as in Section \ref{section_2}  yields the following result.

\begin{theorem}	
\label{theorem_section_5_3}
Let $n\ge3$ and $H_0$ be given by \eqref{divergence_form}. Assume in addition to the above conditions on $G_0$ that $\partial_x^\alpha g^{jk}\in L^\infty(\R^n)$ for $|\alpha|\le 2$ and $G_0,G_1$ are uniformly elliptic in the sense that there exist $C_{1,\ell},C_{2,\ell}>0$ such that for all $x,\xi\in \R^n$,
\begin{align}
\label{theorem_section_5_3_1}
C_{1,\ell}|\xi|^2\le G_\ell(x)\xi\cdot\xi\le C_{2,\ell}|\xi|^2,\quad \ell=0,1.
\end{align}
Suppose also $V$ satisfies Assumption \ref{assumption_A} associated with $H_\ell=-\nabla \cdot G_\ell(x)\nabla$. Then $H=H_0+V$ satisfies the same estimates as \eqref{theorem_section_5_1_1} and \eqref{theorem_section_5_1_2} with $\sigma=1$ and $1-n/2<\gamma<1/2$.
\end{theorem}

We next consider Strichartz estimates. It was proved by \cite{MMT} that $e^{-itH_0}$ and $\Gamma_{H_0}$ satisfy Strichartz estimates for all $n/2$-admissible pairs under the following two conditions:
\vskip0.3cm
\begin{itemize}
\item {\it\underline{ Long-range condition}}: there exists $\mu>0$ such that $G_0(x)$ satisfies
$$
|\partial_x^\alpha (g^{jk}(x)-\delta_{jk})|\le C_\alpha \<x\>^{-\mu-|\alpha|},\quad x\in\R^n,\ |\alpha|\le2.
$$
\item {\it \underline{Non-trapping condition}}: the Hamilton flow $(x(t),\xi(t))$ generated by $G_0$ satisfies $|x(t)|\to \infty$ as $t\to \pm\infty$ for any initial data $(x(0),\xi(0))\in \R^n\times (\R^n\setminus\{0\})$.
\end{itemize}
\vskip0.3cm
This fact, Theorem \ref{theorem_section_5_3} and Lemma \ref{lemma_section_3_3} (2) then yield the following result:
\begin{theorem}	
\label{theorem_section_5_4}
Assume in addition to the conditions in Theorem \ref{theorem_section_5_3} that $G_0$ satisfies the above long-range and nontrapping conditions and that $|x|^2V\in L^\infty(\R^n)$. Then, for any non-endpoint $n/2$-admissible pairs $(p,q),(\tilde p,\tilde q)$ with $p,\tilde p>2$, $H=H_0+V$ satisfies
$$
\norm{\psi}_{L^p_tL^q_x}\lesssim \norm{\psi_0}_{L^2_x}+\norm{F}_{L^{\tilde p'}_tL^{\tilde q'}_x}
$$
where $\psi$ is given by the Duhamel formula \eqref{Duhamel} associated with this $H$.
\end{theorem}
\vskip0.3cm
There exists a sufficiently small $\ep>0$ such that if $G_0$ satisfies
$$
\sum_{j,k=1}^n\sum_{|\alpha|\le2}\<x\>^{\mu+|\alpha|}|\partial^\alpha (g^{jk}(x)-\delta_{jk})|\le \ep
$$
then all the conditions on $G_0$ in Theorem \ref{theorem_section_5_4} are satisfied (see \cite{MMT}). Moreover, in such a case, $V(x)=a|x|^{-2}$ satisfies the conditions in Theorem \ref{theorem_section_5_4} if $a>-(1-3\ep/2)[(n-2)/2]^2$ since $H_0\ge -(1-\ep)\Delta$ and $H_1\ge -(2-3\ep)\Delta$ in the sense of forms on $C_0^\infty(\R^n)$.
\vskip0.3cm
\begin{remark}
Strichartz estimates for Schr\"odinger equations with long-range metrics and decaying potentials have been extensively studied by many authors (see \cite{MMT} and reference therein). However, to the best of authors' knowledge, the potential $V(x)$ has been assumed to be $C^2$ and satisfy $V(x)=O(\<x\>^{-2}(\log\<x\>)^{-2})$ at least in the previous literatures. On the other hand, Theorem \ref{theorem_section_5_4} allows some scaling-critical potentials such as $V(x)=a|x|^{-2}$.
\end{remark}

\section{The critical constant case}
\label{section_6}
In this section we consider the critical constant case $a=-C_{\sigma,n}$, namely the operator
\begin{align}
\label{6_1}
H_{\mathrm{crit},\sigma}:=(-\Delta)^\sigma-C_{\sigma,n}|x|^{-2\sigma},\quad 0<\sigma<n/2,
\end{align}
where we recall that $C_{\sigma,n}$ given by \eqref{best}, is the best constant of Hardy's inequality \eqref{Hardy}. Since $(-\Delta)^\sigma-C_{\sigma,n}|x|^{-2\sigma}$ is non-negative, $H_{\mathrm{crit},\sigma}$ can be defined as its Friedrichs extension as in the subcritical case $(-\Delta)^\sigma+a|x|^{-2\sigma}$ with $a>-C_{\sigma,n}$. However, $H_{\mathrm{crit},\sigma}$ is not uniformly elliptic in contrast with the subcritical case, {\it i.e}, $(-\Delta)^\sigma\not\lesssim H_{\mathrm{crit},\sigma}$ in general. Hence its form domain $D(H_{\mathrm{crit},\sigma}^{1/2})$ is strictly larger than $\H^\sigma$ and $H_{\mathrm{crit},\sigma}$ is not the form-sum of $(-\Delta)^\sigma$ and $-C_{\sigma,n}|x|^{-2\sigma}$. Due to this fact, it is difficult to obtain uniform Sobolev or Strichartz estimates for $H_{\mathrm{crit},\sigma}$ by only applying the perturbation argument as above. Nevertheless, we show that, for functions orthogonal to radial functions, the same  results as in the subcritical case still hold. For the radial case, we also prove some new results for a critical operator related with $H_{\mathrm{crit},\sigma}$ and discuss a few known results and open problems for $H_{\mathrm{crit},\sigma}$.

We introduce a few notations. Let $L^2_{\mathrm{rad}}(\R^n)\cong L^2(\R_+,r^{n-1}dr)$ be the subspace of radial functions in $L^2$ and $P_{\mathrm{rad}}:L^2(\R^n)\to L^2_{\mathrm{rad}}(\R^n)$ and $P_{\mathrm{rad}}^\perp:L^2(\R^n)\to L^2_{\mathrm{rad}}(\R^n)^\perp$ defined by
$$
P_{\mathrm{rad}}f(x)=\frac{1}{\mathop{\mathrm{Vol}}(\mathbb S^{n-1})}\int_{\mathbb S^{n-1}}f(|x|\omega)d\omega,\quad P_{\mathrm{rad}}^\perp:=\Id-P_{\mathrm{rad}}.
$$
$P_{\mathrm{rad}}$ (resp. $P_{\mathrm{rad}}^\perp$) is the orthogonal projection onto $L^2_{\mathrm{rad}}(\R^n)$ (resp. $L^2_{\mathrm{rad}}(\R^n)^\perp$). Note that $P_{\mathrm{rad}}$ and $P_{\mathrm{rad}}^\perp$ leave $\H^\sigma$ invariant since they commute with $|D|^\sigma$. In particular, we can write  $\H^\sigma=(P_{\mathrm{rad}}\H^\sigma)\oplus(P_{\mathrm{rad}}^\perp\H^\sigma)$.
Moreover, $P_{\mathrm{rad}}$ and $P_{\mathrm{rad}}^\perp$ also commute with $H_{\mathrm{crit},\sigma}$ and hence with $f(H_{\mathrm{crit},\sigma})$ for any $f\in L^2_{\loc}(\R)$ by the spectral decomposition theorem. We will also use the orthogonal decomposition of $L^2(\R^n)$ via the spherical harmonics expansion:
\begin{align*}
L^2(\R^n)=\bigoplus_{\ell=0}^\infty \left(L^2_{\mathrm{rad}}(\R^n)\otimes \mathcal Y_\ell\right),
\end{align*}
where $\mathcal Y_\ell\subset L^2(\mathbb S^{n-1})$ is the subspace of spherical harmonics of order $\ell$ (see \cite{StWe} for details). Let $P_\ell$ be the orthogonal projection onto $L^2_{\mathrm{rad}}(\R^n)\otimes \mathcal Y_\ell$. Note that $P_{\mathrm{rad}}=P_0$ and hence
\begin{align}
\label{Decomposition}
L^2_{\mathrm{rad}}(\R^n)^\perp=\bigoplus_{\ell=1}^\infty \left(L^2_{\mathrm{rad}}(\R^n)\otimes \mathcal Y_\ell\right),\quad P_{\mathrm{rad}}^\perp=\sum_{\ell=1}^\infty P_\ell.
\end{align}

 \subsection{The result for the operator $H_{\mathrm{crit},\sigma}$}
 \label{subsection_6_1}
We first state the results for $H_{\mathrm{crit},\sigma}$.

\begin{theorem}	
\label{theorem_6_1}
Let $n\ge2$, $0<\sigma<n/2$ and $H_{\mathrm{crit},\sigma}$ be given by \eqref{6_1}. Then the following holds:
\begin{itemize}
\item Kato--Yajima estimates: if $\sigma-n/2<\gamma<\sigma-1/2$ then
\begin{align}
\label{theorem_6_1_KY}
\sup_{z\in\C\setminus[0,\infty)}\norm{|x|^{-\sigma+\gamma}|D|^{\gamma}P_{\mathrm{rad}}^\perp(H_{\mathrm{crit},\sigma}-z)^{-1}P_{\mathrm{rad}}^\perp|D|^{\gamma}|x|^{-\sigma+\gamma}}_{L^2\to L^2}<\infty.
\end{align}
\vskip0.2cm
\item Uniform Sobolev estimates: if $n\ge3$ and $n/(n+1)\le \sigma<n/2$, then
\begin{align}
\label{theorem_6_1_US}
\norm{(H_{\mathrm{crit},\sigma}-z)^{-1}P_{\mathrm{rad}}^\perp}_{L^p\to L^{p'}}\lesssim |z|^{\frac{n}{\sigma}(\frac{1}{p}-\frac{1}{2})-1},\quad z\in \C\setminus[0,\infty),
\end{align}
for any $2n/(n+2\sigma)\le p\le 2(n+1)/(n+3)$. 
\vskip0.2cm
\item Strichartz estimates: let $n\ge3$, $1/2<\sigma<n/2$ and $(p,q),(\tilde p,\tilde q)$ be $n/2$-admissible pairs. When $1/2<\sigma<1$, we further assume $p,\tilde p>2$. Then we have
\begin{align}
\label{theorem_6_1_Strichartz}
\norm{|D|^{{2(\sigma-1)}/{p}}P_{\mathrm{rad}}^\perp\psi}_{L^p_tL^{q}_x}\lesssim \norm{\psi_0}_{L^2_x}+\norm{|D|^{{2(1-\sigma)}/\tilde p}F}_{L^{\tilde p'}_tL^{\tilde q'}_x}
\end{align}
for the solution $\psi$ to $(i\partial_t-H_{\mathrm{crit},\sigma})\psi=F$ with $\psi|_{t=0}=\psi_0$ given by
$$
\psi(t,x)=e^{-itH_{\mathrm{crit},\sigma}}\psi_0(x)-i\int_0^t e^{-i(t-s)H_{\mathrm{crit},\sigma}}F(s,x)ds.
$$
\end{itemize}
\end{theorem}
\vskip0.2cm
\begin{remark}
\label{remark_6_1}
For the second order case $\sigma=1$, \eqref{theorem_6_1_US} and \eqref{theorem_6_1_Strichartz} were proved by \cite{Mizutani_JST} and \cite{Mizutani_JDE}, respectively. Otherwise, the results in Theorem \ref{theorem_6_1} are  new. As mentioned above, Theorem \ref{theorem_6_1} shows that essentially the same estimates as in the subcritical case still hold for the critical case if we restrict $(H_{\mathrm{crit},\sigma}-z)^{-1}$ and $e^{-itH_{\mathrm{crit},\sigma}}$ on the range of $P_{\mathrm{rad}}^\perp$.
\end{remark}

In the proof of this theorem, the following lemma plays a key role.
\begin{lemma}
\label{lemma_6_2}
Let $0<\sigma<n/2$. Then the following improved Hardy inequality holds:
\begin{align}
\label{improved_Hardy}
\widetilde C_{\sigma,n}\int |x|^{-2\sigma}|u(x)|^2dx\le \int ||D|^\sigma u(x)|^2dx,\quad \widetilde C_{\sigma,n}:=\left\{\frac{2^\sigma\Gamma\left(\frac{n+2\sigma+2}{4}\right)}{\Gamma\left(\frac{n-2\sigma+2}{4}\right)}\right\}^2,
\end{align}
for $u\in P_{\mathrm{rad}}^\perp \H^\sigma$. Moreover, we have $C_{\sigma,n}<\widetilde C_{\sigma,n}$.
\end{lemma}

\begin{proof}
For $\mathcal X\subset \H^\sigma$, let $C[\mathcal X]$ denote the best constant of Hardy's inequality in $\mathcal X$:
$$
C[\mathcal X]\int |x|^{-2\sigma}|u(x)|^2dx\le \int ||D|^\sigma u(x)|^2dx,\quad u\in \mathcal X.
$$
When $\mathcal X=P_\ell C_0^\infty(\R^n\setminus\{0\})$, it follows from \cite[Proposition 2.8]{Yafaev_JFA} that
$$
C[P_\ell C_0^\infty(\R^n\setminus\{0\})]=C_\ell:=\left\{\frac{2^\sigma\Gamma\left(\frac{n+2\sigma+2\ell}{4}\right)}{\Gamma\left(\frac{n-2\sigma+2\ell}{4}\right)}\right\}^2.
$$
Moreover, since $P_\ell$ commutes with $|D|^\sigma$, \eqref{Decomposition} implies that, for any $u\in P_{\mathrm{rad}}^\perp C_0^\infty(\R^n\setminus\{0\})$, $|D|^\sigma u$ can be decomposed as $|D|^\sigma u=|D|^\sigma P_1u\oplus |D|^\sigma P_2u\oplus\cdots$ in $L^2(\R^n)$. Hence,
$$
C[P_{\mathrm{rad}}^\perp C_0^\infty(\R^n\setminus\{0\})]=\inf_{\ell\ge 1}C_{\ell}.
$$
Since $C_\ell$ is strictly increasing in $\ell\in \N\cup\{0\}$ (see \cite[the discussion after (2.28)]{Yafaev_JFA}), we have $$\inf_{\ell\ge 1}C_{\ell}=C_1=\widetilde C_{\sigma,n}>C_0=C_{\sigma,n}$$ and \eqref{improved_Hardy} for $u\in P_{\mathrm{rad}}^\perp C_0^\infty(\R^n\setminus\{0\})$. The lemma then follows by the density argument.
\end{proof}

\begin{proof}[Proof of Theorem \ref{theorem_6_1}]
We first prove \eqref{theorem_6_1_KY}. By completely the same argument as in Section \ref{section_2}, \eqref{theorem_6_1_KY} is deduced from the following resolvent estimate:
\begin{align}
\label{theorem_6_1_KY_proof_1}
\sup_{\delta>0}\sup_{z\in \C\setminus[0,\infty)}\norm{\<A\>^{-s}|D|^\sigma e^{-\delta|D|} P_{\mathrm{rad}}^\perp(H_{\mathrm{crit},\sigma}-z)^{-1}P_{\mathrm{rad}}^\perp e^{-\delta|D|}|D|^\sigma  \<A\>^{-s}}<\infty.
\end{align}
Indeed, the proof of replacing $|x|^{-\sigma+\gamma}|D|^\gamma$ by $\<A\>^{-s}|D|^\sigma e^{-\delta|D|}$ given in Section \ref{section_2} is independent of the operator $H$ and does not use any property of the potential $V$.

In order to prove \eqref{theorem_6_1_KY_proof_1}, one can also use essentially the same argument as in  Section \ref{section_2} with $(H,V)$ replaced by $(H_{\mathrm{crit},\sigma},-C_{\sigma,n}|x|^{-2})$. Indeed, thanks to Lemma \ref{lemma_6_2}, $H=H_{\mathrm{crit},\sigma}$ satisfies \eqref{assumption_A_1}--\eqref{assumption_A_3} for all $u\in P_{\mathrm{rad}}^\perp C_0^\infty(\R^n)$. Therefore, if we consider a self-adjoint operator $H_{\mathrm{crit},\sigma}^\perp:=H_{\mathrm{crit},\sigma}P_{\mathrm{rad}}^\perp$ defined on the closed subspace $P_{\mathrm{rad}}^\perp L^2(\R^n)$ of $L^2(\R^n)$ with form domain $D((H_{\mathrm{crit},\sigma}^\perp)^{1/2})=P_{\mathrm{rad}}^\perp \H^\sigma$, then $H_{\mathrm{crit},\sigma}^\perp$ is subcritical in the sense that
\begin{align*}
\<H_{\mathrm{crit},\sigma}^\perp u,u\>&\gtrsim \<H_0u,u\>,\\
\<[H_{\mathrm{crit},\sigma}^\perp,iA]u,u\>&\gtrsim \<H_0u,u\>,\\
|\<[[H_{\mathrm{crit},\sigma}^\perp,iA],iA]u,u\>|&\lesssim \<[H_{\mathrm{crit},\sigma}^\perp,iA]u,u\>
\end{align*}
hold for $u\in P_{\mathrm{rad}}^\perp \H^\sigma$, where $H_0=(-\Delta)^\sigma$. Moreover, since $P_{\mathrm{rad}}^\perp$ commutes with the self-adjoint operators $|D|$ and $A$, $P_{\mathrm{rad}}^\perp$ also commutes with $f(|D|)$ and $f(A)$ for any $f\in L^2_{\loc}(\R)$ by the spectral decomposition theorem. Hence we can repeat the same argument as in the proof of Theorem \ref{theorem_section_2_1} (with $L^2,\H^\sigma$ and $C_0^\infty(\R^n)$ replaced by $P_{\mathrm{rad}}^\perp L^2,P_{\mathrm{rad}}^\perp \H^\sigma$ and $P_{\mathrm{rad}}^\perp C_0^\infty(\R^n)$, respectively) obtaining the following estimate:
\begin{align*}
\sup_{\delta>0}\sup_{z\in \C\setminus[0,\infty)}\norm{\<A\>^{-s}|D|^\sigma e^{-\delta|D|} (H_{\mathrm{crit},\sigma}^\perp-z)^{-1}e^{-\delta|D|}|D|^\sigma  \<A\>^{-s}}_{P_{\mathrm{rad}}^\perp L^2\to P_{\mathrm{rad}}^\perp L^2}<\infty.
\end{align*}
This estimate is nothing but \eqref{theorem_6_1_KY_proof_1} since $P_{\mathrm{rad}}^\perp$  commutes with $e^{-\delta|D|}|D|^\sigma  \<A\>^{-s}$ and $$(H_{\mathrm{crit},\sigma}^\perp-z)^{-1}P_{\mathrm{rad}}^\perp=P_{\mathrm{rad}}^\perp(H_{\mathrm{crit},\sigma}-z)^{-1}P_{\mathrm{rad}}^\perp.$$ This completes the proof of \eqref{theorem_6_1_KY}.
\vskip0.2cm
We next prove \eqref{theorem_6_1_US} and \eqref{theorem_6_1_Strichartz}. As above, we can replace $H_{\mathrm{crit},\sigma}$ by $H_{\mathrm{crit},\sigma}^{\perp}$ . Let $H_0^\perp:=H_0P_{\mathrm{rad}}^\perp$, $W_1=-C_{\sigma,n}|x|^{-\sigma}P_{\mathrm{rad}}^\perp$ and $W_2=|x|^{-\sigma}P_{\mathrm{rad}}^\perp$. Since $ P_{\mathrm{rad}}^\perp$ commutes with $|D|^\alpha$ for any $\alpha$ and $P_{\mathrm{rad}}^\perp\in \mathbb B(L^p)$ for any $p$, $(H_0^\perp-z)^{-1}$ and $e^{-itH_0^\perp}$ satisfy the same statements as in Lemmas \ref{lemma_section_4_1} and \ref{lemma_section_3_1}, respectively. Moreover, Lemma \ref{lemma_6_2} implies $$D((H_{\mathrm{crit},\sigma}^{\perp})^{1/2})=D((H_0^\perp)^{1/2})=D(W_j)=P_{\mathrm{rad}}^\perp\H^\sigma,\quad j=1,2.$$ Hence, for $u,v\in P_{\mathrm{rad}}^\perp\H^\sigma$, $\<H_{\mathrm{crit},\sigma}^{\perp}u,v\>$ can be written in the form:
$$
\<H_{\mathrm{crit},\sigma}^{\perp}u,v\>=\<H_0^\perp u,v\>+\<W_1u,W_2v\>.
$$
We also know from \eqref{theorem_6_1_KY} that $W_1,W_2$ are $H_{\mathrm{crit},\sigma}^{\perp}$-supersmooth if $\sigma>1/2$. Therefore, we can apply the same perturbation arguments based on Lemmas \ref{lemma_section_4_2} and \ref{lemma_section_3_3} as in Sections \ref{section_4} and \ref{section_3} to obtain \eqref{theorem_6_1_US} and \eqref{theorem_6_1_Strichartz}, respectively.
\end{proof}

\subsection{Some known results and open problems}
By virtue of Theorem \ref{theorem_6_1}, $H_{\mathrm{crit},\sigma}$ satisfies the same resolvent and Strichartz estimates as in the subcritical case except for the radial part. Here we discuss a few known results in the second order case $\sigma=1$ and open problems in the general case $\sigma\neq1$ for the radial part.

Let us first consider the second order case $\sigma=1$. Concerning uniform resolvent estimates, there seems to be no known results on the Kato--Yajima estimate. On the other hand, it was proved by \cite{Mizutani_JST} that the uniform Sobolev estimate
\begin{align}
\label{section_6_3}
\norm{(H_{\mathrm{crit},1}-z)^{-1}P_{\mathrm{rad}}}_{L^p\to L^{p'}}\lesssim |z|^{n(\frac1p-\frac12)-1},\quad z\in \C\setminus[0,\infty),
\end{align}
holds if and only if $2n/(n-2)<p\le 2(n+1)/(n+3)$. Concerning Strichartz estimates for $H_{\mathrm{crit},1}$, the non-endpoint estimates were proved by \cite{Suzuki}. Moreover, it was proved in \cite{Mizutani_JDE} that the following endpoint estimate
\begin{align}
\label{section_6_2}
\norm{e^{-itH_{\mathrm{crit},1}}P_{\mathrm{rad}}\psi_0}_{L^2_tL^{\frac{2n}{n-2},r}_x}\lesssim \norm{\psi_0}_{L^2_x}
\end{align}
holds if and only if $r=\infty$ (see  \cite{Mizutani_JDE}). In particular, the standard endpoint Strichartz estimate (with $r=\frac{2n}{n-2}$) does not hold. Note that \cite{Suzuki,Mizutani_JDE,Mizutani_JST} also considered the part on the range of $P_{\mathrm{rad}}^\perp$ not only the radial part.

An essential ingredient in the proof of \eqref{section_6_3} and \eqref{section_6_2} is that $H_{\mathrm{crit},1}$ is unitarily equivalent to $-\Delta_{\R^2}$ on the range of $P_{\mathrm{rad}}$. More precisely, we have
\begin{align*}
r^{\frac{n-2}{2}}\left(\frac{d^2}{dr^2}+\frac{n-1}{r}\frac{d}{dr}+\frac{(n-2)^2}{4r^2}\right)r^{-\frac{n-2}{2}}=\frac{d^2}{dr^2}+\frac1r\frac{d}{dr},\quad r>0,
\end{align*}
and hence $\Lambda_1^{-1}H_{\mathrm{crit},1}\Lambda_1f=-\Delta_{\R^2}f$ for radial function $f$, where $$\Lambda_\sigma f:=|x|^{-\frac{n-2\sigma}{2}}f.$$ Note that $\Lambda_1$ maps from $L^2_{\mathrm{rad}}(\R^2)$ to $L^2_{\mathrm{rad}}(\R^n)$. It is worth noting that $|x|^{-\frac{n-2}{2}}$ is a zero resonant state of $H_{\mathrm{crit},1}$, {\it i.e.}, a non-$L^2$ ground state of  $H_{\mathrm{crit},1}$. Precisely, $|x|^{-\frac{n-2}{2}}$ satisfies $|x|^{-\frac{n-2}{2}}\notin L^2$, $\<x\>^{-s}|x|^{-\frac{n-2}{2}}\in L^2$ if $s>1$ and $$H_{\mathrm{crit},1}|x|^{-\frac{n-2}{2}}=0,\quad x\neq0.$$ Hence $\Lambda_1^{-1} H_{\mathrm{crit},1}\Lambda_1$ is often called the ground state representation of $H_{\mathrm{crit},1}$ (\cite{FLS}). By virtue of the ground state representation, \eqref{section_6_3} and \eqref{section_6_2} for radial functions follow from the corresponding results for $-\Delta_{\R^2}$ (see \cite{Mizutani_JDE,Mizutani_JST} as well as Subsection \ref{subsection_6_3} below).
\vskip0.2cm
Next we consider the case $\sigma\neq1$ where, to the best of our knowledge, there is no existing result on the uniform resolvent and Strichartz estimates. The zero resonant state for $H_{\mathrm{crit},\sigma}$ is given by $|x|^{-\frac{n-2\sigma}{2}}$. However, the main difference with the case $\sigma=1$ is that the radial part of the ground state representation $\Lambda_\sigma^{-1} H_{\mathrm{crit},\sigma}\Lambda_\sigma$ is not a simple operator such as the fractional or higher-order Laplacians. For instance, for $\sigma=2$, we have
$$
\Lambda_2^{-1}H_{\mathrm{crit},2}\Lambda_2f
=\Delta_{\R^4}^2f-C_{2,n}^{1/2}(\Delta_{\R^4}|x|^{-2}+|x|^{-2}\Delta_{\R^4})f
$$
for radial function $f$ (see \cite{FLS} for the ground state representation of $H_{\mathrm{crit},\sigma}$ with $0<\sigma<1/2$). Hence it seems to be difficult to apply the above simple argument to the general case $\sigma\neq1$ at least directly. It would be interesting to investigate uniform Sobolev and Strichartz estimates for the operator $\Lambda_\sigma^{-1} H_{\mathrm{crit},\sigma}\Lambda_\sigma$ (restricted to radial functions) with $0<\sigma<n/2$ and $\sigma\neq1$, which will certainly play a crucial role in the proof of such estimates for the original operator $H_{\mathrm{crit},\sigma}$.

\subsection{Ground state representation and some results for a critical operator}
\label{subsection_6_3}
Here we shall explain more precisely the above method, based on the ground state representation, reducing the problem to the study of a simple operator without potential, which may be of independent interest. To this end, instead of $H_{\mathrm{crit},\sigma}$, we consider an operator whose radial part is unitarily equivalent to $(-\Delta_{\R^{2\sigma}})^\sigma$ via the ground state representation. Namely, we let $0<\sigma<n/2$ and $\sigma\in \N$ and consider the operator
$$
\widetilde H_{\mathrm{crit},\sigma}:=\Lambda_\sigma(-\Delta_{\R^{2\sigma}})^\sigma\Lambda_\sigma^{-1}\quad\text{on}\ \R^n,
$$
where note that $\Lambda_\sigma^{-1}=|x|^{\frac{n-2\sigma}{2}}:L^2_{\mathrm{rad}}(\R^n)\to L^2_{\mathrm{rad}}(\R^{2\sigma})$ and that, although it will not be used in the following argument, the radial part of this operator is explicitly given by
\begin{align*}
\widetilde H_{\mathrm{crit},\sigma}P_{\mathrm{rad}}
=\left(-\Delta_{\R^n}-\frac{(n-2\sigma)(n+2\sigma-4)}{4|x|^2}\right)^\sigma P_{\mathrm{rad}}.
\end{align*}
It is worth noting that the right hand side coincides with $H_{\mathrm{crit},1}P_{\mathrm{rad}}$ if $\sigma=1$. Then we have the following weak-type endpoint Strichartz estimate:

\begin{theorem}	
\label{theorem_6_2}
Let $n\ge3$, $0<\sigma<n/2$ and $\sigma\in \N$. Then, for any radial $\psi_0\in L^2_{\mathrm{rad}}(\R^n)$, one has
$$
\norm{e^{-it\widetilde H_{\mathrm{crit},\sigma}}\psi_0}_{L^2_tL^{\frac{2n}{n-2\sigma},\infty}_x}\lesssim \norm{\psi_0}_{L^2_x}.
$$
\end{theorem}

\begin{proof}
The definition of $\widetilde H_{\mathrm{crit},\sigma}$ implies
$
e^{-it\widetilde H_{\mathrm{crit},\sigma}}P_{\mathrm{rad}}=\Lambda_\sigma e^{it(-\Delta_{\R^{2\sigma}})^\sigma}\Lambda_\sigma^{-1}P_{\mathrm{rad}}
$ and hence
\begin{align*}
\norm{e^{-it\widetilde H_{\mathrm{crit},\sigma}}\psi_0}_{L^2(\R;L^{\frac{2n}{n-2\sigma},\infty}(\R^n))}
&=\norm{|x|^{-\frac{n-2\sigma}{2}}e^{-it(-\Delta_{\R^{2\sigma}})^\sigma}\Lambda_\sigma^{-1}\psi_0}_{L^2(\R;L^{\frac{2n}{n-2\sigma},\infty}(\R^n))}\\
&\lesssim \norm{|x|^{-\frac{n-2\sigma}{2}}}_{L^{\frac{2n}{n-2\sigma},\infty}(\R^n)}\norm{e^{-it(-\Delta_{\R^{2\sigma}})^\sigma}\Lambda_\sigma^{-1}\psi_0}_{L^2(\R;L^\infty_{\mathrm{rad}}(\R^n))}\\
&\lesssim \norm{e^{-it(-\Delta_{\R^{2\sigma}})^\sigma}\Lambda_\sigma^{-1}\psi_0}_{L^2(\R;L^\infty_{\mathrm{rad}}(\R^{2\sigma}))}
\end{align*}
for $\psi_0\in L^2_{\mathrm{rad}}(\R^n)$, where we have used H\"older's inequality \eqref{Holder} and the fact $L^\infty_{\mathrm{rad}}(\R^n)\cong L^\infty_{\mathrm{rad}}(\R^{2\sigma})$. Here we know from \cite[Theorem 1.2 (1)]{GLNY} (with $(d,a,s)=(2\sigma,2\sigma,0)$ in their notation) that $e^{it(-\Delta_{\R^{2\sigma}})^\sigma}$ satisfies the endpoint Strichartz estimate
$$
\norm{e^{-it(-\Delta_{\R^{2\sigma}})^\sigma}f}_{L^2(\R;L^\infty(\R^{2\sigma}))}\lesssim \norm{f}_{L^2(\R^{2\sigma})}
$$
for any radial $f\in L^2(\R^{2\sigma})$. This, together with the above computations, implies the desired estimate since $\Lambda_\sigma^{-1}:L^2_{\mathrm{rad}}(\R^n)\to L^2_{\mathrm{rad}}(\R^{2\sigma})$.
\end{proof}

\begin{remark}
We hope that the above method could be used in the study of Strichartz estimates for more general Schr\"odinger type operators with zero resonances than $\widetilde H_{\mathrm{crit},\sigma}$. 
\end{remark}

\begin{remark}Using a similar method as in \cite{Mizutani_JST} again based on the ground state representation, it might be also possible to obtain uniform Sobolev estimates for $\widetilde H_{\mathrm{crit},\sigma}$ if $2n/(n+2\sigma)<p\le 2(n+1)/(n+3)$ under the assumption that the uniform Sobolev estimates for $(-\Delta_{\R^{2\sigma}})^{\sigma}$ hold. However, except for the case $\sigma=1$, such estimates  for $(-\Delta_{\R^{2\sigma}})^{\sigma}$ seem to be not known yet. We do not pursue this topic for simplicity.
\end{remark}

\appendix

\section{Strichartz estimates for the free evolution}
\label{appendix_A}
Here we prove Lemmas \ref{lemma_section_3_1} and \ref{lemma_section_3_2}. For Lemma \ref{lemma_section_3_1}, we in fact prove the following more general result to include the operators $P_0(D)$ considered in Subsection \ref{subsection_5_1}.

\begin{lemma}
\label{lemma_A_1}
Assume either that $H_0=(-\Delta)^\sigma$ with $\sigma>0$ and $\sigma\neq1/2$ or that $H_0=P_0(D)$ is given by \eqref{P_0_2} with $\sigma\in \N$ and $\sigma<n/2$. Then \eqref{lemma_section_3_1_1} and \eqref{lemma_section_3_1_2} are satisfied.
\end{lemma}
\vskip0.2cm
The proof of this lemma relies on the Keel--Tao theorem \cite{KeTa}, the Littlewood--Paley square function estimates for $\varphi_j(D)$ given in Subsection \ref{subsection_notation} and the following localized dispersive estimate (with the implicit constant independent of $t$ and $j$):
\begin{align}
\label{lemma_A_1_1}
\norm{e^{-itH_0}\Phi_j(D)}_{L^1\to L^\infty}\lesssim 2^{-jn(\sigma-1)}|t|^{-n/2},\quad t\neq0,\ j\in \Z,
\end{align}
where $\Phi_j(\xi)=\Phi(2^{-j}\xi)$ and $\Phi\in C_0^\infty(\R^n)$ is supported away from the origin.
\eqref{lemma_A_1_1} is easy to obtain if $H_0=(-\Delta)^\sigma$. For $H_0$ given by \eqref{P_0_2}, an essential ingredient for proving \eqref{lemma_A_1_1} is the following decay estimate for the convolution kernel $I(t,x)=\F^{-1}(e^{-itP_0})(x)$:
\begin{align}
|\partial_x^\alpha I(t,x)|
\label{lemma_A_1_2}
\lesssim
\begin{cases}
|t|^{-\frac{n+|\alpha|}{2\sigma}}\<|t|^{-\frac{1}{2\sigma}}x\>^{-\frac{n(\sigma-1)-|\alpha|}{2\sigma-1}}
&\text{if}\ |t|\lesssim 1,\ \text{or}\ |t|\gtrsim 1\ \text{and}\ |t|^{-1}|x|\gtrsim1,\\
|t|^{-\frac{n+|\alpha|}{2\delta }}\<|t|^{-\frac{1}{2\delta }}x\>^{-\frac{n(\delta -1)-|\alpha|}{2\delta -1}}
&\text{if}\ |t|\gtrsim 1\ \text{and}\ |t|^{-1}|x|\lesssim1,\end{cases}
\end{align}
where $|\alpha|\le n(\sigma-1)$, $\delta=\min\{j\ |\ a_j>0\}\le \sigma$. This bound follows from \cite[Theorem 3.1]{HHZ} by taking $m_2=\sigma,m_1=\delta$ in this theorem (see also \cite{CMY} and references therein).
\vskip0.3cm
\begin{proof}[\underline{Proof of Lemma \ref{lemma_A_1}}] The proof is divided into three steps.
\vskip0.2cm
{\it Step 1}.
We first prove that the following square function estimates for the Littlewood--Paley decomposition $\{\varphi_j(D)\}_{j\in \Z}$ hold:
\begin{align}
\label{LP_1}
\norm{f}_{L^{q,2}}&\lesssim \Big(\sum_{j\in \Z}\|\varphi_j(D)f\|_{L^{q,2}}^2\Big)^{1/2},\quad 2\le q<\infty,\\
\label{LP_2}
\norm{f}_{L^{q,2}}&\gtrsim \Big(\sum_{j\in \Z}\|\varphi_j(D)f\|_{L^{q,2}}^2\Big)^{1/2},\quad 1<q\le2.
\end{align}
These estimates can be easily obtained by the usual estimates and the real interpolation theorem. Indeed, if we set $Sf(j,x):=\varphi_j(D)f(x)$, then the usual square function estimate
$$
\norm{Sf}_{L^2_jL^q_x}\lesssim \norm{f}_{L^q_x}
$$
holds for $1<q\le2$. Then the real interpolation theorem (see Appendix \ref{appendix_misc}) implies \eqref{LP_2}. \eqref{LP_1} follows from \eqref{LP_2} and a duality argument. Since $e^{-itH_0}$ and $\Gamma_{H_0}$ commute with $\varphi_j(D)$, by virtue of these estimates \eqref{LP_1} and \eqref{LP_2}, we may assume without loss of generality that $\psi_0=\Phi_j(D)\psi_0,F=\Phi_j(D)F$ with some $\Phi\in C_0^\infty(\R^n)$ supported away from the origin so that $\Phi\equiv1$ on $\supp \varphi$.
\vskip0.2cm
{\it Step 2}. We next prove the dispersive estimate \eqref{lemma_A_1_1}. Suppose first $H_0=(-\Delta)^\sigma$ and $\sigma\neq1/2$. Since $|\mathrm{Hess}(|\xi|^{2\sigma})|\sim 1$ on $\supp\Phi$, the stationary phase theorem yields
$$
\|e^{-itH_0}\Phi(D)f\|_{L^\infty}\lesssim |t|^{-n/2}\|f\|_{L^1},\quad t\neq0,
$$
which implies \eqref{lemma_A_1_1} by scaling $f(x)\mapsto f(2^jx)$ and  the fact that $(-\Delta)^\sigma$ is homogeneous of order $2\sigma$. We next let $\sigma\in \N$, $|\alpha|=n(\sigma-1)$ and $H_0$ given by \eqref{P_0_2}. When $\delta=\sigma$, we have $$|\partial^\alpha I(t,x)|\lesssim |t|^{-n/2},\quad t\neq0,$$ since $(n+n(\sigma-1))/(2\sigma)=n/2$. When $\delta<\sigma$, we similarly have $|\partial^\alpha I(t,x)|\lesssim |t|^{-n/2}$ for the former case in \eqref{lemma_A_1_2}. For the latter case, since  $|x|\lesssim |t|$, $|\partial^\alpha I(t,x)|\lesssim |t|^{-\delta(n,\sigma,\delta)}$ where
$$
\delta(n,\sigma,\delta)=\frac{\sigma n}{2\delta }-\left(1-\frac{1}{2\delta }\right)\frac{n(\sigma-\delta )}{2\delta -1}=\frac{\sigma n-n\sigma+n\delta }{2\delta }=\frac n2.
$$
Therefore, we obtain
\begin{align}
\label{lemma_A_1_proof_1}
\norm{|D|^{n(\sigma-1)}e^{-itH_0}}_{L^1\to L^\infty}\lesssim |t|^{-n/2},\quad t\neq0.
\end{align}
which, together with the bound $\norm{|2^{-j}D|^{n(1-\sigma)}\Phi_j(D)}_{L^\infty\to L^\infty}\lesssim 1$, implies \eqref{lemma_A_1_1}.
\vskip0.2cm
{\it Step 3}. Now we recall Keel--Tao's theorem \cite[Theorem 10.1]{KeTa} which, in a special case, states that if a family of operators $\{U(t)\}_{t\in \R}\subset \mathbb B(L^2(\R^n))$ satisfies
\begin{itemize}
\item $\norm{U(t)}\lesssim 1$ uniformly in $t\in \R$;
\item $\norm{U(t)U(s)^*}_{L^1\to L^\infty}\lesssim |t-s|^{-\alpha}$ for $t\neq s$ with some $\alpha>0$,
\end{itemize}
then, for any $\alpha$-admissible pairs $(p,q)$ and $(\tilde p,\tilde q)$, one has
$$
\norm{U(t)\psi_0}_{L^p_tL^{q,2}_x}\lesssim \norm{\psi_0}_{L^2_x},\quad \bignorm{\int_0^tU(t)U(s)^*F(s)ds}_{L^p_tL^{q,2}_x}\lesssim \norm{F}_{L^{\tilde p'}_tL^{\tilde q',2}_x}.
$$
By Bernstein's inequality \eqref{Bernstein} and \eqref{lemma_A_1_1}, we have
\begin{align*}
\norm{\Phi_j(D)e^{-itH_0}\Phi_j(D)}_{L^1\to L^\infty}\lesssim 2^{-jn(\sigma-1)}|t|^{-n/2}.
\end{align*}
Putting $N_j=2^{-2j(\sigma-1)}$ and making the change of variable $t\mapsto N_jt$, we have
$$
\norm{\Phi_j(D)e^{-iN_j(t-s)H_0}\Phi_j(D)}_{L^1\to L^\infty}\lesssim |t-s|^{-n/2}.
$$
By the unitarity of $e^{-itN_jH_0}$ we also obtain $\norm{\Phi_j(D)e^{-iN_j t H_0}}\lesssim1$. Therefore, one can apply the above Keel--Tao theorem to $U(t)=\Phi_j(D)e^{-iN_j t H_0}$ obtaining
\begin{align*}
\norm{e^{-iN_j t H_0}\Phi_j(D)\psi_0}_{L^{p}_tL^{q,2}_x}&\lesssim \norm{\psi_0}_{L^2_x},\\
\bignorm{\Phi_j(D)\int_0^te^{-iN_j(t-s)H_0}\Phi_j(D)F(s)ds}_{L^{p}_tL^{q,2}_x}&\lesssim\norm{F}_{L^{\tilde p'}_tL^{\tilde q',2}_x}.
\end{align*}
By rescaling $t\mapsto N_j^{-1}t$ and $s\mapsto N_j^{-1}s$ and using \eqref{Bernstein}, we have \eqref{lemma_section_3_1_1} and \eqref{lemma_section_3_1_2} for $\psi_0,F$ replaced by $\Phi_j(D)\psi_0,\Phi_j(D)F$. Thanks to \eqref{LP_1} and \eqref{LP_2}, we obtain \eqref{lemma_section_3_1_1} and \eqref{lemma_section_3_1_2}.
\end{proof}

\begin{proof}[\underline{Proof of Lemma \ref{lemma_section_3_2}}]
Under the conditions in Lemma \ref{lemma_section_3_2}, it has been proved by \cite{GLNY} that
\begin{align}
\label{lemma_section_3_2_proof_1}
\norm{e^{-itH_0}\Phi(D)\psi_0}_{L^{p_1}_t\mathcal L^{q_1}_rL^2_\omega}&\lesssim \norm{\psi_0}_{L^2_x},\\
\label{lemma_section_3_2_proof_2}
\norm{\Gamma_{H_0}\Phi(D)F}_{L^{2}_t\mathcal L^{q_1}_rL^2_\omega}&\lesssim \|F\|_{L^{2}_t\mathcal L^{q_2'}_rL^2_\omega},
\end{align}
where $\Gamma_{H_0}F(t)=\int_0^t e^{-i(t-s)H_0}F(s)ds$. Since $\varphi_j(D)=\Phi_j(D)\varphi_j(D)$, the estimate \eqref{lemma_section_3_2_proof_1} and the same scaling argument as above then imply that
\begin{align*}
\norm{|D|^{s_1}e^{-itH_0}\varphi_j(D)\psi_0}_{L^{p_1}_t\mathcal L^{q_1}_rL^2_\omega}\lesssim 2^{-js_1}\norm{|D|^{s_1}\varphi_j(D)\psi_0}_{L^2_x}
\lesssim \norm{\varphi_j(D)\psi_0}_{L^2_x}.
\end{align*}
Since $p_1>2$, using Minkowski's inequality and this estimate, we have
\begin{align*}
\||D|^{s_1}e^{-itH_0}\psi_0\|_{L^{p_1}_tB[\mathcal L^{q_1}_rL^2_\omega]}^2
\lesssim  \sum_{j\in \Z}\|\varphi_j(D)\psi_0\|_{L^2_x}^2
\lesssim \|\psi_0\|_{L^2_x}^2
\end{align*}
and \eqref{lemma_section_3_2_1} follows. Next, since $\norm{f}_{\mathcal L^q_rL^2_\omega}\sim \norm{f}_{L^q_x}$ under the radial symmetry, we see that \eqref{lemma_section_3_2_proof_1}, \eqref{lemma_section_3_2_proof_2} and the real interpolation theory (see Appendix \ref{appendix_misc}) imply
$$
\norm{e^{-itH_0}\Phi(D)\psi_0}_{L^2_tL^{q_1,2}_x}\lesssim \norm{\psi_0}_{L^2_x},\quad \norm{\Gamma_{H_0}\Phi(D)F}_{L^{2}_tL^{q_1,2}_x}\lesssim \norm{F}_{L^{2}_tL^{q_2',2}_x}
$$
for radially symmetric data $\psi_0,F$. The same scaling argument as above then yields
\begin{align*}
\norm{|D|^{s_1}e^{-itH_0}\varphi_j(D)\psi_0}_{L^2_tL^{q_1,2}_x}&\lesssim \norm{\varphi_j(D)\psi_0}_{L^2_x},\\
\norm{|D|^{s(2,q_1)}\Gamma_{H_0}\varphi_j(D)F}_{L^{2}_tL^{q_1,2}_x}&\lesssim \norm{\varphi_j(D)|D|^{-s(2,q_2)}F}_{L^{2}_tL^{q_2',2}_x},
\end{align*}
which, together with \eqref{LP_1} and \eqref{LP_2}, implies the desired estimate \eqref{lemma_section_3_2_2}.
\end{proof}
\vskip0.3cm

\section{Proof of Example \ref{example_1}}
\label{appendix_example}
Let $H_0=P_0(D)$ be given by $P_0(\xi)=\sum_{j=1}^Ja_j|\xi|^{2\sigma_j}$, where $J\in \N$, $0<\sigma_1<\sigma_2<...<\sigma_J=\sigma<n/2$, $a_j\ge0$ and $a_J=1$.
Recall that, in such a case, $H_\ell$ are given by
$$
H_\ell=(2\sigma)^{\ell}(-\Delta)^\sigma+\sum_{j=1}^{J-1}(2\sigma_j)^{\ell}a_j(-\Delta)^{\sigma_j},\quad \ell=0,1,2.
$$
Here we show that the conditions in Example \ref{example_1} implies Assumption \ref{assumption_A} associated with these $H_\ell$. Firstly, \eqref{example_1_1} is just a paraphrase of \eqref{assumption_A_1}. Secondly, we use \eqref{example_1_2} and the condition $a_1,...,a_{m-1}\ge0$ to obtain \eqref{assumption_A_2}, namely
\begin{align}
\label{example_1_4}
\<(H_1+V_1)u,u\>\ge \<(2\sigma(-\Delta)^\sigma+V_1)u,u\>\gtrsim\<(-\Delta)^\sigma u,u\>.
\end{align}
Finally, writing
$$
H_2+V_2=2\sigma\Big(2\sigma (-\Delta)^\sigma+\sum_{j=1}^{J-1}\frac{\sigma_j}{\sigma}2\sigma_ja_j(-\Delta)^{\sigma_j} +V_1\Big)-2\sigma V_1+V_2
$$
and using the fact $\sigma_ja_j\le \sigma a_j$, we have
$$|\<(H_2+V_2)u,u\>|\le \<(2\sigma(H_1+V_1)u,u\>+|\<(2\sigma V_1-V_2)u,u\>|.$$
This bound, together with \eqref{example_1_3} and the first inequality in \eqref{example_1_4}, implies \eqref{assumption_A_3}.

\section{Some supplementary materials from Harmonic Analysis}
\label{appendix_misc}
Here we record several materials from Harmonic Analysis used frequently in the paper. We refer to textbooks \cite{Gra} and \cite{BeLo} for details.
\vskip0.2cm
(i) \underline{{\it Real interpolation space and theorem}}. Let $(X_1,X_2)$ be a Banach couple, {\it i.e.},  $X_1,X_2$ are two Banach spaces continuously embedded into a Hausdorff topological vector space. For $0<\theta<1$ and $1\le q\le \infty$, the real interpolation space $X_{\theta,q}=(X_1,X_2)_{\theta,q}$ is a Banach space satisfying $X_0\cap X_1\subset X_{\theta,q}\subset X_0+X_1$, $X_{\theta,q}=X_0$ if $X_0=X_1$, $X_{\theta,q}=X_{1-\theta,q}$ and
\begin{align}
\label{embedding}
X_{\theta,1}\hookrightarrow X_{\theta,q_1}\hookrightarrow X_{\theta,q_2}\hookrightarrow X_{\theta,\infty},\quad 1<q_1\le q_2<\infty.
\end{align}
Let $(X_0,X_1)$ and $(Y_0,Y_1)$ be two Banach couples and $T$ be a bounded linear operator from $(X_0,X_1)$ to $(Y_0,Y_1)$ in the sense that $T:X_j\to Y_j$ and $\norm{T}_{X_j\to Y_j}\le M_j$ for $j=0,1$. Then $T$ extends to a bounded operator from $X_{\theta,q}$ to $Y_{\theta,q}$ satisfying $\norm{T}_{X_{\theta,q}\to Y_{\theta,q}}\le M_0^{1-\theta}M_1^\theta$.
\vskip0.2cm
(ii) \underline{{\it Lorentz space}}. Let $(M,d\mu)$ be a $\sigma$-finite measure space. The Lorentz space $L^{p,q}=L^{p,q}(M,d\mu)$ is realized as a real interpolation between Lebesgue spaces, namely $L^{p_\theta,q}=(L^{p_0},L^{p_1})_{\theta,q}$ where $1\le p_0<p_1\le \infty$, $1\le q\le \infty$, $1/p_\theta=(1-\theta)/p_0+\theta/p_1$ and $0<\theta<1$. By \eqref{embedding}, we have the following continuous embeddings:
$$
L^{p,1}\hookrightarrow L^{p,q_1}\hookrightarrow L^{p,p}=L^p\hookrightarrow L^{p,q_2}\hookrightarrow L^{p,\infty},\quad 1\le q_1\le p\le q_2\le \infty.
$$ Moreover, for $1<p,q<\infty$, we have $(L^{p,q})'=L^{p',q'}$ and
$
\norm{f}_{L^{p,q}}\sim \sup_{\norm{g}_{L^{p',q'}}=1}\left|\int fgdx\right|
$.
Finally, the following O'Neil inequality (H\"older's inequality for Lorentz norms) holds:
\begin{align}
\label{Holder}
\norm{fg}_{L^{p,q}}\lesssim \norm{f}_{L^{p_1,q_1}}\norm{g}_{L^{p_2,q_2}},\quad
\norm{fg}_{L^{p,q}}\lesssim \norm{f}_{L^{\infty}}\norm{g}_{L^{p,q}}
\end{align}
where $1\le p,p_1,p_2<\infty$, $1\le q,q_1,q_2\le \infty$, $1/p_1+1/p_2=1/p$ and $1/q_1+1/q_2=1/q$. \vskip0.2cm
(iii) \underline{{\it Bochner space}}. Given a Banach space $X$ and $1\le p\le \infty$, the Bochner space $L^pX=L^p(M,d\mu;X)$ is defined by the norm $\|f\|_{L^pX}=\|\|f\|_X\|_{L^p}$. For any Banach couple $(X_0,X_1)$, $0<\theta<1$, $1<p_0\le p_1<\infty$, the real interpolation space between $L^{p_0}X_0$ and $L^{p_1}X_1$ with the second exponent $q=p_\theta$ is given by $(L^{p_0}X_0,L^{p_1}X_1)_{\theta,p_\theta}=L^{p_\theta}X_{\theta,p_\theta}$. In particular, $(L^2_tL^{q_0}_x,L^2_tL^{q_1}_x)_{\theta,2}=L^2_tL^{q_\theta,2}_x$ for $1<q_0<q_1<\infty$ and $0<\theta<1$.
Note that $(L^{p_0}X_0,L^{p_1}X_1)_{\theta,q}$ is not necessarily equal to $L^{p_\theta}X_{\theta,q}$ if $q\neq p_\theta$.
\vskip0.2cm
(iv) \underline{{\it Sobolev's inequality}}. If $1<p<q<\infty$, $1<s<n$ and $1/p-1/q=s/n$, then
\begin{align}
\label{Sobolev}
\norm{f}_{L^{q,2}(\R^n)}\lesssim \norm{|D|^sf}_{L^{p,2}(\R^n)}.
\end{align}
This inequality follows from the Hardy-Littlewood-Sobolev inequality $|D|^{-s}:L^p\to L^q$ and the real interpolation theorem.
\vskip0.2cm
(v)\underline{ {\it Bernstein's inequality}}. Let $\varphi\in C_0^\infty(\R^n)$ be supported away from the origin. Then, for all $1\le p\le q\le \infty$, $\varphi_j(D)=\varphi(2^{-j}D)$ satisfies
\begin{align}
\label{Bernstein}
\norm{\varphi_j(D)}_{L^{p,2}\to L^{q,2}}\lesssim 2^{-jn(1/q-1/p)},\quad j\in \Z,
\end{align}
with $L^{r,2}$ replaced by $L^r$ if $r=1,\infty$. Since $\varphi(D)f=\check{\varphi}*f$, the special case $j=0$ follows from Young's convolution inequality  and real interpolation theorem.
By virtue of the scaling $f(x)\mapsto f(2^jx)$, the general cases also follow from the case $j=0$.

\vskip0.2cm
(vi)\underline{ {\it Christ--Kiselev's lemma}}. Let $-\infty\le a<b\le \infty$, $\X,\Y$ be Banach spaces of functions on $\R^n$ so that $\X\cap L^2$ is dense in $\X$, and $\{K(t,s)\}_{t,s\in (a,b)}\subset \mathbb B(L^2)$ be  such that $K:L^2\to C((a,b)^2;L^2)$. Define an integral operator $T$ with the operator valued kernel $K$ by
$$
TF(t)=\int_a^bK(t,s)F(s)ds.
$$
Assume $TF(t)\in \Y$ for a.e. $t\in (a,b)$ and there exist $1\le p<q\le\infty$ and $C>0$ such that
\begin{align}
\label{C_5}
\|TF\|_{L^q((a,b);\Y)}\le C\|F\|_{L^p((a,b);\X)}
\end{align}
for any simple function $F:(a,b)\to L^2\cap \X$. Then the operator
$$
\tilde Tf(t)=\int_a^tK(t,s)F(s)ds
$$
satisfies the following estimate  for the same $p,q$:
$$
\|\tilde TF\|_{L^q((a,b);\Y)}\le \tilde C\|F\|_{L^p((a,b);\X)},
$$
where $\tilde C=C2^{1-2(1/p-1/q)}(1-2^{-(1/p-1/q)})^{-1}$. Note that the condition $p<q$ is necessary since $\tilde C\to \infty$ as $p\to q$. This is a minor modification of \cite[Lemma 3.1]{SmSo} (see also the original paper \cite{ChKi}) where  the condition $K\in C(\R^2;\mathbb B(\X,\Y))$ was assumed to define $T,\tilde T$ on $C_t\X\cap L^1_t\X$. In the present setting, the above assumption is sufficient to define  $T,\tilde T$ on $C_t(\X\cap L^2)\cap L^1_t(\X\cap L^2)$ and the  same proof as that of \cite[Lemma 3.1]{SmSo} works well to obtain the above statement. Such a modification is useful when one considers the case with $K(t,s)=e^{-i(t-s)H}$ to prove inhomogeneous Strichartz estimates for $\tilde T=\Gamma_H$ by using the corresponding homogeneous Strichartz estimates for $e^{-itH}$, since $e^{-i(t-s)H}:L^2_x\to C(\R^2;L^2_x)$ for any self-adjoint operator $H$ on $L^2$, while it is not always true that $e^{-itH}: \X\to \Y$ for each $t$ unless $\X=\Y=L^2$. Moreover, the condition that $TF(t)\in \Y$ for a.e. $t$ follows from the corresponding homogeneous Strichartz estimates for $e^{-itH}$. 

\section*{Acknowledgments}
H. Mizutani is partially supported by JSPS KAKENHI Grant-in-Aid for Young Scientists (B) \#JP17K14218 and Grant-in-Aid for Scientific Research (B) \#JP17H02854. X. Yao is partially supported by NSFC grants No.11771165 and 12171182. H. Mizutani would like to express his thanks to Professors Avy soffer and Xiaohua Yao for their invitation to visit Wuhan and their kind hospitality at CCNU. Finally, the authors would like to thank the anonymous referees for careful reading the manuscript and providing valuable suggestions, which substantially helped improving the quality of the paper.



\begin{thebibliography}{99}

\bibitem{BeLo} J. Bergh, J. L\"ofstr\"om, {\it Interpolation spaces. An introduction}, Springer- Verlag, Berlin, 1976, Grundlehren der Mathematischen Wissenschaften, No. 223.

\bibitem{BoMi} J. -M. Bouclet, H. Mizutani, {\it Uniform resolvent and Strichartz estimates for Schr\"odinger equations with critical singularities}, Trans. Amer. Math. Soc. \textbf{370} (2018), 7293--7333.

\bibitem{BPST1} N. Burq, F. Planchon, J. G. Stalker, A. S. Tahvildar-Zadeh, {\it Strichartz estimates for the wave and Schr\"odinger equations with the inverse-square potential}, J. Funct. Anal. \textbf{203} (2003), 519--549.

\bibitem{BPST2} N. Burq, F. Planchon, J. G. Stalker, A. S. Tahvildar-Zadeh, {\it Strichartz estimates for the wave and Schr\"odinger equations with potentials of critical decay}, Indiana Univ. Math. J. \textbf{53}  (2004), 1665--1680.

\bibitem{Carles} R. Carles, W. Lucha, E. Moulay, {\it High order Schr\"odinger and Hartree-Fock equations}, J. Math. Phys \textbf{56} (2015), 122301.

\bibitem{CMY} W. Chen, C. Miao,  X. Yao, {\it  Dispersive estimates with geometry of finite type}. Comm. Partial Differential Equations \textbf{37} (2012), 479--510.

\bibitem{ChKi} M. Christ, A. Kiselev, {\it Maximal functions associated to filtrations}, J. Funct. Anal. \textbf{179} (2001), 409--425.

\bibitem{Cue} J. -C. Cuenin, {\it Eigenvalue bounds for Dirac and fractional Schr\"odinger operators with complex potentials}, J. Funct. Anal. \textbf{272} (2017), 2987--3018.

\bibitem{DAn} P. D'Ancona, \emph{Kato smoothing and Strichartz estimates for wave equations with magnetic potentials}, Commun. Math. Phys. \textbf{335}  (2015), 1--16.

\bibitem{DaHi} E. B. Davies, A. M. Hinz, {\it Explicit constants for Rellich inequalities in $L^p(\Omega)$}, Math. Z. \textbf{227} (1998) 511--523.

\bibitem{Duyckaerts} T. Duyckaerts, {\it Private communication}.

\bibitem{EGG} M. B. Erdo\u{g}an, M. Goldberg, W. R. Green, {\it Limiting absorption principle and Strichartz estimates for Dirac operators in two and higher dimensions}, Commun. Math. Phys. \textbf{367} (2019), 241--263.

\bibitem{FSWY} H. Feng, A. Soffer, Z. Wu, X. Yao, {\it Decay estimates for higher-order elliptic operators}, Trans. Amer. Math. Soc. \textbf{373} (2020), 2805-2859.

\bibitem{FSY}H. Feng, A. Soffer, X. Yao, {\it Decay estimates and Strichartz estimates of fourth-order Schr\"odinger operator}. J. Funct. Anal. \textbf{274} (2018), 605--658

\bibitem{Fra1} R. L. Frank, {\it Eigenvalue bounds for Schr\"odinger operators with complex potentials}, Bull. Lond. Math. Soc. \textbf{43} (2011), 745--750

\bibitem{Fra2} R. L. Frank, {\it Eigenvalue bounds for Schr\"odinger operators with complex potentials. III}, Trans. Amer. Math. Soc. \textbf{370} (2018), 219--240.

\bibitem{FLS} R. L. Frank, E. L. Lieb, R. Seiringer, {\it Hardy-Lieb-Thirring inequalities for fractional Schr\"odinger operators}, J. Amer. Math. Soc. \textbf{21} (2008), 925--950.


\bibitem{GiVe} J. Ginibre, G. Velo, {\it The global Cauchy problem for the non linear Schr\"{o}dinger equation}, Ann.\ lHP-Analyse\ non\ lin\'eaire.\ \textbf{2} (1985), 309--327.

\bibitem{GVV} M. Goldberg, L. Vega, N. Visciglia, {\it Counterexamples of Strichartz inequalities for Schr\"odinger equations with repulsive potentials}, Int. Math. Res. Not. (2006), Art. ID 13927

\bibitem{Gra} L. Grafakos, {\it Classical Fourier analysis. Second edition}, Graduate Texts in Mathematics, 249. Springer, New York, 2008.

\bibitem{GT}W. R. Green, E. Toprak, {\it On the fourth order Schr\"odinger equation in four dimensions: dispersive estimates and zero energy resonances}. J. Differential Equations \textbf{267} (2019), 1899--1954.

\bibitem{Guo} Z. Guo, {\it Sharp spherically averaged Strichartz estimates for the Schr\"odinger equation}, Nonlinearity \textbf{29} (2016), 1668--1686.

\bibitem{GLNY} Z. Guo, J. Li, K. Nakanishi, L. Yan, {\it On the boundary Strichartz estimates for wave and Schr\"odinger equations}, J. Differential Equations \textbf{265} (2018), 5656--5675.

\bibitem{GuNa} Z. Guo, K. Nakanishi, {\it The Zakharov system in 4D radial energy space below the ground state}, arXiv:1810.05794

\bibitem{Gut} S. Guti\'errez, {\it Non trivial $L^q$ solutions to the Ginzburg-Landau equation}, Math. Ann. \textbf{328} (2004), 1--25.

\bibitem{Her} I. W. Herbst, {\it Spectral theory of the operator $(p^2+m^2)^{1/2}-Ze^2/r$}. Commun. Math. Phys. \textbf{53} (1977), 285--294.

\bibitem{Hos} T. Hoshiro, {\it Mourre's method and smoothing properties of dispersive equations}, Commun. Math. Phys. \textbf{202} (1999), 255--265.

\bibitem{HYZ} S. Huang, X. Yao, Q. Zheng, {\it $L^p$-limiting absorption principle of Schr\"odinger operators and applications to spectral multiplier theorems}, Forum Math. \textbf{30} (2018), 43--55.

\bibitem{HuZh} S. Huang, Q. Zheng, {\it Endpoint uniform Sobolev inequalities for elliptic operators with applications}, J. Differential Equations \textbf{267} (2019), 4609--4625.

\bibitem{HHZ} T. Huang, S. Huang, Q. Zheng, {\it Inhomogeneous oscillatory integrals and global smoothing effects for dispersive equations},  J. Differential Equations  \textbf{263} (2017), 8606--8629.

\bibitem{Levy}J. -M. L\'evy-Leblond, {\it Electron capture by polar molecules}, Phys. Rev. \textbf{153} (1967) 1--4.

\bibitem{Kat} T. Kato, \emph{Wave operators and similarity for some non-self-adjoint operators}, Math. Ann. \textbf{162} (1965/1966), 258--279.

\bibitem{KaYa} T. Kato, K. Yajima, \emph{Some examples of smooth operators and the associated smoothing effect}, Rev. Math. Phys. \textbf{1} (1989), 481--496.

\bibitem{KeTa} M. Keel, T. Tao, {\it Endpoint Strichartz estimates}, Amer. J. Math. \textbf{120} (1998), 955-980.

\bibitem{KPV} C. E. Kenig, G. Ponce, L. Vega, {\it Oscillatory integrals and regularity of dispersive equations}, Indiana Univ. Math. J. \textbf{40} (1991), 33--69.

\bibitem{KRS} C. E. Kenig, A. Ruiz, C. D. Sogge, {\it Uniform Sobolev inequalities and unique continuation for second order constant coefficient differential operators}, Duke Math. J. \textbf{55} (1987), 329-347.

\bibitem{KMVZZ} R. Killip, C. Miao, M. Visan, J. Zhang, J. Zheng, {\it The energy-critical NLS with inverse-square potential}, Discrete Contin. Dyn. Syst. \text{37} (2017), 3831--3866.

\bibitem{Laskin} N. Laskin, {\it Fractional quantum mechanics and L\'evy path integrals}, Phys. Lett. A, \textbf{268} (2000), 298--305.

\bibitem{MMT} J. Marzuola, J. Metcalfe, and D. Tataru, {\it Strichartz estimates and local smoothing estimates for asymptotically flat Schr\"odinger equations}, J. Funct. Anal. \textbf{255} (2008), 1497--1553.

\bibitem{Mizutani_JDE} H. Mizutani, {\it Remarks on endpoint Strichartz estimates for Schr\"odinger equations with the critical inverse-square potential}, J. Differential Equations \textbf{263} (2017) 3832--3853.

\bibitem{Mizutani_JST} H. Mizutani, {\it Eigenvalue bounds for non-self-adjoint Schr\"odinger operators with the inverse-square potential}, J. Spectral Theory \textbf{9} (2019), 677--709.

\bibitem{Mizutani_APDE}H. Mizutani, {\it Uniform Sobolev estimates for Schr\"odinger operators with scaling-critical potentials and applications}, Anal. PDE \textbf{13} (2020) 1333--1369.

\bibitem{Mizutani_JFA} H. Mizutani, {\it Strichartz estimates for Schr\"odinger equations with slowly decaying potentials}, J. Funct. Anal. \textbf{279} (2020) 108789.

\bibitem{MiYa2} H. Mizutani, X. Yao, {\it Global Kato smoothing and Strichartz estimates for higher-order Schr\"odinger operators with rough decay potentials}, preprint. arXiv:2004.10115


\bibitem{Mou} \'E. Mourre, {\it Absence of Singular Continuous Spectrum for Certain Self-Adjoint Operators}, Commun. Math. Phys. \textbf{78} (1981), 391--408.

\bibitem{Mourre_CMP} \'E. Mourre, {\it Operateurs conjugu\'es et propri\'et\'es de propagation. [Conjugate operators and propagation properties]}, Comm. Math. Phys. \textbf{91} (1983), no. 2, 279--300.

\bibitem{Pau} B. Pausader, {\it Global Well-Posedness for Energy Critical Fourth-Order Schr\"odinger Equations in the Radial Case}, Dynamics of PDE \textbf{4} (2007), 197--225.

\bibitem{ReSi} M. Reed, B. Simon, {\it Methods of Modern Mathematical Physics}, II,IV Academic Press,  New York-London, 1975, 1978.

\bibitem{RoSc} I. Rodnianski, W. Schlag, {\it Time decay for solutions of Schr\"odinger equations with rough and time-dependent potentials}, Invent. Math. \textbf{155}  (2004), 451--513.

\bibitem{RuSu} M. Ruzhansky, M. Sugimoto, {\it Smoothing properties of evolution equations via canonical transforms and comparison principle}, Proc. London Math. Soc. \textbf{105} (2012),  393--423.


\bibitem{SYY} A. Sikora, L. Yan, X. Yao, {\it Spectral multipliers, Bochner--Riesz means and uniform Sobolev inequalities for elliptic operators}, Int. Math. Res. Not. IMRN 2018, 3070--3121.

\bibitem{SmSo} H. F. Smith, C. D. Sogge, {\it Global Strichartz estimates for nontrapping perturbations of the Laplacian}, Comm. Partial Differential Equations \textbf{25} (2000), 2171--2183.

\bibitem{Ste} E. M. Stein, {\it Interpolation of linear operators}, Trans. Amer. Math. Soc. \textbf{83} (1956), 482--492

\bibitem{StWe} E. M. Stein, G. Weiss, {\it Introduction to Fourier Analysis on Euclidean Spaces}, Princeton Mathematical
Series, No. 32, Princeton University Press, Princeton, NJ, 1971.

\bibitem{Str}R. Strichartz, {\it Restrictions of Fourier transforms to quadratic surfaces and decay of solutions of wave equations}, Duke\ Math.\ J.\ \textbf{44} (1977), 705--714

\bibitem{Sug} M. Sugimoto, {\it Global smoothing properties of generalized Schr\"odinger equations}, J. Anal. Math. \textbf{76} (1998), 191--204.

\bibitem{Suzuki} T. Suzuki, {\it Solvability of nonlinear Schr\"odinger equations with some critical singular potential via generalized Hardy-Rellich inequalities}, Funkcialaj Ekvacioj, \textbf{59} (2016), 1--34.

\bibitem{Tao} T. Tao, {\it Nonlinear dispersive equations. Local and global analysis}, CBMS Regional Conference Series in Mathematics, 106. Published for the Conference Board of the Mathematical Sciences, Washington, DC; by the American Mathematical Society, Providence, RI, 2006.

\bibitem{Wat} K. Watanabe, {\it Smooth perturbations of the self-adjoint Operator $|\Delta|^{\alpha/2}$}, Tokyo J. Math. \textbf{14} (1991), 239--250.

\bibitem{Yafaev_JFA} D. Yafaev, {\it Sharp constants in the Hardy--Rellich inequalities}, J. Funct. Anal. \textbf{168} (1999), 121-144.


\bibitem{Yaj}K. Yajima, {\it Existence of solutions for Schr\"{o}dinger evolution equations}, Comm\ Math.\ Phys.\ \textbf{110} (1987), 415--426.

\bibitem{ZZ} J. Zhang and J. Zheng, {\it Scattering theory for nonlinear Schr\"odinger with inverse-square potential}, J. Funct. Anal. \textbf{267} (2014), 2907--2932.
\end{thebibliography}
\end{document}